\documentclass[11pt,a4paper,reqno]{amsart}
\usepackage[english]{babel}
\usepackage[T1]{fontenc}
\usepackage{verbatim}
\usepackage{palatino}
\usepackage{amsmath}
\usepackage{mathabx}
\usepackage{amssymb}
\usepackage{amsthm}
\usepackage{amsfonts}
\usepackage{graphicx}
\usepackage{esint}
\usepackage{color}
\usepackage{mathtools}
\usepackage{overpic}
\DeclarePairedDelimiter\ceil{\lceil}{\rceil}

\usepackage[colorlinks = true, citecolor = black]{hyperref}
\author{Damian D{\k a}browski}
\address{Department of Mathematics and Statistics\\ University of Jyv\"askyl\"a,
	P.O. Box 35 (MaD)\\
	FI-40014 University of Jyv\"askyl\"a\\
	Finland\vspace{1em}} 

\email{\href{mailto:damian.m.dabrowski@jyu.fi}{damian.m.dabrowski@jyu.fi}}

\author{Tuomas Orponen}
\email{\href{mailto:tuomas.t.orponen@jyu.fi}{tuomas.t.orponen@jyu.fi}}

\title{On the logarithmic equilibrium measure on curves}
\date{\today}
\thanks{D.D is supported by the Research Council of Finland postdoctoral grant \emph{Quantitative rectifiability and harmonic measure beyond the Ahlfors-David-regular setting}, grant no. 347123. T.O. is supported by the European Research Council (ERC) under the European Union’s Horizon Europe research and innovation programme (grant agreement No 101087499), and by the Research Council of Finland via the project \emph{Approximate incidence geometry}, grant no. 355453.}

\newcommand{\R}{\mathbb{R}}

\newcommand{\N}{\mathbb{N}}

\newcommand{\C}{\mathbb{C}}

\newcommand{\spt}{\operatorname{spt}}
\newcommand{\Hd}{\dim_{\mathrm{H}}}

\newcommand{\diam}{\operatorname{diam}}

\newcommand{\dist}{\operatorname{dist}}

\newcommand{\sgn}{\operatorname{sgn}}
\newcommand{\Rea}{\operatorname{Re}}

\newcommand{\bmu}{\boldsymbol{\mu}}

\def\Barint_#1{\mathchoice
          {\mathop{\vrule width 6pt height 3 pt depth -2.5pt
                  \kern -8pt \intop}\nolimits_{#1}}%
          {\mathop{\vrule width 5pt height 3 pt depth -2.6pt
                  \kern -6pt \intop}\nolimits_{#1}}%
          {\mathop{\vrule width 5pt height 3 pt depth -2.6pt
                  \kern -6pt \intop}\nolimits_{#1}}%
          {\mathop{\vrule width 5pt height 3 pt depth -2.6pt
                  \kern -6pt \intop}\nolimits_{#1}}}

\numberwithin{equation}{section}

\theoremstyle{plain}
\newtheorem{thm}{Theorem}
\numberwithin{thm}{section}
\newtheorem*{"thm"}{"Theorem"}

\newtheorem{lemma}[thm]{Lemma}

\newtheorem{ex}[thm]{Example}
\newtheorem{cor}[thm]{Corollary}
\newtheorem{proposition}[thm]{Proposition}

\theoremstyle{definition}

\newtheorem{definition}[thm]{Definition}

\newtheorem{notation}[thm]{Notation}

\theoremstyle{remark}
\newtheorem{remark}[thm]{Remark}

\newcommand{\nref}[1]{(\hyperref[#1]{#1})}

\DeclareMathSymbol{\intop}  {\mathop}{mathx}{"B3}
\DeclareMathOperator{\lip}{Lip}
\DeclareMathOperator{\cpc}{Cap}
\DeclareMathOperator{\graph}{graph}

\begin{document}

\begin{abstract} Let $\mu$ be the logarithmic equilibrium measure on a compact set $\gamma \subset \R^{d}$. We prove that $\mu$ is absolutely continuous with respect to the length measure on the part of $\gamma$ which can be locally expressed as the graph of a $C^{1,\alpha}$-function $\R \to \R^{d - 1}$, $\alpha > 0$. 

For $d = 2$, at least in the case where $\gamma$ is a compact $C^{1,\alpha}$-graph, our result can also be deduced from the classical fact that $\mu$ coincides with the harmonic measure of $\Omega = \R^{2} \, \setminus \, \gamma$ with pole at $\infty$. For $d \geq 3$, however, our result is new even for $C^{\infty}$-graphs. In fact, up to now it was not even known if the support of $\mu$ has positive dimension. \end{abstract}

\maketitle

\tableofcontents

\section{Introduction}

\subsection{Equilibrium measures and previous results on their structure}\label{s:background} One of the key concepts in potential theory is the \emph{equilibrium measure}. If $X$ is a topological space, $k \colon X \times X \to \R$ is a sufficiently nice \emph{kernel}, and $K \subset X$ is compact, the \emph{$k$-equilibrium measure on $K$} is the unique probability measure $\mu$ supported on $K$ which minimises the \emph{$k$-energy} 
\begin{displaymath} \mathcal{E}_{k}(\mu) = \iint k(x,y) \, d\mu(x) \, d\mu(y). \end{displaymath}
(This notion is also sometimes known as the \emph{$k$-capacitary measure}.) The definition only makes sense if the existence and uniqueness of $\mu$ are guaranteed, but this is true under fairly general hypotheses on $X$ and $k$, see Fuglede's paper \cite{MR117453}. 

In this paper, we are only interested in the special case where $X = \R^{d}$, and $k$ is the \emph{$s$-Riesz kernel} $k(x,y) = |x - y|^{-s}$ with $s > 0$, or the \emph{logarithmic kernel $k(x,y) = -\log |x - y|$}. In these cases, the equilibrium measure exists and is unique on a compact set $K \subset \R^{d}$, provided that $K$ supports at least one probability measure with finite $k$-energy. An existence proof covering the Riesz and logarithmic kernels simultaneously can be found in \cite[Theorem 5.4]{hayman1976subharmonic}. The uniqueness in the case of positive-order Riesz potentials is explained and stated in \cite[Corollary, p. 146]{MR350027}, and in the case of the logarithmic potential in \cite[Section 2]{MR125248}. The existence and uniqueness are also explained in the unpublished notes \cite[Theorem 6.8 + Theorem 6.11]{PratsTolsa}. These are not the original references: \cite{MR125248} quotes the thesis of Frostman \cite{Solr-kuopio.291385} and the paper \cite{Riesz1988} of Riesz for the main observations needed in the proof of the existence of the logarithmic equilibrium measure.

From now on, the equilibrium measure for the $s$-Riesz kernel will be called the \emph{$s$-equilibrium measure}, and the \emph{$\log$-equilibrium measure} is the one associated with the logarithmic kernel. How do the $(s/\log)$-equilibrium measures look like on "natural" compact sets $K \subset \R^{d}$? Here are some known results:
\begin{itemize}
	\item[(a)] In concrete cases such as line segments, balls, spheres, ellipsoids, explicit formulae for the (densities of the) measures, or at least their potentials, are available, see \cite[p. 166, 172]{MR350027} and \cite[Section 4.6]{MR3970999}. For example, the density of the logarithmic equilibrium measure $\mu_{\log}$ on $[-1,1]$ is given (see \cite[Proposition 4.6.1]{MR3970999})  by
	\begin{displaymath} d\mu_{\log}(x) = \frac{dx}{\pi\sqrt{1 - x^{2}}}. \end{displaymath}
	It is noteworthy that $d\mu_{\log}$ is an unbounded function, and indeed so singular the highest (global) exponent $s > 0$ such that $\mu_{\log}(B(x,r)) \lesssim r^{s}$ is $s = \tfrac{1}{2}$.
	\item[(b)] If $\Omega \subset \R^{d}$ is a Wiener regular unbounded open set with $\partial \Omega$ compact, the $(d - 2)$-equilibrium measure on $\partial \Omega$ (or the $\log$-equilibrium measure for $d = 2$) equals a constant times $\omega^{\infty}_{\Omega}$, the harmonic measure of $\Omega$ with pole at $\infty$, see \cite[Theorem 7.32]{PratsTolsa}. The connection between harmonic measure in $\R^{2}$ and the $\log$-equilibrium measure is also explained in \cite[Chapter III]{MR2150803}, although the result above is not stated in this generality.
\end{itemize}

Harmonic measure in $\R^{d}$ is quite well-understood. For example, it is known \cite{MR1078740} that 
\begin{displaymath} \mathcal{H}^{d - 1}|_{\partial \Omega} \ll \omega^{\infty}_{\Omega} \ll \mathcal{H}^{d - 1}|_{\partial \Omega} \end{displaymath}
whenever $\partial \Omega$ is Ahlfors $(d - 1)$-regular, and $\Omega$ an NTA domain. In particular, if $\Omega$ is a Lipschitz domain, then the harmonic measure is mutually absolutely continuous with respect to the $(d - 1)$-dimensional Hausdorff measure on $\partial \Omega$, and indeed an $A_{\infty}$-weight. (This special case of Lipschitz domains is due to Dahlberg \cite{MR466593}.) For more general conditions for the absolute continuity of harmonic measure, see \cite{MR1078268} (in the plane) or \cite[Theorem I]{MR3903916} and \cite[Theorem III]{MR4251289} in general dimensions.

 These results translate immediately to $(d - 2)$-equilibrium measures thanks to (b). A little imprecisely, one could summarise by saying that the $(d - 2)$-equilibrium measure on "Lipschitz $(d - 1)$-manifolds" is absolutely continuous with respect to $\mathcal{H}^{d - 1}$.

\begin{itemize}
	\item[(c)] Another result which applies to a wide variety of sets is the following due to Calef and Hardin \cite[Theorem 1.3]{MR2491459}: if $A \subset \R^{d}$ is a compact \emph{strongly $m$-rectifiable set} (\cite[Definition 1.1]{MR2491459}), then the $s$-equilibrium measures $\mu_{s}$ weak* converge to the normalised $m$-dimensional Hausdorff measure on $A$ as $s \nearrow m$. The same result for (subsets of) $C^{1}$-manifolds is due to Frostman and Wallin \cite[Theorem 2]{10.1007/BFb0086328}. Calef \cite{MR2670134} also obtained a corresponding result for $m$-dimensional self-similar sets in $\R^{d}$ (where $m$, unlike above, need not be an integer). These results apply to many sets, but say nothing about the measure $\mu_{s}$ for $s$ fixed.
	\item[(d)] For $s \in (0,d) \, \setminus \, \{d - 2\}$ fixed, results are scarce. A paper of Wallin \cite{MR204693} from the 60s concerns sets $K \subset \R^{d}$ with non-empty interior: for $s \in (d - 2,d)$, he proves that the $s$-equilibrium measure on $K$ is mutually absolutely continuous with respect to Lebesgue measure -- indeed with analytic density -- on the interior of $K$. 
	\item[(e)] A version of Wallin's result for $(d - 1)$-dimensional $C^{\infty}$-manifolds $\Gamma \subset \R^{d}$ is due to Hardin, Reznikov, Saff, and Volberg \cite[Theorem 2.7]{MR4009530} from 2017: for $s \in (d - 2,d - 1)$, the $s$-equilibrium measure is absolutely continuous with respect to $\mathcal{H}^{d - 1}|_{\Gamma}$, with locally bounded density. The result also holds in the "harmonic" case $s = d - 2$, but this was known earlier (under milder regularity).
\end{itemize}

\begin{remark} The proof of \cite[Theorem 2.7]{MR4009530} contained a point we did not fully understand. It is based on the invertibility of elliptic pseudo-differential operators, but it seemed to us that operator "$P$", whose symbol "$p$" is defined above \cite[(7.10)]{MR4009530}, is not elliptic in the claimed domain $\tilde{B}$. Indeed, the cut-off function $u$ (in the definition of the symbol) is compactly supported on $\mathrm{int\,} \tilde{B}$, so the $p(\tilde{x},\xi) = 0$ for all $(\tilde{x},\xi)$ with $\tilde{x} \in \tilde{B} \, \setminus \, \spt u$.  \end{remark} 

All the results in points (b)-(e) concern $s$-equilibrium measures for $s \in [d - 2,d]$. For $s \in [0,d - 2)$, there is a lack of "general" results on the structure of the $s$-equilibrium measure -- by this, we mean results beyond a few very specific sets (see (a)) where explicit formulae are available. One exception is a result \cite{MR2297961} of Hardin, Saff, and Stahl, which studies the support of the $\log$-equilibrium measure on \emph{surfaces of revolution} in $\R^{3}$. They prove that $\spt \mu_{\log}$ is contained in the "outermost" part of the surface. For example, the torus is a surface of revolution, based on a circle $\R^{2}$. The result says that $\spt \mu_{\mathrm{log}}$ is contained on the part of the torus obtained by revolving only the "outer" semi-circle. 

\subsubsection{The discrete energy minimisation problem} We briefly mention the discrete variant of our problem, and its connection to the continuous variant. Let $k$ be either the logarithmic or the Riesz kernel on $\R^{d}$ (although what follows would work under much broader generality, see \cite[Chapter 4]{MR3970999}). For a compact infinite set $K \subset \R^{d}$, and $N \in \N$ fixed, a \emph{$k$-minimising $N$-set on $K$} is any finite set $P = \{p_{1},\ldots,p_{N}\} \subset K$ which minimises the discrete $k$-energy
\begin{displaymath} \mathcal{E}_{k}(P) := \sum_{i \neq j} k(p_{i},p_{j}). \end{displaymath}
Define $\mathcal{E}_{k}(K,N)$ as the discrete $k$-energy of a $k$-minimising $N$-set on $K$. Define also the normalised quantity $\tau_{k}(K,N) := \mathcal{E}_{k}(K,N)/[N(N - 1)]$. It is known \cite[Proposition 2.1.1]{MR3970999} (the result was originally discovered by Fekete \cite{MR1544613}) that $N \mapsto \tau_{k}(K,N)$ is increasing, so there exists a limit $\tau_{k}(K) \in \R \cup \{+\infty\}$, known as the \emph{$k$-transfinite diameter of $K$}. The relation between $k$-minimising $N$-sets, $k$-transfinite diameter, and the equilibrium measure, is now summarised in \cite[Theorem 4.2.2]{MR3970999}:
\begin{itemize}
	\item $\tau_{k}(K)$ equals the $k$-energy of the $k$-equilibrium measure on $K$. 
	\item Let $P_{2},P_{3},\ldots \subset K$ be a sequence of finite sets such that $P_{N}$ is a $k$-minimising $N$-set. Then the discrete measures
	\begin{displaymath} \mu_{N} := \frac{1}{N} \sum_{p \in P_{N}} \delta_{p} \end{displaymath}
	converge weak* to the $k$-equilibrium measure on $K$ (provided that the $k$-equilibrium measure exists, i.e. $K$ supports measures with finite $k$-energy).
\end{itemize}
Therefore, any insight into the distribution of the equilibrium measure yields information about the asymptotic distribution, as $N \to \infty$, of discrete measures supported on $k$-minimising $N$-sets -- which are often called sets of \emph{Fekete points}.

A significant new feature in the discrete -- versus continuous -- energy minimisation problem is that it makes sense for $s$-Riesz kernels with $s > \Hd K$: while there are no measures $\mu$ supported on $K$ with finite $s$-energy, it is still reasonable to ask questions about the distribution of $s$-minimising $N$-sets. Hardin and Saff \cite[Theorem 2.4]{MR2132763} proved that if $K$ is an $m$-rectifiable manifold (see \cite[(19)]{MR2132763} for the definition), and $s > m$, then any sequence of $s$-minimising $N$-sets distributes asymptotically uniformly, as $N \to \infty$, with respect to the normalised $m$-dimensional Hausdorff measure on $K$.

Problems related to discrete energy minimisation form a large research area of their own, with connections to mathematical physics and discrepancy theory. Since this topic is a little outside the scope of our paper, we refer the reader to \cite{MR4496723,2023arXiv230213067B,MR3841839,MR3970999,MR3825947,MR4431907} for further information. We finally mention that the discrete energy minimisation problem (with various exponents) in the special case of the unit sphere $S^{d - 1}$ has attracted a lot of attention, see e.g. \cite{MR2398782,MR1972249,MR1415178}, perhaps due to its connection to the $7^{th}$ problem on Smale's list \cite{MR1631413} (provide an efficient algorithm for finding $\log$-minimising $N$-sets on $S^{2}$).

\subsection{New results} The main purpose of this paper is to provide the first (fairly) "general" result of the structure of the $\log$-equilibrium measure on curves in $\R^{d}$. The main result is Theorem \ref{main}, which roughly says that the $\log$-equilibrium measure on $C^{1,\alpha}$-curves is absolutely continuous with respect to the length measure. We first set up some terminology needed to state our theorem precisely.

\begin{definition}[$\gamma_{\mathrm{reg}}$]
	Given a compact set $\gamma\subset\R^d$ and $\alpha\in (0,1]$, we say that $x\in \gamma$ is a \emph{$C^{1,\alpha}$-regular point} if $\gamma$ is locally a $1$-dimensional $C^{1,\alpha}$-graph in a neighbourhood of $x$. That is, there exists $A\in C^{1,\alpha}(\R;\R^{d-1})$ such that
	\begin{equation*}
		\gamma\cap B(x,r)=\rho(\graph(A))\cap B(x,r)
	\end{equation*}
	for some radius $r>0$ and rotation $\rho\in \mathrm{SO}(d)$. We denote $\gamma_{\mathrm{reg}}$ the set of all points $x\in\gamma$ which are $C^{1,\alpha}$-regular points for some $\alpha\in (0,1]$. \end{definition}


\begin{thm}\label{main} Let $\gamma\subset \R^{d}$ be a compact set, and let $\mu$ be the logarithmic equilibrium measure on $\gamma$. Then $\mu|_{\gamma_{\mathrm{reg}}}$ is absolutely continuous with respect to the length measure. \end{thm}

In particular, the $\log$-equilibrium measure on a $C^{1,\alpha}$-regular curve, or a finite union of such curves, is absolutely continuous with respect to the length measure.


\begin{remark}\label{rem6} It remains an interesting open problem to determine conditions under which the $\log$-equilibrium measure is \emph{mutually} absolutely continuous with respect to the length measure. The hypotheses of Theorem \ref{main} are not sufficient, even when $\gamma=\gamma_{\mathrm{reg}}$ and $d = 2$. To see why, let $\gamma := S(1) \cup S$, where $S(1)$ is the unit circle, and $S \subset B(1)$ is an arbitrary compact set. If $S$ is a smooth curve, then $\gamma = \gamma_{\mathrm{reg}}$. It nonetheless turns out that the $\log$-equilibrium measure $\mu_{\log,\gamma}$ equals $\sigma$, the normalised length measure on $S(1)$.
	
	To see this, first note that 
	\begin{equation}\label{form158} \mu_{\log,\bar{B}(0,1)} = \sigma \end{equation}
	by \cite[Theorem 4.6.7]{MR3970999}. Thus, $\sigma$ uniquely minimises the logarithmic energy among probability measures supported on $\bar{B}(0,1)$. In particular, $\sigma$ uniquely minimises the logarithmic energy among probability measures supported on $\gamma \subset \bar{B}(0,1)$. Therefore $\sigma = \mu_{\mathrm{log},\gamma}$. \end{remark} 

We will reduce the proof of Theorem \ref{main} to a somewhat more general statement concerning measures on $[0,1]$, Theorem \ref{mainTechnical}. We first make a definition concerning the \emph{a priori} regularity of the measures considered in Theorem \ref{mainTechnical}.

\begin{definition}[Measures with continuous logarithmic potential]\label{def5Intro} For $\gamma \subset \R^{d}$, we let $\mathcal{M}(\gamma)$ be the family of compactly supported Radon measures on $\gamma$ such that
	\begin{itemize}
		\item[(a)] $y \mapsto -\log |x - y| \in L^{1}(\mu)$ for $x \in \gamma$, and
		\item[(b)] $x \mapsto \int -\log |x - y| \, d\mu(y) \in C(\spt \mu)$.
	\end{itemize}
\end{definition}

\begin{remark}\label{rem7} If properties (a)-(b) are satisfied, then $x \mapsto \int -\log |x - y| \, d\mu(y) \in C(\R^{d})$. This is known as the \emph{continuity principle} for the logarithmic potential, see Theorem \ref{thm:continuity principle}. \end{remark}

We postpone the proof of the following easy lemma to Section \ref{s:preliminariesII}.

\begin{lemma}\label{lemma13} Let $\gamma \subset \R^{d}$, and let $f \colon \gamma \to \R^{D}$ be a bi-Lipschitz embedding. If $\mu \in \mathcal{M}(\gamma)$, then $f\mu \in \mathcal{M}(f(\gamma))$. \end{lemma} 

We may use Lemma \ref{lemma13} to make sense of the following \emph{logarithmic graph potential}:

\begin{definition}[$U^{\Gamma}\mu$]\label{defGraphPotential} Let $d \geq 2$, and let $A \colon \R \to \R^{d - 1}$ be Lipschitz. Then the \emph{graph map} $\Gamma(x) := (x,A(x))$ is a bi-Lipschitz embedding $\R \to \R^{d}$. Let $\mu \in \mathcal{M}(\R)$. Then the following map $U^{\Gamma}\mu \colon \R \to \R$ is well-defined and continuous:
	\begin{displaymath} U^{\Gamma}\mu(x) := \int \log \frac{1}{|\Gamma(x) - \Gamma(y)|} \, d\mu(y), \qquad x \in \R. \end{displaymath} 
	Indeed, the continuity of $U^{\Gamma}\mu$ at $x \in \R$ is equivalent to the continuity of $U(\Gamma\mu)$ at $\Gamma(x)$ along $\Gamma(\R)$, and this is true by Lemma \ref{lemma13}.
\end{definition}

We may now state the main technical result of the paper:

\begin{thm}\label{mainTechnical} Let $\alpha \in (0,1]$, $p\in (1,\infty)$ and $A \in C^{1,\alpha}(\R;\R^{d-1})$. If $\lip(A)\le\delta$, where $\delta=\delta(p,\alpha,d)\in (0,1)$, then the following holds.
	
	Let $\mathcal{L} \colon \R \to \R$ be Lipschitz. Assume that $\mu \in \mathcal{M}(\R)$ has the following properties on a compact interval $I_{0} \subset \R$:
	\begin{itemize}
		\item $U^{\Gamma}\mu(x) = \mathcal{L}(x)$ for all $x \in \spt \mu \cap I_{0}$,
		\item $U^{\Gamma}\mu(x) \geq \mathcal{L}(x)$ for all $x \in I_{0}$.
	\end{itemize}
	Here $\Gamma(x) = (x,A(x))$. Then $\mu$ is absolutely continuous with respect to Lebesgue measure on $I_{0}$. In fact, if $I \subset \mathrm{int \,} I_{0}$ is a compact subinterval, then $\mu \in L^{p}(I)$. \end{thm}

The proof of Theorem \ref{main} based on Theorem \ref{mainTechnical} is straightforward, but requires a few standard pieces of potential theory, so we postpone the argument to Section \ref{s:mainProof}.

\subsection{Proof of Theorem \ref{mainTechnical}: an outline}\label{s:outline} This section will also serve as a guide to the paper. Why are there so few earlier results on the $s$-equilibrium measure for $s \in [0,d - 2)$? The short reason is that the $s$-Riesz potential 
\begin{displaymath} U_{s}\mu(x) = \int \frac{d\mu(y)}{|x - y|^{s}}, \qquad x \in \R^{d}, \end{displaymath}
is (in general) subharmonic in $\R^{d} \, \setminus \, \spt \mu$ only when $s \geq d - 2$. We now discuss the relevance of this fact. Let $\gamma \subset \R^{d}$ be a compact set with positive $s$-capacity, and let $\mu$ be the $s$-Riesz equilibrium measure on $\gamma$. Then there exists a constant $C = \mathrm{Cap}_{s}(\gamma)$ such that $U_{s}\mu \leq C$ everywhere on $\spt \mu$, and $U_{s}\mu \geq C$ \emph{approximately everywhere} on $\gamma$, see \cite[p. 138]{MR350027}. If $\gamma$ is Ahlfors regular, the lower bound $U_{s}\mu \geq C$ holds everywhere on $\gamma$ by \cite[Theorem 2.5]{reznikov2017minimum}. 

Now, when $s \in [d - 2,d]$, the upper bound $U_{s}\mu \leq C$ on $\spt \mu$ can be upgraded for free to $U_{s}\mu \leq C$ on $\R^{d}$, see \cite[(1.3.11)]{MR350027}. In particular, if $\gamma$ is Ahlfors regular, then $U_{s}\mu = C$ everywhere on $\gamma$ -- and not just everywhere on $\spt \mu$. 

The argument above is based on the subharmonicity of the $s$-Riesz potential, so it is not available for the logarithmic potential $U_{\log}$ when $d \geq 3$. Assume, nevertheless, that for some reason $U_{\log}\mu = C$ everywhere on $\gamma$. Let us discuss why this would be useful. Assuming that $\gamma$ is a Lipschitz graph with sufficiently small Lipschitz constant, the potential $U_{\log}\mu$ turns out to be bounded and invertible $L^{2}(\mathcal{H}^{1}|_{\gamma}) \to \dot{H}^{1}(\mathcal{H}^{1}|_{\gamma})$ (the homogeneous Sobolev space where first order derivatives "along $\gamma$" lie in $L^{2}$). Since $U_{\log}\mu \equiv C \in \dot{H}^{1}(\mathcal{H}^{1}|_{\gamma})$, we might now conclude that $\mu \in L^{2}(\mathcal{H}^{1}|_{\gamma})$, and in particular $\mu \ll \mathcal{H}^{1}|_{\gamma}$. We will never actually define the space $\dot{H}^{1}(\mathcal{H}^{1}|_{\gamma})$ in the paper, but the invertibility of $U_{\log} \colon L^{2}(\mathcal{H}^{1}|_{\gamma}) \to \dot{H}^{1}(\mathcal{H}^{1}|_{\gamma})$ corresponds to the case $\beta = 0$ of Theorem \ref{thm1}. 

The main obstacle is that we do not know \emph{a priori} that $U_{\log}\mu = C$ everywhere on $\gamma$. We only know this on $\spt \mu$, and it is far from obvious that $\spt \mu = \gamma$ (recall Remark \ref{rem6}). We now outline how we overcome the obstacle, and prove Theorem \ref{mainTechnical}. Abbreviate $U := U^{\Gamma}$. Then the hypotheses are that $U\mu(x) = \mathcal{L}(x)$ for $x \in \spt \mu$, and $U\mu(x) \geq \mathcal{L}(x)$ everywhere. The key idea in the proof is to start by decomposing the potential $U\mu$ as a sum $U\mu = P\mu + R\mu$, where (roughly speaking) $P\mu$ is convex outside $\spt \mu$, and $R\mu$ is $\alpha$-H\"older continuous (where $\alpha$ is the H\"older regularity exponent of $\Gamma$), see Proposition \ref{prop15} and Lemma \ref{lemma10} for more precise statements. From all of these pieces of information combined, we are able to deduce that $U\mu$ itself is $\alpha$-H\"older continuous, the main tool for this is Proposition \ref{prop16}. 

From this point on, the proof outline contains many white lies, so the statements should not be taken literally (but we give pointers to the precise versions). To use the $\alpha$-H\"older continuity of $U\mu = U^{\Gamma}\mu$, we prove that the operators
\begin{displaymath} T_{\beta} := \Delta^{\beta/2}U\Delta^{(1 - \beta)/2}, \qquad \beta \in [0,1], \end{displaymath}
are bounded and invertible on $L^{p}(\R)$, for arbitrary $p \in (1,\infty)$, provided that $\mathrm{Lip}(\Gamma)$ is sufficiently small, depending on $d,p$ (see Theorem \ref{thm1}). Here $\Delta^{\beta/2}$ is a fractional Laplacian, see Section \ref{s:fractionalLaplacians}. The fact that we cannot take $p = \infty$ is a technical nuisance, but let us assume for now that the invertibility holds for $p = \infty$. Then, taking $\beta = \alpha$, we may deduce
\begin{equation}\label{form159} \|\Delta^{(\alpha - 1)/2}\mu\|_{L^{\infty}} \lesssim \|\Delta^{\alpha/2}U\Delta^{(1 - \alpha)/2}\Delta^{(\alpha - 1)/2}\mu\|_{L^{\infty}} = \|\Delta^{\alpha/2}U\mu\|_{L^{\infty}}. \end{equation}
Making this formal estimate rigorous is the content of Proposition \ref{prop8-2}. Now, the $\alpha$-H\"older continuity of $U\mu$ implies (up to changing the exponent infitesimally) that $\Delta^{\alpha/2}U\mu \in L^{\infty}$ (see Proposition \ref{prop18}), and therefore $\Delta^{(\alpha - 1)/2}\mu \in L^{\infty}$. But $\Delta^{(\alpha - 1)/2}\mu$ is the $\alpha$-Riesz potential of $\mu$, so we have managed to deduce that $\mu(B(x,r)) \lesssim r^{\alpha}$. In the actual proof we only get that $\Delta^{(\alpha - 1)/2}\mu \in L^{p}$ for high values of $p$, but this also implies Frostman regularity for $\mu$ with a lower exponent, depending on $p$, see Proposition \ref{frostmanProp}.

This is a non-trivial estimate on the dimension of $\mu$, but unsatisfactory as a final conclusion: the "gain" in the dimension of $\mu$ so far only matchs the H\"older regularity of $\Gamma$. However, it turns out that the potential $R\mu$ is actually $\min\{1,(\alpha + \kappa)\}$-H\"older continuous, provided that $\mu(B(x,r)) \lesssim r^{\kappa}$, see Lemma \ref{lemma10}. This observation gives rise to a bootstrapping scheme, where any "known" Frostman exponent for $\mu$ implies a strictly higher Frostman exponent -- until the exponent reaches (almost) $1$.

Since we already learned above that $\mu(B(x,r)) \lesssim r^{\alpha}$, we may deduce that $R\mu$ is $(2\alpha)$-H\"older continuous. Combining this information with the inequality \eqref{lemma10} (replacing $\alpha$ by $2\alpha$) then shows that $\Delta^{(2\alpha - 1)/2}\mu \in L^{\infty}(\R)$, and therefore $\mu(B(x,r)) \lesssim r^{2\alpha}$.

This "loop" is iterated $\sim \alpha^{-1}$ times (see Section \ref{s:proofMainTechnical} for the details), until we know that $\mu(B(x,r)) \lesssim r^{\kappa}$ for some $\kappa > 1 - \alpha$. At this point we may deduce that $R\mu$ is Lipschitz continuous. This (almost) implies that $U\mu$ is also Lipschitz continuous, and in particular $U\mu \in \dot{H}^{1}(\mathcal{H}^{1}|_{\gamma})$. As discussed below \eqref{lemma10}, this implies $\mu \in L^{2}(\R)$, and the proof of Theorem \ref{mainTechnical} is complete.

\subsection{Further research} Developing the regularity theory for $s$-equilibrium measures seems like a worthwhile research program. We list a few concrete directions for further research.
\begin{enumerate}
\item Relax $C^{1,\alpha}$-regularity to Lipschitz regularity, and generalise Theorem \ref{main} from curves to $m$-dimensional surfaces (replacing the $\log$-equilibrium measure by the $(m - 1)$-equilibrium measure). The Lipschitz regularity seems to be beyond our method, given the bootstrapping scheme explained in Section \ref{s:outline}.
\item For compact smooth $m$-dimensional surfaces, is it plausible that the $s$-equilibrium measures are absolutely continuous for all $s \in (m - 2,m)$?
\item If $\gamma \subset \R^{d}$ is a compact Lipschitz graph, does the support of the $\log$-equilibrium measure coincide with $\gamma$? How about if one adds $C^{1,\alpha}$-regularity for $\alpha > 0$?
\item Our result shows that $\mu\ll \mathcal{H}^1$ for the equilibrium measure on a $C^{1,\alpha}$-graph, with some local $L^p$ estimates for the density of $\mu$. Can one show improved regularity of this density depending on $\alpha$? For example: local boundedness, continuity, Sobolev regularity...
\item For "fractal" compact sets $K \subset \R^{d}$, find estimates for the dimension of the $s$-equilibrium measure. For $s = d - 2$, there are several results (but also major open questions) on the dimension of harmonic measure, see for example \cite{MR4140087,MR874032,MR4664655,MR4735637}.
\end{enumerate}

\begin{remark} A theory of \emph{higher-codimensional harmonic measure} has been recently developed by David, Feneuil, and Mayboroda \cite{MR4341338,MR4597210,MR3634676,MR3926132,MR4456214}. It would be interesting to know if there is any connection between the $\log$-equilibrium measure on a Lipschitz graph $\gamma \subset \R^{d}$, and the DFM harmonic measure on $\R^{d} \, \setminus \, \gamma$. As explained in Section \ref{s:background}(b), the $\log$-equilibrium measure on a compact Lipschitz graph $\Gamma \subset \R^{2}$ coincides with the usual harmonic measure of $\Omega = \R^{2} \, \setminus \, \Gamma$ with pole at $\infty$.

How about higher dimensions? The $\log$-equilibrium measure $\mu_{\log}$ is isometry-invariant. Therefore, if the curve $\gamma \subset \R^{2}$ from Remark \ref{rem6} is isometrically embedded into $\R^{d}$ with $d \geq 3$, the $\log$-equilibrium measure remains unchanged; the support of $\mu_{\log}$ will remain a strict subset of $\gamma$. On the other hand, the DFM harmonic measure $\omega_{\mathrm{DFM}}$ is mutually absolutely continuous with respect to the length measure on $\gamma$ (provided that $\gamma$ is chosen Ahlfors $1$-regular), see \cite[Theorem 1.1]{MR4597210} or \cite[Theorem 1.14]{MR4456214} (\cite[Theorem 6.7]{MR4255042} contains a special case). This indicates that $\mu_{\log}$ and $\omega_{\mathrm{DFM}}$ are \emph{a priori} rather different objects. \end{remark} 

\section{Preliminaries}\label{s:preliminaries}
\subsection{Classical potential theory}
Given a set $\gamma\subset\R^d$ denote by $\mathcal{P}(\gamma)$ the set of Radon probability measures with compact support contained in $\gamma$. Given $\mu\in\mathcal{P}(\R^d)$ we define its \emph{logarithmic energy} as 
\begin{equation*}
	\mathcal{E}_{\log}(\mu)\coloneqq \int U\mu(x)\, d\mu(x) = \iint -\log|x-y|\, d\mu(x)d\mu(y).
\end{equation*}
We say that a compact set $\gamma$ has \emph{non-zero capacity}, denoted by $\cpc(\gamma)\neq 0$, if there exists some $\mu\in\mathcal{P}(\gamma)$ with $\mathcal{E}_{\log}(\mu)<\infty$.

\begin{thm}\label{thm:equilibrium-existence}
	Suppose that $\gamma\subset\R^d$ is compact and $\cpc(\gamma)\neq 0$. Then there is a unique $\mu\in\mathcal{P}(\gamma)$ which minimizes the logarithmic energy in $\mathcal{P}(\gamma)$. That is,
	\begin{equation*}
		\mathcal{E}_{\log}(\mu) = \min_{\nu\in\mathcal{P}(\gamma)}\mathcal{E}_{\log}(\nu).
	\end{equation*}
\end{thm}
In the planar case $d=2$ the standard reference for this and other results of this section is \cite[Chapter II]{MR350027}. In case $d>2$ the proof of existence can be found in \cite[Theorem 5.4]{hayman1976subharmonic}. For uniqueness, see \cite[p. 5]{MR125248} or \cite[Theorem 2.5]{cegrell1998two}.

The unique minimizer from Theorem \ref{thm:equilibrium-existence} is the \emph{logarithmic equilibrium measure on $\gamma$}, but we will often abbreviate this to ``equilibrium measure''. The following classical estimates for the potentials of equilibrium measures can be found in \cite[Theorem 5.8]{hayman1976subharmonic}.
\begin{thm}\label{thm:eq-estimates}
	Suppose that $\gamma\subset\R^d$ is compact with $\cpc(\gamma)\neq 0$, and let $\mu$ be the equilibrium measure on $\gamma$. Then
	\begin{align*}
		U\mu(x) &\le \mathcal{E}_{\log}(\mu)\quad\quad\quad\qquad\text{for all $x\in \spt\mu$},\\
		U\mu(x)&\ge \mathcal{E}_{\log}(\mu)\quad\quad\quad\qquad\text{for approximately all $x\in \gamma$.}
	\end{align*}
\end{thm}

The terminology \emph{approximately all} above means ``for all $x\in \gamma$ except for a set of inner capacity zero'', see \cite[Definition, p. 135]{MR350027}. This terminology is quite irrelevant to us thanks to the following minimum principle, Theorem \ref{thm:minimum principle}, which is a special case of \cite[Theorem 2.5]{reznikov2017minimum}. A closed set $\gamma\subset\R^d$ is called {Ahlfors $s$-regular} if there exists a constant $C\in (1,\infty)$ such that
\begin{equation*}
	C^{-1}r^s \le \mathcal{H}^s(\gamma\cap B(x,r))\le Cr^s\quad\text{for all $x\in \gamma$ and $0<r<\diam(\gamma)$.}
\end{equation*}
\begin{thm}\label{thm:minimum principle}
	Let $s \in (0,d]$, and let $\gamma\subset \R^d$ be a compact Ahlfors $s$-regular set. Let $\mu$ be a Radon measure on $\gamma$ satisfying 
	\begin{equation*}
		U\mu(x)\ge M\quad\text{for approximately all $x\in \gamma$}
	\end{equation*}
	for some $M>0$. Then $U\mu(x)\ge M$ for all $x\in \gamma$. \end{thm}

We need an extension of Theorem \ref{thm:minimum principle}, where the "global" Ahlfors regularity is relaxed:
\begin{cor}\label{thm:minimum principleExtension}
	Let $\gamma\subset \R^d$ be compact. Let $\Omega \subset \R^{d}$ be open, and assume that $\gamma \cap \overline{\Omega}$ is contained in a compact Ahlfors $s$-regular set for some $s \in (0,d]$. Let $\mu$ be a Radon measure on $\gamma$ satisfying 
	\begin{equation*}
		U\mu(x)\ge M\quad\text{for approximately all $x\in \gamma \cap \Omega$}
	\end{equation*}
	for some $M>0$. Then $U\mu(x)\ge M$ for all $x\in \gamma \cap \Omega$. \end{cor}

There are two ways to get convinced about Corollary \ref{thm:minimum principleExtension}. One way is to take a look at the proof of \cite[Theorem 2.5]{reznikov2017minimum}, and verify that the arguments are local: so long as one is only interested in the estimate $U\mu \geq M$ on $\gamma \cap \Omega$, the Ahlfors regularity outside $\Omega$ is irrelevant. The caveat is that one needs to read the whole proof of \cite[Theorem 2.5]{reznikov2017minimum}. We give another way in Appendix \ref{appB}, where the reader needs less background on \cite{reznikov2017minimum}.


The following classical \emph{continuity principle} can be found in \cite[Theorem 5.1]{hayman1976subharmonic}.
\begin{thm}\label{thm:continuity principle}
	Let $\mu$ be a compactly supported Radon measure on $\R^{d}$. If $x \in \spt \mu$, and $U\mu$ is continuous at $x$ along $\spt \mu$, then $U\mu$ is continuous at $x$ along $\R^{d}$. In particular: if $\Omega \subset \R^{d}$ is an open set, and $U\mu \in C(\Omega \cap \spt\mu)$, then $U\mu\in C(\Omega)$.
\end{thm}

\subsection{Lemmas on measures in $\mathcal{M}(\gamma)$}\label{s:preliminariesII} Recall $\mathcal{M}(\gamma)$ from Definition \ref{def5Intro}. Here we prove two basic properties of $\mathcal{M}(\gamma)$, one of which was already stated as Lemma \ref{lemma13}.

\begin{lemma}\label{lemma12} Let $\gamma \subset \R^{d}$. Assume that $\mu \in \mathcal{M}(\gamma)$, and $B \subset \gamma$ is Borel. Then $\mu|_{B} \in \mathcal{M}(\gamma)$. More precisely: if $U\mu$ is continuous at any point $x_{0} \in \gamma$, then $U(\mu|_{B})$ is continuous at $x_{0}$. \end{lemma} 

\begin{proof} Property (a) in Definition \ref{def5Intro} remains clear for $\mu|_{B}$, so it remains to check property (b). Fix $x_{0} \in \gamma$ such that $U\mu$ is continuous at $x_{0}$. Abbreviate $k(x) := \max\{-\log |x|,0\} \in C_{c}(\R \, \setminus \, \{0\})$.  The continuity of $U\mu$ (resp. $U(\mu|_{B})$) at $x_{0}$ is equivalent to the continuity of $U_{k}\mu$ (resp. $U_{k}(\mu|_{B})$) at $x_{0}$, where $U_{k}$ the potential associated with the (non-negative) kernel $k$. So, it suffices to fix $\epsilon > 0$, and show that 
\begin{equation}\label{form152}  \limsup_{x \to x_{0}} \left| \int k(x - y) d(\mu|_{B})(y) - \int k(x_{0} - y) d(\mu|_{B})(y) \right| \leq 2\epsilon. \end{equation}

 By property (a), there exists $\delta_{0} > 0$ such that
\begin{displaymath} \int_{B(x_{0},\delta_{0})} k(x_{0} - y) \, d\mu(y) < \epsilon/2. \end{displaymath} 
Next, observe that the assumed continuity of $U\mu$ (hence $U_{k}\mu$) at $x_{0}$ also implies the continuity of $x \mapsto \int_{B(x_{0},\delta_{0})} k(x - y) \, d\mu(y)$ at $x_{0}$. Consequently, there exists $\delta_{1} > 0$ such that
\begin{displaymath} \int_{B(x_{0},\delta_{0})} k(x - y) \, d(\mu|_{B})(y) \leq \int_{B(x_{0},\delta_{0})} k(x - y) \, d\mu(y) < \epsilon, \qquad x \in B(x_{0},\delta_{1}). \end{displaymath}
This implies \eqref{form152}, since $x \mapsto \int_{\R^{d} \, \setminus \, B(x_{0},\delta_{0})} k(x - y) \, d(\mu|_{B})(y)$ is continuous at $x_{0}$. \end{proof}

We repeat the statement of Lemma \ref{lemma13}:

\begin{lemma} Let $\gamma \subset \R^{d}$, and let $f \colon \gamma \to \R^{D}$ be a bi-Lipschitz embedding. If $\mu \in \mathcal{M}(\gamma)$, then $f\mu \in \mathcal{M}(f(\gamma))$. \end{lemma} 

\begin{proof} Write $k(x) := \max\{-\log |x|,0\}$. Let $c \in (0,1]$ be a constant such that $|f(x) - f(y)| \geq c|x - y|$ for all $x,y \in \gamma$. Then, since $k$ is non-increasing,
\begin{equation}\label{form154} k(f(x) - f(y)) \leq k(c(x - y)) \leq k(x - y) + k(c). \end{equation}
Since $\mu$ and $f\mu$ are finite measures, this implies
\begin{displaymath} \bar{y} \mapsto \int k(\bar{x} - \bar{y}) \, d(f\mu)(\bar{y}) \in L^{1}(f\mu), \qquad \bar{x} \in f(\gamma), \end{displaymath}
verifying property (a) for $f\mu$.

To verify property (b), fix $\bar{x}_{0} = f(x_{0}) \in f(\gamma)$, and $\epsilon > 0$. We aim to show that
\begin{displaymath} \mathop{\limsup_{\bar{x} \to \bar{x}_{0}}}_{\bar{x} \in f(\gamma)} \left| \int k(\bar{x} - \bar{y}) \, d(f\mu)(\bar{y}) - \int k(\bar{x}_{0} - \bar{y}) \, d(f\mu)(\bar{y}) \right| \leq 4\epsilon, \end{displaymath}
or equivalently
\begin{equation}\label{form153} \limsup_{x \to x_{0}} \left| \int k(f(x) - f(y)) \, d\mu(y) - \int k(f(x_{0}) - f(y)) \, d\mu(y) \right| \leq 4\epsilon. \end{equation} 
Arguing as in Lemma \ref{lemma12}, there exist $\delta_{0},\delta_{1} \in (0,1]$ such that
\begin{displaymath} \int_{B(x_{0},\delta_{0})} k(x - y) \, d\mu(y) < \epsilon, \qquad x \in B(x_{0},\delta_{1}). \end{displaymath}
We may additionally choose $\delta_{0} > 0$ so small that $k(c) \cdot \mu(B(x_{0},\delta_{0})) < \epsilon$, since $\mu$ is non-atomic. Together with \eqref{form154}, the previous estimate implies
\begin{displaymath} \int_{B(x_{0},\delta_{0})} k(f(x) - f(y)) \, d\mu(y) < 2\epsilon, \qquad x \in B(x_{0},\delta_{1}). \end{displaymath}
This yields \eqref{form153}, since 
\begin{displaymath} x \mapsto \int_{\R^{d_{1}} \, \setminus B(x_{0},\delta_{0})} k(f(x) - f(y)) \, d\mu(y) \end{displaymath}
is continuous at $x_{0}$ by the bi-Lipschitz continuity of $f$. \end{proof} 

\subsection{Proof of Theorem \ref{main} assuming Theorem \ref{mainTechnical}}\label{s:mainProof} Let $\boldsymbol{\mu}$ be the logarithmic equilibrium measure on $\gamma$. It suffices to show that if $x_{0} \in \gamma_{\mathrm{reg}}$, then there exists $r > 0$ such that $\bmu|_{\gamma \cap B(x_{0},r)} \ll \mathcal{H}^{1}$. For notational convenience, we assume that $x_{0} = 0 \in \gamma_{\mathrm{reg}}$ and set $B(r) \coloneqq B(0,r)$.
	
	Let $p=2$ and $\delta=\delta(2,\alpha,d)\in (0,1)$ be the parameter from Theorem \ref{mainTechnical}. By the definition of $\gamma_{\mathrm{reg}}$, there exist small $\alpha,\epsilon,r > 0$ such that (after possibly rotating $\gamma$) we have
	\begin{equation}\label{form149} 
		\gamma \cap \bar{B}(3r) = \graph(A) \cap \bar{B}(3r)\quad \text{and} \quad \Gamma([-\epsilon r,\epsilon r]) \subset \gamma \cap \bar{B}(r), 
	\end{equation}
	where $A:\R \to \R^{d-1}$ is a $C^{1,\alpha}$-function with $\lip(A)\le\delta$, and $\Gamma(x)=(x,A(x))$.
	We will show that $\bmu|_{\Gamma([-\epsilon r,\epsilon r])} \ll \mathcal{H}^{1}$.

Decompose 
	\begin{displaymath} \boldsymbol{\mu} := \boldsymbol{\mu}|_{\bar{B}(2r)} + \boldsymbol{\mu}|_{\gamma \, \setminus \, \bar{B}(2r)} := \bmu_{1} + \bmu_{2}. \end{displaymath}
We start with a few elementary observations. First,
	\begin{equation}\label{form162} U\bmu_{2} \in \mathrm{Lip}(\bar{B}(r)), \end{equation}
	since $\spt \bmu_{2} \subset \R^{d} \, \setminus \, B(2r)$. Second, $U\bmu \geq \mathcal{E}_{\log}(\mu)$ approximately everywhere on $\spt \bmu$ by Theorem \ref{thm:eq-estimates}. In particular this lower bound holds approximately everywhere on $\spt \bmu \cap B(3r)$. But $\spt \bmu \cap \bar{B}(3r)$ is contained on the Ahlfors $1$-regular set $\Gamma$, so Corollary \ref{thm:minimum principleExtension} applied with $M := \mathcal{E}_{\log}(\mu)$ implies $U\bmu \geq \mathcal{E}_{\log}(\mu)$ everywhere on $\spt \bmu \cap B(3r)$. The corresponding upper bound $U\mu \leq \mathcal{E}_{\log}(\mu)$ on $\spt \bmu$ is also valid by Theorem \ref{thm:eq-estimates}, so we may conclude
	\begin{displaymath} U\bmu \in C(\spt \bmu \cap B(3r)). \end{displaymath}
	Theorem \ref{thm:continuity principle} then implies $U\bmu \in C(B(3r))$, and next $U\bmu_{1} \in C(B(3r))$ by Lemma \ref{lemma12}. But since $\spt \bmu_{1} \subset \bar{B}(2r)$, this shows $U\bmu_{1} \in C(\R^{d})$, and in particular $U\bmu_{1} \in \mathcal{M}(\Gamma)$.
		
	Let $\mu := \pi \bmu_{1}$, where $\pi(x_{1},\ldots,x_{d}) = x_{1}$ is the projection to the first coordinate. In other words $\mu(B) = \boldsymbol{\mu}_{1}(\pi^{-1}(B))$ for all Borel sets $B \subset \R$. Since $\bmu_{1} \in \mathcal{M}(\Gamma)$, and $\pi \colon \Gamma \to \R$ is bi-Lipschitz, also $\mu \in \mathcal{M}(\R)$ by Lemma \ref{lemma13}.
	
	We claim that also $\bmu_{1} = \Gamma\mu$. To see this, note that every set $B \subset \Gamma$ can be written as $B = \Gamma \cap \pi^{-1}(\pi(B))$. Since $\spt \bmu_{1} \subset \Gamma$ by \eqref{form149}, we have
	\begin{displaymath} \bmu_{1}(B) = \bmu_{1}(\pi^{-1}(\pi(B))) = \mu(\pi(B)) = (\Gamma\mu)(B), \qquad B \subset \Gamma. \end{displaymath}
	Now, the relation $\bmu_{1} = \Gamma \mu$ yields $U\bmu_{1}(\Gamma(x)) = U^{\Gamma}\mu(x)$ for all $x \in \R$. In particular,
	\begin{displaymath} U^{\Gamma}\mu(x) = U\bmu(\Gamma(x)) - U\bmu_{2}(\Gamma(x)), \qquad x \in \R. \end{displaymath}
	Recall again that $U\bmu = \mathcal{E}_{\log}(\mu)$ on $\spt \bmu$, and $U\bmu \geq \mathcal{E}_{\log}(\mu)$ on $\gamma$. Therefore: 
	\begin{itemize}
		\item $U^{\Gamma}\mu(x) = \mathcal{E}_{\log}(\mu) - U\bmu_{2}(\Gamma(x))$ for $x \in \spt \mu$, since $\Gamma(\spt \mu) = \spt \bmu_{1} \subset \spt \bmu$, and
		\item $U^{\Gamma}\mu(x) \geq \mathcal{E}_{\log}(\mu) - U\bmu_{2}(\Gamma(x))$ for $x \in [-\epsilon r,\epsilon r]$, since $\Gamma([-\epsilon r,\epsilon r]) \subset \gamma$ by \eqref{form149}.
		\item According to \eqref{form162}, 
		\begin{displaymath} x \mapsto \mathcal{E}_{\log}(\mu) - U\bmu_{2}(\Gamma(x)) \end{displaymath}
		is Lipschitz on $I_{0} := [-\epsilon r,\epsilon r]$, since $\Gamma([-\epsilon r,\epsilon r]) \subset \bar{B}(r)$.
	\end{itemize}
	From these points, we see that the hypotheses of Theorem \ref{mainTechnical} are satisfied for $\mu$, $I_{0}$, and the Lipschitz function $\mathcal{L}$ defined by $\mathcal{L}(x) = \mathcal{E}_{\log}(\mu) - U\bmu_{2}(\Gamma(x))$, for $x \in I_{0}$ (the behaviour of $\mathcal{L}$ outside $I_{0}$ is irrelevant, so it may be extended in an arbitrary Lipschitz fashion). Theorem \ref{mainTechnical} then implies that $\mu|_{I_{0}}$ is absolutely continuous with respect to Lebesgue measure. This implies that $\bmu|_{\Gamma(I_{0})} \ll \mathcal{H}^{1}$, since $\bmu|_{\Gamma(I_{0})} = \bmu_{1}|_{\Gamma(I_{0})} = (\Gamma \mu)|_{\Gamma(I_{0})}$, and $\Gamma$ is bi-Lipschitz. 
	
\subsection{$L^{p}$ bounds for potentials vs. Frostman estimates for measures}
The following proposition shows that $L^p$ estimates for the Riesz potential $k_\alpha\ast\mu$ imply Frostman estimates for $\mu$. For our purposes it is more convenient to state it in terms of the fractional Laplacian, which will be properly introduced in Section \ref{s:fractionalLaplacians}.

\begin{proposition}\label{frostmanProp} Let $\mu$ be a Radon measure. Then, for all $\alpha \in (0,1)$, $p \in [1,\infty]$,
	\begin{equation}\label{form69} \mu(B(x,r)) \lesssim_{\alpha} \|\Delta^{(\alpha - 1)/2}\mu\|_{L^{p}} \cdot r^{\alpha - 1/p}, \qquad x \in \R, \, r \in (0,1]. \end{equation}
\end{proposition}

\begin{remark} The non-negative function $\Delta^{(\alpha - 1)/2}\mu$ is pointwise defined as a convolution with (a constant times) the $\alpha$-dimensional Riesz kernel $k_\alpha(x)=|x|^{-\alpha}$, see \eqref{form161}. \end{remark}
\begin{remark}
	For $p=\infty$ the sharp version of \eqref{form69} involves $\|\Delta^{(\alpha - 1)/2}\mu\|_{\mathrm{BMO}}$ and is due to Adams \cite{adams1975note}.
\end{remark}

\begin{proof}[Proof of Proposition \ref{frostmanProp}] We may assume that $x = 0$ without loss of generality (the right hand side of \eqref{form69} is invariant under translating $\mu$). We may also assume that $r \in (0,\tfrac{1}{10}]$ by a trivial covering argument, since we only claim \eqref{form69} for $r \in (0,1]$.
	
	Let $\varphi \in C^{\infty}(\R)$ with $\mathbf{1}_{B(1)} \leq \varphi \leq \mathbf{1}_{B(2)}$. Write $\varphi_{r} \coloneqq \varphi(x/r)$ (without the usual $r^{-1}$-normalisation). By Plancherel and H\"older,
	\begin{align} \mu(B(r)) \leq \int \varphi_{r} \, d\mu & = \int |\xi|^{\alpha - 1}\hat{\mu}(\xi) \cdot |\xi|^{1 - \alpha}\widehat{\varphi_{r}}(\xi) \, d\xi \notag\\
		& =\int (\Delta^{(\alpha - 1)/2}\mu)(x) \cdot (\Delta^{(1 - \alpha)/2}\varphi_{r})(x) \, dx \notag\\
		&\label{form71} \leq \|\Delta^{(\alpha - 1)/2}\mu\|_{L^{p}}\|\Delta^{(1 - \alpha)/2}\varphi_{r}\|_{L^{p'}}, \end{align}
	where $p' \in [1,\infty]$ satisfies $1/p + 1/p' = 1$. 
	
	To establish that $\|\Delta^{(1 - \alpha)/2}\varphi_{r}\|_{L^{p'}} \lesssim_{\alpha} r^{\alpha -  1/p}$, we use the pointwise formula in \cite[(2.4.7)]{MR3243734}, similarly to \eqref{form53}-\eqref{form54}. The result looks like this:
	\begin{multline*} \sigma(1 - \alpha)(\Delta^{(1 - \alpha)/2}\varphi_{r})(x)  = \int_{|x - y| \geq 1} \frac{\sigma(\alpha - 2)\varphi_{r}(y)}{|y - x|^{2 - \alpha}} \, dy + b(\alpha)\varphi_{r}(x)\\ + \sigma(\alpha - 2)\int_{|x - y| < 1} \frac{\varphi_{r}(y) - \varphi_{r}(x)}{|y - x|^{2 - \alpha}} \, dy =: \Sigma_{1}(x) + \Sigma_{2}(x) + \Sigma_{3}(x).  \end{multline*} 
	Here $\sigma(1 - \alpha)$ and $\sigma(\alpha - 2)$ are the constants defined in \eqref{form157}, and $b(\alpha) = 2\sigma(\alpha - 2)/(\alpha - 1)$. (The formula for $\Delta^{(1 - \alpha)/2}$ looks here slightly different than the formula for $\Delta^{\beta/2}$ at \eqref{form53}-\eqref{form54}, because now we are applying \cite[(2.4.7)]{MR3243734} with $N = 0$ instead of $N = 1$.) We next claim that all the three terms $\Sigma_{j}$ in the pointwise formula for $\Delta^{(1 - \alpha)/2}\varphi_{r}$ satisfy
	\begin{equation}\label{form70} |\Sigma_{j}(x)| \lesssim_{\alpha} \min\{r^{\alpha - 1},r \cdot |x|^{\alpha - 2}\}, \qquad x \in \R. \end{equation}
	Note that $\min\{r^{\alpha - 1},r \cdot |x|^{\alpha - 2}\} = r^{\alpha - 1}$ for $|x| \leq r$.
	
	The term $\Sigma_{2}(x) = b(\alpha)\varphi_{r}(x)$ is simplest:
	\begin{displaymath} |\Sigma_{2}(x)| \leq |b(\alpha)| \cdot \mathbf{1}_{B(2r)}(x) \lesssim_{\alpha} r^{\alpha - 1}\mathbf{1}_{B(2r)}(x) \lesssim_{\alpha} \min\{r^{\alpha - 1},r \cdot |x|^{\alpha - 2}\}. \end{displaymath}
	
	We then deal with the terms $\Sigma_{1}$ and $\Sigma_{3}$. We consider separately the cases where $|x| > 3r$ and $|x| < 3r$. Assume first that $|x| > 3r$. Since $\spt \varphi_{r} \subset \bar{B}(2r)$, in this case $\varphi_{r}(x) = 0$. Moreover, whenever $\varphi_{r}(y) \neq 0$, it holds $|y - x| \sim |x|$. Consequently,
	\begin{displaymath} |\Sigma_{1}(x)| \lesssim_{\alpha} \int_{B(2r)} \frac{dy}{|x|^{2 - \alpha}} \sim r \cdot |x|^{\alpha - 2} \lesssim \min\{r^{\alpha - 1},r \cdot |x|^{\alpha - 2}\}. \end{displaymath}
	The term $\Sigma_{3}(x)$ can be estimated similarly, since $\varphi_{r}(x) = 0$.
	
	Assume next that $|x| < 3r$. In this case $\Sigma_{1}(x) = 0$, since $\varphi_{r}(y) \neq 0$ implies $|y| \leq 2r$, and therefore $|x - y| \leq 5r \leq \tfrac{1}{2}$ by our hypothesis $r \in (0,\tfrac{1}{10}]$. Thus $\varphi_{r}(y) = 0$ for $|x - y| \geq 1$.
	
	Regarding the term $\Sigma_{3}(x)$, for the part of the integral over $|y| \leq 6r$, we employ the Lipschitz estimate $|\varphi_{r}(x)  - \varphi_{r}(y)| \lesssim r^{-1}|x - y|$. For $|y| > 6r$, we insted use $|\varphi_{r}(x) - \varphi_{y}(y)| \leq 1$, and $|x - y| \sim |y|$. Combining these ingredients results in
	\begin{displaymath} |\Sigma_{3}(x)| \lesssim_{\alpha}  r^{-1} \int_{|y| \leq 6r} \frac{dy}{|y - x|^{1 - \alpha}} + \int_{|y| \geq 6r} \frac{dy}{|y|^{2 - \alpha}} \lesssim_{\alpha} r^{\alpha - 1} \lesssim \min\{r^{\alpha - 1},r \cdot |x|^{\alpha - 2}\}.  \end{displaymath}
	This completes the proof of \eqref{form70} for all $j \in \{1,2,3\}$, and all $x \in \R$.
	
	Using \eqref{form70}, we finally derive the desired estimate for $\|\Delta^{(1 - \alpha)/2}\varphi_{r}\|_{L^{p'}}$:
	\begin{align*} \|\Delta^{(1 - \alpha)/2}\varphi_{r}\|_{L^{p'}} & \leq \|\Delta^{(1 - \alpha)/2}\varphi_{r}\|_{L^{p'}(B(r))} + \|\Delta^{(1 - \alpha)/2}\varphi_{r}\|_{L^{p'}(\R \, \setminus \, B(r))}\\
		& \lesssim r^{\alpha - 1 + 1/p'} + \left(\int_{\R \, \setminus \, B(r)} (r \cdot |x|^{\alpha - 2})^{p'} \right)^{1/p'} \sim_{\alpha} r^{\alpha - 1 + 1/p'}. \end{align*} 
	Combined with \eqref{form71}, this completes the proof of Proposition \ref{frostmanProp}. \end{proof} 

\subsection{Fractional Laplacians acting on various function spaces}\label{s:fractionalLaplacians} For $\beta \in \C$ with $\Rea \beta > -1$, the fractional Laplacian $\Delta^{\beta/2}$ is formally the Fourier multiplier with symbol $|\xi|^{\beta}$. For $f \in \mathcal{S}(\R)$, the following pointwise definition is well-posed:
\begin{equation}\label{fracL} (\Delta^{\beta/2}f)(x) \stackrel{\mathrm{def.}}{=} \int_{\R} e^{2\pi i x \xi} |\xi|^{\beta}\hat{f}(\xi) \, d\xi.  \end{equation}
The integral is absolutely convergent thanks to the assumption $\Rea \beta > -1$. We will additionally need to make sense of $\Delta^{\beta/2}$ acting on $\dot{H}^{\beta}$, $L^{2}$, and the space of locally bounded functions with logarithmic growth at infinity, see Definitions \ref{def4}, \ref{L2FracL}, and \ref{def6}.  

\begin{definition}[Homogeneous Sobolev space $\dot{H}^{\beta}$]\label{def3} Let $\beta \in \C$ with $\Rea \beta > -1$. The homogeneous Sobolev space $\dot{H}^{\beta}$ consists of $\Lambda \in \mathcal{S}'(\R)$ such that $\hat{\Lambda} \in L^{1}_{\mathrm{loc}}(\R)$, and 
	\begin{displaymath} \xi \mapsto \hat{\Lambda}(\xi)|\xi|^{\beta} \in L^{2}(\R). \end{displaymath}
\end{definition} 

The fractional Laplacian $\Delta^{\beta/2}$ may be defined on $\dot{H}^{\beta}$ as an element of $L^{2}(\R)$:
\begin{definition}[$\Delta^{\beta/2}$ acting on $\dot{H}^{\beta}$]\label{def4} Let $\beta \in \C$ with $\Rea \beta > -1$, and $\Lambda \in \dot{H}^{\beta}$. Then $\xi \mapsto \hat{\Lambda}(\xi)|\xi|^{\beta} \in L^{2}(\R)$ is the Fourier transform of some $g \in L^{2}(\R)$. We define 
	\begin{displaymath} \Delta^{\beta/2}\Lambda \coloneqq g \in L^{2}(\R). \end{displaymath}   \end{definition}

For $\Rea \beta \geq 0$, we also need a definition of fractional Laplacians acting on $L^{2}(\R)$. If $f \in L^{2}(\R)$ and $\beta \in \C$ with $\Rea \beta \geq 0$, note that $\xi \mapsto |\xi|^{\beta}\hat{f}(\xi)$ defines a tempered distribution, indeed by Cauchy-Schwarz and Plancherel
\begin{displaymath} \left| \int |\xi|^{\beta}\hat{f}(\xi) \cdot \varphi(\xi) \, d\xi \right| \leq \|f\|_{L^{2}} \|(1 + |\xi|)^{2\ceil{\Rea \beta}}\varphi\|_{L^{2}}, \qquad \varphi \in \mathcal{S}(\R). \end{displaymath} 
\begin{definition}[$\Delta^{\beta/2}$ acting on $L^{2}(\R)$]\label{L2FracL} For $f \in L^{2}(\R)$ and $\beta \in \C$ with $\Rea \beta \geq 0$, we define $\Delta^{\beta/2}f \in \mathcal{S}'(\R)$ as the inverse (distributional) Fourier transform of $\xi \mapsto |\xi|^{\beta}\hat{f}(\xi)$. In other words, 
\begin{displaymath} (\Delta^{\beta/2}f)g := \int |\xi|^{\beta/2}f(\xi) \cdot \widecheck{g}(\xi) \, d\xi, \qquad g \in \mathcal{S}(\R). \end{displaymath} \end{definition} 

\subsubsection{Pointwise formulas and decay bounds} We next record some pointwise expressions for $\Delta^{\beta/2}f$, $f \in \mathcal{S}(\R)$. For $\Rea \beta \in (0,1)$, the negative order fractional Laplacian $\Delta^{-\beta/2}f$ coincides with convolution with a constant times the $(1 - \beta)$-dimensional Riesz kernel $k_{1 - \beta}$, where 
\begin{displaymath} k_{\beta}(x) = |x|^{-\beta}, \qquad x \in \R \, \setminus \, \{0\}. \end{displaymath} 
More precisely this constant is $c(\beta) = \pi^{\beta - 1/2}\Gamma((1 - \beta)/2)/\Gamma(\beta/2)$, thus
\begin{equation}\label{form161} (\Delta^{-\beta/2}f)(x) = c(\beta)\int f(y)|x - y|^{\beta - 1}, \qquad x \in \R, \, \Rea \beta \in (0,1). \end{equation}
For a proof, see \cite[Proposition 4.1]{MR2003254}. In particular, the formula \eqref{form161} allows one to extend the pointwise definition of $(\Delta^{-\beta/2}\mu)(x)$ to all positive Borel measures $\mu$ on $\R$ (although it is now possible that $(\Delta^{-\beta/2}\mu)(x) = +\infty$).

We will also need a pointwise expression for positive-order fractional Laplacians. In the terminology and notation of \cite[Section 2.4.3]{MR3243734}, we may  write
\begin{equation}\label{form52} \sigma(\beta)(\Delta^{\beta/2}f)(x) = u_{\beta}(e_{x}\hat{f}), \qquad \Rea \beta > -1, \, f \in \mathcal{S}(\R), \end{equation} 
where 
\begin{equation}\label{form157} \sigma(\beta) := \frac{\pi^{(\beta + 1)/2}}{\Gamma((\beta + 1)/2)}, \end{equation}
$u_{\beta}$ is the \emph{homogeneous distribution} of order $\beta$, and $e_{x}(\xi) \coloneqq e^{2\pi i x \cdot \xi}$ for $x,\xi \in \R$. The homogeneous distribution $u_{\beta}$ is a special tempered distribution defined for all $\beta \in \C$, and by \cite[Theorem 2.4.6]{MR3243734} its Fourier transform satisfies
\begin{displaymath} \widehat{u_{\beta}} = u_{-1 - \beta}, \qquad \beta \in \C, \end{displaymath}
In particular, continuing from \eqref{form52}, and noting that $e_{x}\hat{f}$ coincides with the Fourier transform of $y \mapsto f(x + y) =: (\tau_{x}f)(y)$, we find
\begin{displaymath} (\Delta^{\beta/2}f)(x) = \sigma(\beta)^{-1}\widehat{u_{\beta}}(\tau_{x}f) = \sigma(\beta)^{-1}u_{-1 - \beta}(\tau_{x}f). \end{displaymath} 
Now, assume that $\Rea \beta < 2$. Then $\Rea (-1 - \beta) > -3 = -N - n - 1$ with $N = n = 1$. Therefore \cite[(2.4.7)]{MR3243734} with the choice $N = 1$ provides the following expression for $u_{-1 - \beta}$:
\begin{align} \sigma(\beta)(\Delta^{\beta/2}f)(x) & = \int_{|y| \geq 1} \sigma(-1 - \beta)[\tau_{x}f](y) \cdot |y|^{-1 - \beta} \, dy \notag\\
	&\qquad + b(\beta,0)[\tau_{x}f](0) + b(\beta,1)[\tau_{x}f]'(0) \notag\\
	& \qquad + \int_{|y| < 1} \sigma(-1 - \beta) \left\{[\tau_{x}f](y) - [\tau_{x}f](0) - [\tau_{x}f]'(0)y \right\} \cdot |y|^{-1 - \beta} \, dy \notag\\
	&\label{form53} = \int_{|x - y| \geq 1} \frac{\sigma(-1 - \beta)f(y)}{|x - y|^{1 + \beta}} \, dy + b(\beta,0)f(x) + b(\beta,1)f'(x)\\
	&\label{form54}\qquad + \int_{|x - y| < 1} \frac{\sigma(-1 - \beta)[f(y) - f(x) - f'(x)(y - x)]}{|y - x|^{1 + \beta}} \, dy. \end{align}
The values of the constants $b(\beta,0)$ and $b(\beta,1)$ are computed above \cite[(2.4.7)]{MR3243734}. In fact $b(\beta,1) = 0$, and
\begin{displaymath} b(\beta,0) = \frac{2\sigma(-1 - \beta)}{-\beta} = \frac{2\pi^{-\beta/2}}{-\beta \cdot \Gamma\left(\frac{-\beta}{2} \right)}. \end{displaymath} 
The function $\beta \mapsto \Gamma(-\beta/2)$ has a simple pole at $\beta = 0$, so the function $\beta \mapsto b(\beta,0)$ stays bounded in a neighbourhood of $\beta = 0$.

Note that the integrals in \eqref{form53}-\eqref{form54} are absolutely convergent for $\Rea \beta < 2$, and for $f \in \mathcal{S}(\R)$. From the expressions \eqref{form53}-\eqref{form54}, we may deduce the following decay estimate:
\begin{lemma}\label{lemma3} Let $-1 < \Rea \beta < 2$. Then 
	\begin{equation}\label{form74} |\Delta^{\beta/2}f(x)| \lesssim_{\beta,f} (1 + |x|)^{-1 - \Rea \beta}, \qquad x \in \R, f \in \mathcal{S}(\R). \end{equation} 
	Moreover, the implicit constants remain uniformly bounded if $f \in \mathcal{S}(\R)$ is fixed, and $\beta$ ranges in a compact subset of the strip $-1 < \Rea \beta < 2$.
\end{lemma}

To prove this, we record an elementary lemma on how convolution preserves decay:
\begin{lemma}\label{lemma5} Let $\alpha > 0$, $C_{1},C_{2} \geq 1$, and $d \in \N$, and let $F,G \colon \R^{d} \to \C$ be functions satisfying
	\begin{displaymath} |F(x)| \leq C_{1}(1 + |x|)^{-\alpha} \quad \text{and} \quad |G(x)| \leq C_{2}(1 + |x|)^{-d - \alpha}, \qquad x \in \R^{d}. \end{displaymath}
	Then,
	\begin{displaymath} |(F \ast G)(x)| \lesssim C_{1}C_{2}(1 + |x|)^{-\alpha}, \qquad x \in \R^{d}. \end{displaymath}
\end{lemma}

\begin{proof} Replacing $F$ by $F/C_{1}$ and $G$ by $G/C_{2}$, we may assume that $C_{1} = 1 = C_{2}$. In this proof, the "$\lesssim$" notation may hide constants depending on $\alpha,d$. Fix $x \in \R^{d}$. For $|x| < 2$, we use the trivial bound
	\begin{displaymath} |(F \ast G)(x)| \leq \|F\|_{L^{\infty}}\|G\|_{L^{1}} \lesssim 1. \end{displaymath} 
	Assume next that $|x| \geq 2$, and decompose
	\begin{displaymath} |(F \ast G)(x)| \leq \int |F(y)||G(x - y)| \, dy \lesssim \sum_{j \geq 0} 2^{-j(d + \alpha)} \int_{B(x,2^{j})} |F(y)| \, dy. \end{displaymath}
	For $2^{j} \leq |x|/2$ and $y \in B(x,2^{j})$, note that $|y| \geq |x|/2$. Therefore,
	\begin{displaymath} \int_{B(x,2^{j})} |F(y)| \, dy \lesssim 2^{jd}|x|^{-\alpha}. \end{displaymath} 
	and consequently
	\begin{displaymath} \sum_{1 \leq 2^{j} \leq |x|/2} 2^{-j(d - \alpha - \alpha)} \int_{B(x,2^{j})} |F(y)| \, dy \lesssim |x|^{-\alpha} \sum_{1 \leq 2^{j} \leq |x|/2} 2^{-\alpha j} \sim |x|^{-\alpha}. \end{displaymath}
	Next, for the "tail" terms $2^{j} > |x|/2$, we estimate $\int_{B(x,2^{j})} |F(y)| \, dy \lesssim 2^{jd}$. Consequently,
	\begin{displaymath} \sum_{2^{j} > |x|/2} 2^{-j(d + \alpha)} \int_{B(x,2^{j})} |F(y)| \, dy \lesssim \sum_{2^{j} > |x|/2} 2^{-j(d + \alpha)} \cdot 2^{jd} \lesssim |x|^{-\alpha}. \end{displaymath} 
	This concludes the proof of Lemma \ref{lemma5}. \end{proof} 

We then complete the proof of Lemma \ref{lemma3}:

\begin{proof}[Proof of Lemma \ref{lemma3}] The $L^{\infty}$-estimate $|\Delta^{\beta/2}f(x)| \lesssim_{\beta,f} 1$ is clear from the pointwise definition \eqref{fracL}:
	\begin{displaymath} |(\Delta^{\beta/2}f)(x)| \leq \||\cdot|^{\Rea \beta}\hat{f}\|_{L^{1}} < \infty. \end{displaymath}
	Let us then establish the decay estimate $|\Delta^{\beta/2}f(x)| \lesssim_{\beta,f} |x|^{-1 - \Rea \beta}$ for $|x| \geq 2$. Fix $x \in \R$ with $|x| \geq 2$. There are altogether 4 terms in \eqref{form53}-\eqref{form54}, and we establish the decay stated in \eqref{form74} for all of them individually. The terms $b(\beta,0)f(x)$ and $b(\beta,1)f'(x)$ are the easiest:
	\begin{displaymath} \max\{ |b(\beta,0)f(x)|,|b(\beta,1)f'(x)|\} \lesssim |x|^{-3} \leq |x|^{-1 - \Rea \beta}, \end{displaymath}
	since $f \in \mathcal{S}$, and $\Rea \beta < 2$.
	
	Next, we handle the integral term on line \eqref{form53}:
	\begin{equation}\label{form142} \left| \int_{|x - y| \geq 1} \frac{f(y)}{|x - y|^{1 + \beta}} \, dy \right| \leq (F \ast |f|)(x), \end{equation} 
	where $F(u) \coloneqq \min\{|y|^{-1 - \Rea \beta},1\}$. This term can be handled with Lemma \ref{lemma5} with $G \coloneqq |f|$ and $\alpha \coloneqq 1 + \Rea \beta$. Evidently $|G(x)| \lesssim (1 + |x|)^{-1 - \alpha}$, since $f \in \mathcal{S}(\R)$. Therefore, from Lemma \ref{lemma5} we deduce that $(F \ast |f|)(x) \lesssim |x|^{-1 - \Rea \beta}$.
	
	We finally estimate the integral term in \eqref{form54}, reproduced here: 
	\begin{displaymath} \int_{|x - y| < 1} \frac{\sigma(-1 - \beta)[f(y) - f(x) - f'(x)(y - x)]}{|y - x|^{1 + \beta}} \, dy. \end{displaymath} 
	Recall that $|x| \geq 2$. For $y \in \bar{B}(x,1)$, estimate as follows:
	\begin{displaymath} |f(y) - f(x) - f'(x)(y - x)| = \left| \int_{x}^{y} \int_{x}^{s} f''(r) \, dr \, ds \right| \leq \|f''\|_{L^{\infty}(\bar{B}(x,1))} |x - y|^{2} . \end{displaymath} 
	Since $f \in \mathcal{S}(\R)$, and we assumed $\Rea \beta < 2$, it holds $\|f''\|_{L^{\infty}(\bar{B}(x,1))} \lesssim_{f} |x|^{-3} \leq |x|^{-1 - \Rea \beta}$. This leads to 
	\begin{displaymath} \left| \int_{|x - y| < 1} \frac{\sigma(-1 - \beta)[f(y) - f(x) - f'(x)(y - x)]}{|y - x|^{1 + \beta}} \, dy \right| \lesssim_{f} |x|^{-1 - \Rea \beta} \int_{|x - y| < 1} \frac{|\sigma(-1 - \beta)|}{|y - x|^{\Rea \beta - 1}} \, dy. \end{displaymath}
	Again using the assumption $\Rea \beta < 2$, the integral on the right is a finite constant depending on $\beta$. This completes the proof of Lemma \ref{lemma3}.  \end{proof}

Finally, for $\beta \in \C$ with $\Rea \beta > 0$, Lemma \ref{lemma3} allows us to define fractional Laplacians $\Delta^{\beta/2}$ acting on locally bounded functions with logarithmic growth at infinity:
\begin{definition}[$\Delta^{\beta/2}$ acting on $(1 + \log_{+}) L^{\infty}(\R)$]\label{def6} We write $(1 + \log_{+}) L^{\infty}$ for the space of measurable functions $U \colon \R \to \R$ for which there exists a constant $C > 0$ such that
	\begin{displaymath} |f(x)| \leq C(1 + \log_{+}|x|), \qquad x \in \R, \end{displaymath}
	Here $\log_{+}(x) = \max\{0,\log x\}$ for $x > 0$. For $U \in (1 + \log_{+}) L^{\infty}$, and $\beta \in \C$ with $\Rea \beta > 0$, we define the tempered distribution $\Delta^{\beta/2}U$ by
	\begin{displaymath} (\Delta^{\beta/2}U)(\psi) := \int U \cdot \Delta^{\beta/2}\psi, \qquad \psi \in \mathcal{S}(\R). \end{displaymath}
	The integral on the right is absolutely convergent thanks to Lemma \ref{lemma3}. \end{definition}

\begin{ex} The prototypical example of a function $U \in (1 + \log_{+})L^{\infty}(\R)$ is given by $U = U^{\Gamma}\mu$, where $\mu \in \mathcal{M}(\R)$, recall Definition \ref{def5Intro}. \end{ex}

\subsubsection{Imaginary Laplacians and related operators}  For $\Rea \beta = 0$, the fractional Laplacian $\Delta^{\beta/2}$ extends to a bounded operator on $L^{p}(\R) \to L^{p}(\R)$, for every $p \in (1,\infty)$. This follows from Mihlin's multiplier theorem \cite[Theorem 5.2.7]{MR3243734}, which we state here:
\begin{thm}\label{t:mihlin} Let $A > 0$, and let $m \colon \R^{n} \, \setminus \, \{0\} \to \C$ be a function satisfying
	\begin{equation}\label{form105} |\partial^{\alpha}m(\xi)| \leq A|\xi|^{-\alpha}, \qquad |\alpha| \leq \ceil{n/2} + 1. \end{equation}
	Then the Fourier multiplier $T_{m}$ with symbol $m$ is bounded on $L^{p}(\R)$ for every $1 < p < \infty$. In fact,
	\begin{displaymath} \|T_{m}\|_{L^{p} \to L^{p}} \lesssim_{n} \max\{p,(p - 1)^{-1}\} \cdot (A + \|m\|_{L^{\infty}}). \end{displaymath}\end{thm}

We will only apply Mihlin's theorem to the following symbols:
\begin{cor}\label{cor3} Let $\beta \in \C$ with $\Rea \beta = 0$, 
	\begin{displaymath} m_{\beta}(\xi) := |\xi|^{\beta} \quad \text{and} \quad n_{\beta}(\xi) := |\xi|^{1 + \beta}/\xi, \end{displaymath}
	where $\xi \in \R \, \setminus \, \{0\}$. Then $m_{\beta},n_{\beta}$ define Fourier multipliers $T_{m_{\beta}},T_{n_{\beta}}$ satisfying
	\begin{displaymath} \max\{\|T_{m_{\beta}}\|_{L^{p} \to L^{p}},\|T_{n_{\beta}}\|_{L^{p} \to L^{p}}\} \lesssim \max\{p,(p - 1)^{-1}\} \cdot (1 + (\mathrm{Im\,} \beta)^{2}). \end{displaymath}  \end{cor} 

\begin{proof} We apply Theorem \ref{t:mihlin} with $n = 1$, so we only need to make sure that the derivative condition \eqref{form105} is valid for $\alpha \in \{0,1,2\}$. The case $\alpha = 0$ follows from $\Rea \beta = 0$, which implies $\|m_{\beta}\|_{L^{\infty}} = 1 = \|n_{\beta}\|_{L^{\infty}}$. For $\xi \neq 0$, the first and second derivatives satisfy
	\begin{displaymath} |m_{\beta}'(\xi)| = |\beta||\xi|^{-1} \quad \text{and} \quad |m_{\beta}''(\xi)| = |\beta||\beta - 1||\xi|^{-2},  \end{displaymath} 
	and the same formulae are valid for $|n_{\beta}'(\xi)|,|n_{\beta}''(\xi)|$. Therefore \eqref{form105} has been verified with constant $A \lesssim 1 + (\mathrm{Im\,} \beta)^{2}$, as claimed. \end{proof} 

\section{H\"older continuity of \texorpdfstring{$U\mu$}{U mu}}\label{s:holder}
The following result will be shown in Sections \ref{s:TGammaBeta}--\ref{s:limit operators}, see Proposition \ref{prop8-2}.

\begin{proposition}\label{prop8} Let $\beta \in [0,1)$ and $p \in (1/(1 - \beta),\infty)$. Assume that $\mathrm{Lip}(A)$ is sufficiently small depending on $d,p$.	
	Then, for any $\mu\in\mathcal{M}(\R)$
	\begin{displaymath} \|\Delta^{-\beta/2}\mu\|_{L^{p}} \lesssim_{p} \|\Delta^{(1 - \beta)/2}U^{\Gamma}\mu\|_{L^{p}}. \end{displaymath}
	Here $\Delta^{(1 - \beta)/2}U^{\Gamma}\mu \in \mathcal{S}'(\R)$ is the tempered distribution introduced in Definition \ref{def6}. \end{proposition}
\begin{remark}
	Proposition \ref{prop8} says that Sobolev regularity of $U^{\Gamma}\mu$ implies certain (negative order) Sobolev regularity of $\mu$. This is similar to a classical result of Wallin \cite{wallin1966existence} which states that H\"older continuity of a Riesz potential $k_\alpha\ast\mu$ implies Frostman regularity of $\mu$.
\end{remark}
We will take this result for granted now. In Sections \ref{s:holder} and \ref{s:Lpfractional} we use Proposition \ref{prop8} to prove Theorem \ref{mainTechnical}. 

Taking $\beta=0$ and some $p\in (1,\infty)$, we see from Proposition \ref{prop8} that
\begin{equation*}
	\|\Delta^{1/2}U^\Gamma(\mu|_I)\|_{L^p}<\infty\quad\Rightarrow\quad \mu\in L^p(I),
\end{equation*}
which is the aim of Theorem \ref{mainTechnical}. Thus, assuming Proposition \ref{prop8}, the goal becomes proving $\|\Delta^{1/2}U^\Gamma(\mu|_I)\|_{L^p}<\infty$ for some $p\in (1,\infty)$. 

In this section we prove H\"older estimates for the potentials $U^\Gamma\mu$. The main result is Proposition \ref{prop17a}. In Section \ref{s:Lpfractional} these H\"older estimates are used to infer $L^p$-bounds for the fractional Laplacians of $U^\Gamma\mu$, see Corollary \ref{cor14}. Finally, in Subsection \ref{s:proofMainTechnical} we use these $L^p$-bounds together with Proposition \ref{prop8} to prove $\|\Delta^{1/2}U^\Gamma(\mu|_I)\|_{L^p}<\infty$. The proof involves a bootstrapping scheme which uses Proposition \ref{prop8} for many different $\beta\in [0,1)$.

\subsection{Kernel decomposition} The purpose of this section is to define the decomposition $U\mu = P\mu + R\mu$ into a "principal" potential and "remainder" potential, as we briefly discussed in Section \ref{s:outline}. Let $A \colon \R \to \R^{d - 1}$ be $\delta_{1}$-Lipschitz. Assume moreover that $\nabla A$ is $\alpha$-H\"older continuous with constant $\delta_{2} > 0$:
\begin{displaymath} |\nabla A(x) - \nabla A(y)| \leq \delta_{2} |x - y|^{\alpha}, \qquad x,y \in \R. \end{displaymath}
Let $U=U^\Gamma$ be the kernel
\begin{displaymath} U(x,y) \coloneqq \log|\Gamma(x)-\Gamma(y)|^{-1}= \log|(x - y,A(x) - A(y))|^{-1}, \qquad x,y \in \R, \, x \neq y. \end{displaymath}
(We will use the notation "$U$" for both the kernel and the operator $\mu \mapsto U\mu$.) 

In what follows we use the notation 
\begin{displaymath} B_{r}(x) \coloneqq (B \ast \psi_{r})(x), \qquad x \in \R, \end{displaymath}
where $\psi_{r}(y) = r^{-1}\psi(y/r)$ is a standard approximate identity, in particular $\spt \psi \subset [-1,1]$. 

Consider the following "principal" kernel $P$:
\begin{equation}\label{form25} P(x,y) \coloneqq \log|(x - y,\nabla A_{|x - y|}(x)(x - y) )|^{-1}, \qquad (x,y) \in \R, \, x \neq y. \end{equation}
Above $\nabla A_{|x - y|}(x) = (\nabla A)\ast \psi_{|x-y|}(x)\in\R^{d-1}$ with the convolution applied to each component of $\nabla A$ separately. 
\begin{remark} The kernel $P(x,y)$ can also be written as
	\begin{displaymath} P(x,y) = \log|(1,\nabla A_{|x - y|}(x))|^{-1} + \log|x - y|^{-1}. \end{displaymath} 
\end{remark} 

Finally, we define the "remainder" kernel
\begin{displaymath} R(x,y) \coloneqq U(x,y) - P(x,y), \end{displaymath}
so that $U = P + R$. The next two lemmas contain the main properties of $P$ and $R$:
\begin{lemma}\label{PLemma} Let $x,y \in \R$ with $x \neq y$. If the Lipschitz constant $\delta_{1} > 0$ is sufficiently small, then,
	\begin{displaymath} \partial_{x}^{2}P(x,y) \geq 0 \quad \text{and} \quad \partial_{y}^{2} P(x,y) \geq 0. \end{displaymath} \end{lemma}

\begin{remark} In particular, this lemma does not use the H\"older continuity of $\nabla A$. \end{remark}

We will use Lemma \ref{PLemma} in the form of the next corollary:

\begin{proposition}\label{prop15} Let $\nu$ be a finite (positive) Borel measure on $[0,1]$, and let $I \subset \R \, \setminus \, \spt \nu$ be an open interval. Then $P\nu$ is convex on $I$.
\end{proposition} 

\begin{proof} Fix $x \in I$, and note that $\dist(x,\spt \nu) > 0$. This enables us to differentiate (twice) under the integral sign, and conclude that
	\begin{displaymath} \partial_{x}^{2} (P\nu)(x) = \partial_{x}^{2} \int_{\spt \nu} P(x,y) \, d\nu(y) = \int_{\spt \nu} \partial_{x}^{2}P(x,y) \, d\nu(y) \geq 0. \end{displaymath}
	This concludes the proof. \end{proof} 

\begin{lemma}\label{RLemma} The following holds if the Lipschitz constant $\delta_{1} > 0$ is sufficiently small. Let $x,y \in \R$ with $x \neq y$. Then,
	\begin{equation}\label{size} R(x,y) \lesssim \min\{\delta_{1},\delta_{2}|x - y|^{\alpha}\}. \end{equation}
	Moreover, if $x' \in \R$ with $|x - x'| \leq \tfrac{1}{2}|x - y|$, then
	\begin{equation}\label{continuity} |R(x,y) - R(x',y)| \lesssim \frac{|x - x'| \cdot \min\{\delta_{1},\delta_{2}|x - y|^{\alpha}\}}{|x - y|}. \end{equation}
\end{lemma} 

\begin{remark} If $\delta_{1}$ is not assumed small, the only change in the statement would be that the condition $|x - x'| \leq \tfrac{1}{2}|x - y|$ needs to be replaced by $|x - x'| \leq c(\delta_{1})|x - y|$. \end{remark} 

We prove the lemmas.
\begin{proof}[Proof of Lemma \ref{PLemma}] We give the details for the lower bound $\partial_{x}^{2}P(x,y) \geq 0$, and indicate during the argument where the proof differs slightly for the estimate $\partial_{y}^{2}P(x,y) \geq 0$. Fix $x,y \in \R$ with $x \neq y$, and write 
	\begin{displaymath} P(x,y)  = \log|(1,\nabla A_{|x - y|}(x))|^{-1} + \log|x - y|^{-1} =: f(x) + g(x). \end{displaymath} 
	With this notation,
	\begin{displaymath} \partial_{x}^{2} P(x,y) = f''(x)+g''(x). \end{displaymath} 
	The $g''(x)$-term is strictly positive, while it will turn out that $f''(x)$ can be made "small enough" by choosing $0<\delta_{1} <1$ sufficiently small. Clearly
	\begin{displaymath} g''(x) =  \frac{1}{|x - y|^{2}}. \end{displaymath}
	
	Regarding $f''(x)$, we claim
	\begin{equation}\label{form33} |f''(x)| \lesssim \frac{\delta_{1}}{|x - y|^{2}}. \end{equation}
	To verify this, one needs to check by elementary computation that
	\begin{equation}\label{form51} |\partial_{x} \psi_{|x - y|}(x - z)| \lesssim \frac{\mathbf{1}_{B(x,|x - y|)}(z)}{|x - y|^{2}} \quad \text{and} \quad |\partial_{x}^{2} \psi_{|x - y|}(x - z)| \lesssim \frac{\mathbf{1}_{B(x,|x - y|)}(z)}{|x - y|^{3}}. \end{equation}
	for $x,y,z \in \R$ with $x \neq y$. In the proof for $\partial_{y}^{2}P(x,y) \geq 0$, we would instead apply similar upper bounds for the $\partial_{y}$ and $\partial_{y}^{2}$ derivatives of $\psi_{|x - y|}(x - z)$.
	
	We then start the proof of \eqref{form33}.	Denoting
	\begin{displaymath} \gamma(x) \coloneqq \gamma_{y}(x) \coloneqq (1,\nabla A_{|x - y|}(x)) \end{displaymath} 
	we have
	\begin{displaymath} f'(x) = - \frac{\gamma(x) \cdot \dot{\gamma}(x)}{|\gamma(x)|^{2}}. \end{displaymath} 
	and
	\begin{equation}\label{eq:f''}
		f''(x) = 2 \cdot \frac{(\gamma(x) \cdot \dot{\gamma}(x))^{2}}{|\gamma(x)|^{4}} - \frac{|\dot{\gamma}(x)|^{2}}{|\gamma(x)|^{2}} -  \frac{\gamma(x) \cdot \ddot{\gamma}(x)}{|\gamma(x)|^{ 2}}.  
	\end{equation}
	Since $|\gamma(x)| \geq 1$, the first and second terms above are bounded in absolute value by $\lesssim |\dot{\gamma}(x)|^{2}.$ Noting that $\dot{\gamma}(x) = (0,\partial_{x}\nabla A_{|x - y|}(x))$,
	\begin{displaymath} |\dot{\gamma}(x)| \leq \int |\nabla A(z)||\partial_{x}\psi_{|x - y|}(x - z)| \, dz \stackrel{\eqref{form51}}{\lesssim} \delta_{1} \int_{B(x,|x - y|)} \frac{dz}{|x - y|^{2}} \lesssim \frac{\delta_{1}}{|x - y|}. \end{displaymath}
	For the third term in \eqref{eq:f''} we obtain a similar estimate using \eqref{form51}:
	\begin{displaymath} |\ddot{\gamma}(x)| \leq \int |\nabla A(z)||\partial_{xx}\psi_{|x - y|}(x - z)| \, dz \stackrel{\eqref{form51}}{\lesssim} \delta_{1} \int_{B(x,|x - y|)} \frac{dz}{|x - y|^{3}} \lesssim \frac{\delta_{1}}{|x - y|^{2}}. \end{displaymath}
	All in all, $|f''(x)| \lesssim (\delta_{1}^2 + \delta_1)/|x - y|^{2}\lesssim  \delta_1/|x - y|^{2}$. This completes the proof of \eqref{form33}, and the lemma. \end{proof}

We move to the proof of Lemma \ref{RLemma}. The H\"older continuity estimate \eqref{continuity} will be based on the following general lemma, which charts how much smooth maps may increase "second order differences": 

\begin{lemma}\label{lemma2} Let $a,b,c,d \in \R^{d}$ be distinct. Write
	\begin{displaymath} M \coloneqq \min\{|b - a|,|d - c|\}, \, m \coloneqq \max\{|c - a|,|d - b|\}, \quad \text{and} \quad \epsilon \coloneqq |(b - a) - (d - c)|. \end{displaymath}
	Assume that $f \in C^{2}(\bar{\Omega})$, where where $\bar{\Omega} \coloneqq \mathrm{conv}(a,b,c,d)$ (the closed convex hull of the points $a,b,c,d$). Then,
	\begin{displaymath} |(f(b) - f(a)) - (f(d) - f(c))| \leq m \cdot M \cdot \|\nabla^{2}f\|_{L^{\infty}(\bar{\Omega})} + \epsilon \cdot \|\nabla f\|_{L^{\infty}(\bar{\Omega})}. \end{displaymath} \end{lemma}

\begin{remark} The geometric intuition is that $\bar{\Omega}$ is a quadrilateral with "long" sides having length roughly $M$ and "short" sides having length roughly $m$. The factor $m \cdot M$ is like the area of $\bar{\Omega}$, whereas the quantity $\epsilon$ measures how far $\bar{\Omega}$ is from being a parallelogram.  \end{remark}

\begin{proof}[Proof of Lemma \ref{lemma2}] Assume for example that $M = |b - a| > 0$. Let $\gamma,\eta \colon [0,|b - a|] \to \R^{d}$ be the following paths parametrising the segments $[a,b]$ and $[c,d]$:
	\begin{displaymath} \gamma(\theta) \coloneqq \frac{\theta b}{|b - a|} + \frac{(|b - a| - \theta)a}{|b - a|}, \qquad \theta \in [0,|b - a|], \end{displaymath} 
	and
	\begin{displaymath} \eta(\theta) \coloneqq \frac{\theta d}{|b - a|} + \frac{(|b - a| - \theta)c}{|b - a|}, \qquad \theta \in [0,|b - a|]. \end{displaymath}
	With this notation, and since $f(b) = (f \circ \gamma)(|b - a|)$ etc, we may write
	\begin{align} (f(b) - f(a)) - (f(d) - f(c)) & = \int_{0}^{|b - a|} (\nabla f(\gamma(\theta)) \cdot \dot{\gamma}(\theta) - \nabla f(\eta(\theta)) \cdot \dot{\eta}(\theta)) \, d\theta \notag\\
		&\label{form37} = \int_{0}^{|b - a|} (\nabla f(\gamma(\theta)) - \nabla f(\eta(\theta))) \cdot \dot{\gamma}(\theta) \, d\theta\\
		&\label{form38} \qquad + \int_{0}^{|b - a|} \nabla f(\eta(\theta)) \cdot [\dot{\gamma}(\theta) - \dot{\eta}(\theta)] \, d\theta. \end{align}
	To estimate further, we note that
	\begin{displaymath} |\gamma(\theta) - \eta(\theta)| \leq \frac{\theta|d - b|}{|b - a|} + \frac{([|b - a| - \theta)|c - a|}{|b - a|} \leq \max\{|c - a|,|d - b|\} = m, \end{displaymath} 
	and $\dot{\gamma}(\theta) = (b - a)/|b - a|$ (in particular $|\dot{\gamma}(\theta)| \equiv 1$), and
	\begin{displaymath} |\dot{\gamma}(\theta) - \dot{\eta}(\theta)| = \frac{|(b - a) - (d - c)|}{|b - a|} = \frac{\epsilon}{|b - a|}. \end{displaymath}
	Moreover, the segment $[\gamma(\theta),\eta(\theta)] \subset \bar{\Omega}$ lies in the domain of $f$, so we may estimate
	\begin{displaymath} |\nabla f(\gamma(\theta)) - \nabla f(\eta(\theta))| \leq m \cdot \|\nabla^{2}f\|_{L^{\infty}(\bar{\Omega})}. \end{displaymath}
	Plugging these bounds back into \eqref{form37}-\eqref{form38}, we deduce
	\begin{displaymath} |(f(b) - f(a)) - (f(d) - f(c))| \leq |b - a| \cdot m \cdot \|\nabla^{2}f\|_{L^{\infty}(\bar{\Omega})} + \epsilon \cdot \|\nabla f\|_{L^{\infty}(\bar{\Omega})}. \end{displaymath}
	Since $M = |b - a|$, this is what was claimed. \end{proof}

\begin{proof}[Proof of Lemma \ref{RLemma}] We start with the size estimate \eqref{size}. Note that
	\begin{displaymath} R(x,y) = U(x,y) - P(x,y) = \log|(x - y,A(x) - A(y))|^{-1} - \log|(x - y,\nabla A_{|x - y|}(x)(x - y)|^{-1}. \end{displaymath}
	Here, for $x \neq y$, 
	\begin{displaymath} A(x) - A(y) - \nabla A_{|x - y|}(x)(x - y) = \int_{x}^{y} \left( \nabla A(s) - \nabla A_{|x - y|}(x) \right) \, ds, \end{displaymath}
	and further, for $s \in [x,y]$,
	\begin{equation}\label{form45} |\nabla A(s) - \nabla A_{|x - y|}(x)| \leq \int_{B(x,|x - y|)} |\nabla A(s) - \nabla A(z)| \psi_{|x - y|}(x - z) \, dz \lesssim \delta_{2}|x- y|^{\alpha}. \end{equation}
	Evidently also $|\nabla A(s) - \nabla A_{|x - y|}(x)| \leq 2\delta_{1}$. Therefore, 
	\begin{equation}\label{form43} |A(x) - A(y) - \nabla A_{|x - y|}(x)(x - y)| \lesssim |x - y| \cdot \min\{\delta_{1},\delta_{2}|x - y|^{\alpha}\}. \end{equation}
	Now, using $\nabla_{z} \log|z|^{-1} = - z/|z|^{2}$, and noting that all points $z$ on the segment
	\begin{displaymath} I = [(x - y,A(x) - A(y)),(x - y,\nabla A_{|x - y|}(x)(x - y)] \end{displaymath}
	have norm $|z| \geq |x - y|$, we may estimate the difference $U(x,y) - P(x,y)$ by the mean value theorem:
	\begin{multline*}
		|U(x,y) - P(x,y)| \le \mathcal{H}^1(I)\cdot \sup_{z\in I} |\nabla_{z} \log|z|^{-1}|\\
		\lesssim \frac{|x -y| \cdot \min\{\delta_{1},\delta_{2}|x - y|^{\alpha}\}}{|x - y|} = \min\{\delta_{1},\delta_{2}|x - y|^{\alpha}\}.
	\end{multline*} 
	This proves \eqref{size}. 
	
	We move on to the proof of the H\"older continuity estimate \eqref{continuity}. Fix $x,x',y$ with $|x - x'| \leq \tfrac{1}{2}|x - y|$. We may write
	\begin{align*} R(x,y) - R(x',y) & = [U(x,y) - P(x,y)] - [U(x',y) - P(x',y)]\\
		&= \left( \log|a|^{-1} -\log|b|^{-1} \right) - \left( \log|c|^{-1} -\log|d|^{-1} \right), \end{align*}
	where
	\begin{displaymath} \begin{cases} a \coloneqq (x - y,A(x) - A(y)), \\ 
			b \coloneqq (x - y,\nabla A_{|x - y|}(x)(x - y)), \\
			c \coloneqq (x' - y,A(x') - A(y)),\\
			d \coloneqq (x' - y,\nabla A_{|x' - y|}(x')(x' - y)). \end{cases} \end{displaymath} 
	This invites us to apply Lemma \ref{lemma2} to the points $a,b,c,d$ listed above, and the map $f = \log|\cdot|^{-1}$. To apply Lemma \ref{lemma2}, we need good estimates for the quantities
	\begin{displaymath} m = \max\{|c - a|,|d - b|\}, \, M = \min\{|b - a|,|d - c|\}, \quad \text{and} \quad \epsilon = |(a - b) - (c - d)|. \end{displaymath}
	We claim the following bounds (below $C > 0$ is an absolute constant):
	\begin{equation}\label{form41} m \leq (1 + C\delta_{1})|x - x'| \quad \text{and} \quad M \lesssim |x - y| \cdot \min\{\delta_{1},\delta_{2}|x - y|^{\alpha}\}, \end{equation}
	and
	\begin{equation}\label{form41b} \epsilon \lesssim |x - x'| \cdot \min\{\delta_{1},\delta_{2}|x - y|^{\alpha}\}. \end{equation} 
	The bound in \eqref{form41} for $M$ follows directly from \eqref{form43}, noting that 
	\begin{displaymath} |b - a| = |A(x) - A(y) - \nabla A_{|x - y|}(x)(x - y)|. \end{displaymath}
	
	We next estimate $m$. First, note that 
	\begin{align*} |c - a| & \leq |x - x'| + |A(x) - A(y) - (A(x') - A(y))|\\
		& = |x - x'| + |A(x) - A(x')| \leq (1 + \delta_{1})|x - x'|. \end{align*}
	The difference $|d - b|$ satisfies a similar bound, but the proof is more technical. We first write
	\begin{multline*} \nabla A_{|x - y|}(x)(x - y) - \nabla A_{|x' - y|}(x')(x' - y)  = \nabla A_{|x - y|}(x)[(x - y) - (x' - y)]\\
		+ (\nabla A_{|x - y|}(x) - \nabla A_{|x' - y|}(x'))(x' - y).  \end{multline*} 
	Since $A$ is $\delta_{1}$-Lipschitz, the first term is bounded in absolute value by $\delta_{1} |x - x'|$. Since $|x' - y| \sim |x - y|$ by the hypothesis $|x' - x| \leq \tfrac{1}{2}|x - y|$, it suffices to check that
	\begin{equation}\label{form44} |\nabla A_{|x - y|}(x) - \nabla A_{|x' - y|}(x')| \lesssim \frac{|x - x'| \cdot \min\{\delta_{1},\delta_{2}|x - y|^{\alpha}\}}{|x - y|}. \end{equation}
	(In fact, only the bound $\delta_{1}|x - x'|/|x - y|$ is needed for estimating $m$, but \eqref{form44} as stated will be useful in proving \eqref{form41}.)
	Writing out the definitions, and using the fact that both $z \mapsto \psi_{|x - y|}(x - z)$ and $z \mapsto \psi_{|x' - y|}(x' - z)$ have integral one,
	\begin{align} |\nabla A_{|x - y|}(x) - \nabla A_{|x' - y|}(x')| & \leq \int |\nabla A(z) - \nabla A(y)||\psi_{|x - y|}(x - z) - \psi_{|x' - y|}(x' - z)| \, dz \notag\\
		&\label{form40} = \int |\nabla A(z) - \nabla A(y)|\left| \int_{x}^{x'} \partial_{s} \psi_{|s - y|}(s - z) \, ds \right| \, dz. \end{align} 
	Recall from \eqref{form51} that
	\begin{equation}\label{form39} |\partial_{s} \psi_{|s - y|}(s - z)| \lesssim \mathbf{1}_{B(s,|y - s|)} \cdot \frac{1}{|s - y|^{2}}. \end{equation}
	In \eqref{form40}, we have $|s - y| \sim |x - y|$ for all $s \in [x,x']$, and $B(s,|y - s|) \subset B(y,2|x - y|)$. Thus, 
	\begin{align*} |\nabla A_{|x - y|}(x) - \nabla A_{|x' - y|}(x')| & \lesssim |x - x'| \int_{B(y,2|x - y|)} \frac{|\nabla A(z) - \nabla A(y)|}{|x - y|^{2}} \, dz\\
		& \lesssim \frac{|x - x'| \cdot \min\{\delta_{1},\delta_{2}|x - y|^{\alpha}\}}{|x - y|}. \end{align*}
	This verifies \eqref{form41} for $m$.
	
	Finally, we turn to the proof of \eqref{form41b}. Write $a = (a_{1},a_{2})$ etc, and $\epsilon_{1} = |(a_{1} - b_{1}) - (c_{1} - d_{1})|$ and $\epsilon_{2} = |(a_{2} - b_{2}) - (c_{2} - d_{2})|$. In fact, $\epsilon_{1} = 0$, so it suffices to estimate $\epsilon_{2}$. Writing $A(x) - A(y)$ and $A(x') - A(y)$ in terms of $\nabla A$, we find
	\begin{align*} |(a_{2} - b_{2}) - (c_{2} - d_{2})| = \left| \int_{y}^{x} \left(\nabla  A(s) - \nabla A_{|x - y|}(x) \right) \, ds - \int_{y}^{x'} \left( \nabla A(s) - \nabla A_{|x' - y|}(x') \right) \, ds \right|. \end{align*}
	Adding and subtracting the cross term $(x' - y)\nabla A_{|x - y|}(x)$, we end up estimating separately
	\begin{displaymath} I_{1} \coloneqq \left| \int_{y}^{x} \left(\nabla A(s) - \nabla A_{|x - y|}(x) \right) \, ds - \int_{y}^{x'} \left(\nabla A(s) - \nabla A_{|x - y|}(x) \right) \, ds \right|  \end{displaymath} 
	and $I_{2} \coloneqq |y - x'||\nabla A_{|x' - y|}(x') - \nabla A_{|x - y|}(x)|$. Noting that $|y - x'| \sim |y - x|$, we already established at \eqref{form44} that $I_{2} \lesssim |x - x'| \cdot \min\{\delta_{1},\delta_{2}|x - y|^{\alpha}\}$. Finally, also
	\begin{displaymath} I_{1} \leq \int_{x}^{x'} |\nabla A(s) - \nabla A_{|x - y|}(x)| \, ds \stackrel{\eqref{form45}}{\lesssim} |x - x'| \cdot \min\{\delta_{1},\delta_{2}|x - y|^{\alpha}\},  \end{displaymath} 
	concluding the proof of \eqref{form41b}.
	
	We are now in a position to apply Lemma \ref{lemma2}. Let $\bar{\Omega} \coloneqq \mathrm{conv}(a,b,c,d)$, and recall that $f(z)=\log|z|^{-1}$. We claim that
	\begin{displaymath} \|\nabla f\|_{L^{\infty}(\bar{\Omega})} \lesssim \frac{1}{|x - y|} \quad \text{and} \quad \|\nabla^{2} f\|_{L^{\infty}(\bar{\Omega})} \lesssim \frac{1}{|x - y|^{2}}. \end{displaymath}
	Indeed, for example $|a| \geq |x - y|$, and since $|x - x'| \leq \tfrac{1}{2}|x - y|$, it follows from the estimates for $m,M$ (assuming $\delta_{1} > 0$ sufficiently small) that $\rho \coloneqq \max\{|b - a|,|c - a|,|d - a|\} \leq \tfrac{1}{2}|a|$. Since $\bar{\Omega} \subset \bar{B}(a,\rho)$, this shows that $|z| \gtrsim |x - y|$ for all $z \in \bar{\Omega}$. The bounds for $\nabla f$ and $\nabla^{2} f$ follow.  We may finally deduce from \eqref{form41}-\eqref{form41b} and Lemma \ref{lemma2} that
	\begin{align*} |R(x,y) - R(x',y)| & \lesssim m \cdot M \cdot \|\nabla^{2}f\|_{L^{\infty}(\bar{\Omega})} + \epsilon \cdot \|\nabla f\|_{L^{\infty}(\bar{\Omega})}\\
		& \lesssim \frac{|x - x'| \cdot \min\{\delta_{1},\delta_{2}|x - y|^{\alpha}\}}{|x - y|}. \end{align*} 
	This completes the proof of Lemma \ref{RLemma}. \end{proof}

\subsection{H\"older estimates for $U\mu$} In this section, $\Gamma$ refers to a $C^{1,\alpha}$-regular graph. The Lipschitz constant of $\Gamma$ is assumed so small that the conclusions of Lemmas \ref{PLemma} and \ref{RLemma} hold. The symbol $\mu$ generally refers to an element of $\mathcal{M}(\R)$ as in Theorem \ref{mainTechnical}.

We abbreviate
\begin{displaymath} \{U,P,R\} := \{U^{\Gamma},P^{\Gamma},R^{\Gamma}\}. \end{displaymath}
The goal of the section is to prove that if $\mu$ is restricted to any interval $I \subset [0,1]$ where $\mu$ satisfies a $1$-dimensional Frostman condition at both endpoints, then $U(\mu|_{I})$ is H\"older continuous on $\R$. The H\"older continuity exponent depends on $\alpha$ (the H\"older continuity exponent of the graph $\Gamma$) and the Frostman exponent of $\mu|_{I}$. 

In fact, if the Frostman exponent of $\mu$ is sufficiently close to one, then $U(\mu|_{I})$ is Lipschitz continuous outside endpoints of $I$, and the "blow-up" of the Lipschitz constants at the endpoints is so mild that $\partial U(\mu|_{I}) \in L^{p}(\R)$ for $p \in [1,\infty)$. These results are formalised in Corollaries \ref{cor11} and \ref{cor12}.

The potential $U\mu$ has an inconveniently heavy "tail". Since the H\"older and Lipschitz continuity questions discussed here are anyway local, we prefer to cut off this tail immediately, and concentrate on the "local" part. To accomplish this, we use a tool which is already available to us -- Lemma \ref{truncationLemma} -- and write
\begin{displaymath} \log|x|^{-1} = \log_{1}(x) + \log_{2}(x). \end{displaymath} 
Here $\log_{1}$ is the truncation of the kernel $x \mapsto \log |x|^{-1}$ defined in Lemma \ref{truncationLemma} with parameter $\epsilon = 1$, see also Remark \ref{rem2}. Repeating Remark \ref{rem2}, we know that $\log_{1} \in C^{4}(\R \, \setminus \, \{0\})\cap\lip(\R)$, $\log_{1}(x) = \log |x|^{-1}$ for $|x| \geq 1$, and $\|\log_{1}\|_{L^{\infty}([-1,1]) }\lesssim 1$. We then set 
\begin{equation}\label{form133} U_{j}\nu(x) := \int \log_{j} |\Gamma(x) - \Gamma(y)| \, d\nu(y), \qquad x \in \R, \, j \in \{1,2\}, \end{equation}
whenever $\nu \in \mathcal{M}([-1,1])$ (see Definition \ref{def5Intro}). In fact, $U_{1}\nu(x)$ is pointwise well-defined whenever $\nu$ is a finite measure with compact support. In a typical application, $\nu$ is the restriction of the equilibrium measure to some interval. The moral is that $U_{1}\nu$ is the "global" part of $U\nu$ with a (possibly) heavy tail, but good local behaviour. In contrast, the local behaviour (H\"older and Lipschitz continuity of $U\nu$) is determined by $U_{2}\nu$. On the bright side, $U_{2}\nu$ has no tail at all, since $\spt \log_{2} \subset [-1,1]$, which implies $\spt U_{2}\nu \subset [-2,2]$.

We will use the following elementary proposition to deal with $U_{1}\nu$:
\begin{proposition}\label{prop:lap-potential-Lp} Let $\nu$ be a Borel measure with $\spt \nu \subset [-1,1]$ and $\|\nu\| \leq 1$, and let $\Gamma(x) = (x,A(x))$, $x \in \R$, where $A \colon \R \to \R^{d-1}$ is Lipschitz. For $\alpha \in (0,1]$, let $\Delta^{\alpha/2}U_{1}\nu$ be the tempered distribution defined by
	\begin{displaymath} \int \Delta^{\alpha/2}U_{1}\nu \cdot g \coloneqq \int U_{1}\nu \cdot \Delta^{\alpha/2}g, \qquad g \in \mathcal{S}(\R). \end{displaymath} 
	(The right hand side is well-defined by Lemma \ref{lemma3}, and since $U_{1}\nu$ has logarithmic growth.) Then $U_1\nu\in \mathrm{Lip}(\R)$ and $\Delta^{\alpha/2}U_{1}\nu \in L^{p}(\R)$ for every $p \in (\alpha^{-1},\infty)$, with $\|\Delta^{\alpha/2}U_{1}\nu\|_{L^{p}} \lesssim_{p,\lip(A)} 1$.
\end{proposition} 

\begin{proof} Note that $U_{1}\nu \in \mathrm{Lip}(\R)$ with derivative
	\begin{displaymath} \partial U_{1}\nu(x) = \int k_{1} (|\Gamma(x) - \Gamma(y)|) \cdot \frac{(x - y) + \nabla A(x)\cdot (A(x) - A(y))}{|\Gamma(x) - \Gamma(y)|} \, d\nu(y). \end{displaymath}
	Here $k_{1} = \partial \log_{1}$. Since $k_{1} \in L^{\infty}(\R)$ has no singularity, and $\nu$ is a finite measure, the integral on the right is absolutely convergent, and indeed $\|\partial U_{1}\nu\|_{L^{\infty}} \lesssim 1$. From the decay $|k_{1}(x)| \lesssim |x|^{-1}$, combined with $\spt \nu \subset [-1,1]$ and $\|\nu\| \le 1$, we also obtain
	\begin{equation}\label{form84} |\partial U_{1}\nu(x)| \lesssim_{\lip(A)} \min\{1,|x|^{-1}\}, \qquad x \in \R. \end{equation}
	Consequently $\|\partial U_{1}\nu\|_{L^{q}} \lesssim_{\lip(A),p} 1$ for all $q \in (1,\infty]$.
	
	We now fix $\alpha^{-1} < p < \infty$, and start verifying that $\|\Delta^{\alpha/2}U_{1}\nu\|_{L^{p}} \lesssim_{\mathrm{Lip}(A),p} 1$. Fix $g \in \mathcal{S}(\R)$ with $\|g\|_{L^{p'}} \leq 1$. Note that $1 < p' < 1/(1 - \alpha)$. We integrate by parts:
	\begin{displaymath} \int \Delta^{\alpha/2}U_{1}\nu \cdot g \stackrel{\mathrm{def.}}{=} \int U_{1}\nu \cdot \Delta^{\alpha/2}g = \Big/_{-\infty}^{\infty}  U_{1}\nu \cdot \partial^{-1}\Delta^{\alpha/2}g - \int \partial U_{1}\nu \cdot \partial^{-1}\Delta^{\alpha/2}g,  \end{displaymath} 
	where $\partial^{-1}\Delta^{\alpha/2}g \in L^{\infty}(\R)$ stands for
	\begin{displaymath} \partial^{-1}\Delta^{\alpha/2}g(x) := \int_{-\infty}^{x} \Delta^{\alpha/2}g(y) \, dy, \qquad x \in \R. \end{displaymath}
	The integration by parts is legitimate, since $U_{1}\nu$ is Lipschitz, and $\Delta^{\alpha/2}g \in L^{1}(\R)$. Further,
	\begin{displaymath} \Big/_{-\infty}^{\infty}  U_{1}\nu \cdot \partial^{-1}\Delta^{\alpha/2}g = 0, \end{displaymath}
	because $|\partial^{-1}\Delta^{\alpha/2}g(x)| \lesssim_{\alpha,g} (1 + |x|)^{-\alpha}$. This follows from the zero mean of $\Delta^{\alpha/2}g$, and the decay $|\Delta^{\alpha/2}g| \lesssim_{\alpha,g} (1 + |x|)^{-1 - \alpha}$ shown in Lemma \ref{lemma3}. 
	
	Therefore, using  H\"older's inequality with exponent $q$, whose dual exponent $q'$ equals
	\begin{displaymath} q' \coloneqq \frac{p'}{1 - (1 - \alpha)p'} \in [p',\infty) \subset (1,\infty), \end{displaymath}
	we find
	\begin{displaymath} \left| \int \Delta^{\alpha/2}U_{1}\nu \cdot g \right| \leq \|\partial U_{1}\nu\|_{L^{q}}\|\partial^{-1}\Delta^{\alpha/2}g\|_{L^{q'}} \lesssim_{\lip(A),q'} \|\Delta^{(\alpha - 1)/2}g\|_{L^{q'}}. \end{displaymath} 
	In the final inequality, we also used that the Fourier transforms of $\partial^{-1}\Delta^{\alpha/2}g$ and $\Delta^{(\alpha - 1)/2}g$ only differ by $\sgn(\xi)$, so the two functions have comparable $L^{q'}$-norms by Corollary \ref{cor3} (note that $\sgn(\xi) = n_{0}(\xi)$). Moreover, by the Hardy-Littlewood-Sobolev theorem \cite[Chapter V \S1.2]{MR290095}
	\begin{displaymath} \|\Delta^{(\alpha - 1)/2}g\|_{L^{q'}} \lesssim \|g\|_{L^{p'}} \lesssim 1. \end{displaymath}
	This completes the proof. \end{proof}

We then turn our attention to the "local" part $U_{2}\mu$. Here is the main result:
\begin{proposition}\label{prop17a} Let $\alpha \in (0,1]$. Let $\Gamma(x) = (x,A(x))$, where $A \in \mathrm{Lip}(\R;\R^{d-1}) \cap C^{1,\alpha}(\R;\R^{d-1})$, and $\mathrm{Lip}(A)$ is so small that the conclusions of Lemmas \ref{PLemma} and \ref{RLemma} hold.
	
	Assume that $\mu \in \mathcal{M}(\R)$ satisfies the hypotheses of Theorem \ref{mainTechnical} on some compact interval $I_{0} \subset \R$. Thus, for some Lipschitz function $\mathcal{L} \colon \R \to \R$,
	\begin{itemize}
		\item $U^{\Gamma}\mu(x) = \mathcal{L}(x)$ for all $x \in \spt \mu \cap I_{0}$,
		\item $U^{\Gamma}\mu(x) \geq \mathcal{L}(x)$ for all $x \in I_{0}$.
	\end{itemize}
	Let further $I = [x_{0},y_{0}] \subset I_{0}$ be an interval such that
	\begin{displaymath} \mu(B(z_{0},r)) \lesssim r, \qquad z_{0} \in \{x_{0},y_{0}\}, \, r > 0. \end{displaymath}
	Assume finally that $\mu(B(x,r)) \lesssim r^{\gamma_{0}}$ for all $x \in I$ and $r > 0$, where $\gamma_{0} \in [0,1]$.
	
	Then $U_{2}(\mu|_{I})$ is $\zeta$-H\"older continuous on $\R$ for all $\zeta \in [0,\min\{1,\alpha + \gamma_{0}\})$. Furthermore, if $\alpha + \gamma_{0} > 1$, then the distributional derivative of $U_{2}(\mu|_{I})$ lies in $L^{p}(\R)$ for every $p \in [1,\infty)$.
\end{proposition} 	

The two conclusions (about H\"older continuity and $L^{p}$-derivatives) are proved separately in Corollaries \ref{cor11} and \ref{cor12}. Before that, we need a few auxiliary results. We first use Lemma \ref{RLemma} to prove the H\"older continuity of the potential $R\nu(x)$. Recall that for $\gamma \in (0,1]$ the $\gamma$-dimensional Riesz kernel is denoted by $k_\gamma(x) = |x|^{-\gamma}$.

\begin{lemma}\label{lemma10} Assume that $\Gamma(x) = (x,A(x))$, where $A \in \mathrm{Lip}(\R;\R^{d-1})$ and the Lipschitz constant of $A$ is so small that the conclusion of Lemma \ref{RLemma} holds. Let $\nu$ be a finite Radon measure on $\R$. Then	
	\begin{equation}\label{eq:holder0}
		|R\nu(x)-R\nu(y)|\lesssim |x-y|^\alpha \nu(\R).
	\end{equation}
	Moreover, for any $\gamma \in (0,1-\alpha]$ and $x\neq y\in \R$
	\begin{equation}\label{eq:holder}
		|R\nu(x)-R\nu(y)|\lesssim |x-y|^{\alpha+\gamma}\left(k_\gamma\ast \nu(x)+k_\gamma\ast\nu(y)\right).
	\end{equation}
	In particular, if $\nu(B(x,r))\lesssim r^{\gamma_0}$ for all $x \in \R$, $r \in (0,1)$, and some $\gamma_0>\gamma$, then $R\nu$ is $(\alpha + \gamma)$-H\"older continuous in $\R$.
\end{lemma}
\begin{proof}
	We estimate
	\begin{multline}\label{eq:decomp}
		|R\nu(x)-R\nu(y)|=\left|\int (R(x,z)-R(y,z))\, d\nu(z)\right|\\
		\le\left|\int_{|x-z|\ge 2|x-y|} (R(x,z)-R(y,z))\, d\nu(z)\right| + \left|\int_{|x-z|\le 2|x-y|} (R(x,z)-R(y,z))\, d\nu(z)\right|.
	\end{multline}
	To bound the first term we use \eqref{continuity}
	\begin{multline*}
		\left|\int_{|x-z|\ge 2|x-y|} (R(x,z)-R(y,z))\, d\nu(z)\right|\lesssim \int_{|x-z|\ge 2|x-y|} \frac{|x-y|}{|x-z|^{1-\alpha}}\, d\nu(z)\\
		\lesssim\int_{|x-z|\ge 2|x-y|} |x-y|^{\alpha}\, d\nu(z)\le |x-y|^{\alpha}\nu(\R),
	\end{multline*}
	which corresponds to \eqref{eq:holder0}, and
	\begin{multline*}
		\left|\int_{|x-z|\ge 2|x-y|} (R(x,z)-R(y,z))\, d\nu(z)\right|\lesssim \int_{|x-z|\ge 2|x-y|} \frac{|x-y|}{|x-z|^{1-\alpha}}\, d\nu(z)\\
		= |x-y|\int_{|x-z|\ge 2|x-y|} \frac{|x-z|^{\alpha+\gamma-1}}{|x-z|^{\gamma}}\, d\nu(z)\lesssim |x-y|^{\alpha+\gamma} k_\gamma\ast\nu(x),
	\end{multline*}
	which corresponds to \eqref{eq:holder}. 
	
	To estimate the second term from \eqref{eq:decomp} we use \eqref{size}. We get
	\begin{multline*}
		\left|\int_{|x-z|\le 2|x-y|} (R(x,z)-R(y,z))\, d\nu(z)\right|\lesssim \int_{|x-z|\le 2|x-y|} |x-z|^{\alpha}+{|y-z|^{\alpha}}\, d\nu(z)\\
		\lesssim |x-y|^{\alpha}\nu(\R)
	\end{multline*}
	and
	\begin{multline*}
		\left|\int_{|x-z|\le 2|x-y|} (R(x,z)-R(y,z))\, d\nu(z)\right|\lesssim \int_{|x-z|\le 2|x-y|} |x-z|^{\alpha}+{|y-z|^{\alpha}}\, d\nu(z)\\
		=\int_{|x-z|\le 2|x-y|} \frac{|x-z|^{\alpha+\gamma}}{|x-z|^{\gamma}}+\frac{|y-z|^{\alpha+\gamma}}{|y-z|^{\gamma}}\, d\nu(z)
		\lesssim |x-y|^{\alpha+\gamma}(k_\gamma\ast\nu(x)+k_\gamma\ast\nu(y)).
	\end{multline*}
	Collecting the estimates gives \eqref{eq:holder0} and \eqref{eq:holder}.
\end{proof}

To prove the H\"older continuity of $U\mu$ (as claimed in Proposition \ref{prop17a}), the main idea is to use the relation $U\mu = P\mu + R\mu$, the convexity of $P\mu$, and the H\"older continuity of $R\mu$ (as given by Lemma \ref{lemma10}). To combine the information, we state the following "abstract" proposition, which does not use the special form of the functions $U\mu,P\mu$ or $R\mu$:

\begin{proposition}\label{prop16} Let $I = [a,b] \subset \R$ be an interval. Let $U,P,R \colon I \to \R$ be functions with the following properties:
	\begin{enumerate}
		\item $U = P + R$,
		\item $P$ is convex,
		\item $R$ is $\alpha$-H\"older continuous with constant $C$. 
	\end{enumerate}
	Then,
	\begin{equation}\label{form124} U(x) - U(a) \leq (x - a) \frac{|U(b) - U(a)|}{b - a} + 2C|x - a|^{\alpha}, \end{equation}
	and
	\begin{equation}\label{form125} U(x) - U(b) \leq (b - x)\frac{|U(b) - U(a)|}{b - a} + 2C |b - x|^{\alpha}. \end{equation}
\end{proposition}

\begin{remark} Proposition \ref{prop16} claims no lower bounds for $U(x) - U(a)$ and $U(x) - U(b)$, and indeed (comparable) lower bounds would be false under the stated hypotheses. \end{remark}

\begin{proof}[Proof of Proposition \ref{prop16}] Write
	\begin{displaymath} U((1 - \theta)a + \theta b) = P((1 - \theta)a + \theta b) + R((1 - \theta)a + \theta b), \qquad \theta \in [0,1]. \end{displaymath}
	The convexity of $P$, and $P(a) = U(a) - R(a)$ and $P(b) = U(b) - R(b)$, next give the relation
	\begin{displaymath} P((1 - \theta)a + \theta b) \leq (1 - \theta)P(a) + \theta P(b) = [(1 - \theta)U(a) + \theta U(b)] - (1 -\theta)R(a) - \theta R(b). \end{displaymath}
	Thus,
	\begin{align*} U((1 - \theta)a + \theta b) & \leq [(1 - \theta)U(a) + \theta U(b)]\\
		& \quad + R((1 - \theta)a + \theta b) - (1 - \theta)R(a) - \theta R(b), \qquad \theta \in [0,1]. \end{align*}
	This yields
	\begin{align} U((1 - \theta)a + \theta b) - U(a) & \leq \theta |U(b) - U(a)| \notag\\
		&\label{form122} \quad + |R((1 - \theta)a + \theta b) - (1 - \theta)R(a) - \theta R(b)|, \end{align}
	and similarly
	\begin{align} U((1 - \theta)a + \theta b) - U(b) & \leq (1 - \theta)|U(a) - U(b)| \notag\\
		&\label{form123} \quad + |R((1 - \theta)a + \theta b) - (1 - \theta)R(a) - \theta R(b)|. \end{align}
	Note that for $x = (1 - \theta)a + \theta b \in I$, 
	\begin{displaymath} \theta = \frac{x - a}{b - a} \quad \text{and} \quad 1 - \theta = \frac{b - x}{b - a}. \end{displaymath}
	Therefore, the H\"older continuity of $R$ yields (after adding and subtracting $(1 - \theta)R((1 - \theta)a + \theta b)$),
	\begin{align*} & |R((1 - \theta)a + \theta b) - (1 - \theta)R(a) - \theta R(b)|\\
		& \quad \leq (1 - \theta)|R((1 - \theta)a + \theta b) - R(a)| + \theta|R((1 - \theta)a + \theta b) - R(b)|\\ 
		& \quad \leq C(1 - \theta)|\theta(b - a)|^{\alpha} + C\theta |(1 - \theta)(b - a)|^{\alpha}\\
		& \quad = C\frac{b - x}{b - a}(x - a)^{\alpha} + C\frac{x - a}{b - a}(b - x)^{\alpha}. \end{align*}
	Substituting these bounds to \eqref{form122}-\eqref{form123} yields
	\begin{displaymath} U(x) - U(a) \leq \frac{x - a}{b - a}|U(b) - U(a)| + C\frac{b - x}{b - a}(x - a)^{\alpha} + C\frac{x - a}{b - a}(b - x)^{\alpha}, \end{displaymath} 
	and similarly
	\begin{displaymath} U(x) - U(b) \leq \frac{b - x}{b - a}|U(a) - U(b)| + C\frac{b - x}{b - a}(x - a)^{\alpha} + C\frac{x - a}{b - a}(b - x)^{\alpha}.  \end{displaymath} 
	This is what was claimed. These bounds yield \eqref{form124}-\eqref{form125} immediately. To derive \eqref{form124}, estimate $b - x \leq b - a$ and $x - a \leq (x - a)^{\alpha}(b - a)^{1 - \alpha}$. To derive \eqref{form125}, estimate instead $x - a \leq b - a$ and $b - x \leq (b - x)^{\alpha}(b - a)^{1 - \alpha}$. The proof is complete. \end{proof}

The next "localisation lemma" shows that if $\nu$ is a finite measure on $[0,1]$ which vanishes outside some interval $I$, and has one-dimensional decay at the endpoints of $I$, then $U_{2}\nu$ is locally Lipschitz continuous on the interior of $I$, and the derivative of $U_{2}\nu$ only has a logarithmic singularity at $\partial I$. This lemma is needed, because in Proposition \ref{prop17a} we claim something about the logarithmic potentials of restrictions $\mu|_{I}$. The idea is that $\nu = \mu - \mu|_{I}$ satisfies the hypotheses of Lemma \ref{prop14}, and therefore $U_{2}\nu$ is "almost" Lipschitz continuous on $I$. In short, the H\"older regularity of $U_2(\mu|_{I})$ does not essentially depend on the behaviour of $\mu$ outside $I$ (provided that \eqref{form116} holds).

\begin{lemma}[Localisation lemma]\label{prop14} Let $\mathbf{C} > 0$, and let $\nu$ be a finite Borel measure on $[0,1]$. Let $[x_{0},y_{0}] \subset \R$ be an interval such that $\spt \nu \cap (x_{0},y_{0}) = \emptyset$, and
	\begin{equation}\label{form116} \nu(B(z_{0},r)) \leq \mathbf{C}r \text{ for } z_{0} \in \{x_{0},y_{0}\} \text{ and }r > 0. \quad \end{equation}
	We also allow the cases $x_{0} = -\infty$ or $y_{0} = \infty$ with the interpretation that the requirement \eqref{form116} is only posed for the finite endpoint of the interval. Then, $U_{2}\nu \in \mathrm{Lip}_{\mathrm{loc}}(x_{0},y_{0})$, and
	\begin{equation}\label{form136} |\partial U_{2}\nu(x)| \lesssim \mathbf{C}(1 + \mathrm{Lip}(A)) \left( \log \tfrac{100}{x - x_{0}} + \log \tfrac{100}{y_{0} - x} \right) \mathbf{1}_{[-2,2]}(x) , \qquad \text{a.e. } x \in (x_{0},y_{0}). \end{equation} 
\end{lemma} 

\begin{proof} Since $(x_{0},y_{0}) \cap \spt \nu = \emptyset$, the following differentiation under the integral sign is easily justified whenever $\log_{2} |\Gamma(\cdot) - \Gamma(y)|$ is differentiable at $x \in (x_{0},y_{0})$ (which is true for a.e. $x \in (x_{0},y_{0})$):
	\begin{displaymath} \partial U_{2}\nu(x) = \int \partial_{x} (x \mapsto \log_{2} |\Gamma(x) - \Gamma(y)|) \, d\nu(y). \end{displaymath} 
	To proceed, we note that
	\begin{displaymath} |\partial_{x}(x \mapsto \log_{2} |\Gamma(x) - \Gamma(y)|)| \lesssim (1 + \mathrm{Lip}(A)) \frac{\mathbf{1}_{B(x,1)}(y)}{|x - y|},  \end{displaymath} 
	which follows readily (or see \eqref{form134}) by recalling that $\log_{2}(x) = \log |x|^{-1} - \log_{1}(x)$, and that $\partial \log_{1} = k_{1}$ is the standard kernel from Lemma \ref{truncationLemma}. Therefore,
	\begin{displaymath} |\partial U_{2}\nu(x)| \lesssim \int_{[0,x_{0}] \cap B(x,1)} \frac{d\nu(y)}{|x - y|} + \int_{[y_{0},1] \cap B(x,1)} \frac{d\nu(y)}{|x - y|} =: I_{[0,x_{0}]}(x) + I_{[y_{0},1]}(x), \end{displaymath}
	where the implicit constants hide the factor $1 + \mathrm{Lip}(A)$. The estimates for $I_{[0,x_{0}]}(x)$ and $I_{[y_{0},1]}(x)$ are similar, so we only include the details for $I_{[0,x_{0}]}(x)$. First, if $|x| > 2$, then $B(x,1) \cap \spt \nu = \emptyset$, and therefore 
	\begin{displaymath} |\partial U_{2}\nu(x)| = 0, \qquad |x| > 2. \end{displaymath}
	Let us then assume that $|x| \leq 2$. In this case we may perform the following estimate by splitting into dyadic annuli around $x$:
	\begin{equation}\label{form135} \int_{0}^{x_{0}} \frac{d\nu(y)}{|x - y|} \leq 2 \sum_{j = -2}^{\infty} \nu([0,x_{0}] \cap B(x,2^{-j})) \cdot 2^{j}. \end{equation} 
	(The hypothesis $|x| \leq 2$ was used to ensure that $\spt \nu \cap B(x,2^{-j}) \, \setminus \, B(x,2^{-j - 1}) = \emptyset$ for $j \leq -3$.) In \eqref{form135}, note that $\nu([0,x_{0}] \cap B(x,2^{-j})) = 0$ for all $2^{-j} \leq x - x_{0}$. On the other hand, for $2^{-j} \geq x - x_{0}$, 
	\begin{displaymath} \nu(B(x,2^{-j})) \leq \nu(B(x_{0},2^{-j + 1})) \leq \mathbf{C}2^{-j + 1} \end{displaymath}
	by the hypothesis \eqref{form116}. Therefore,
	\begin{displaymath} \int_{0}^{x_{0}} \frac{d\nu(y)}{|x - y|} \leq 4\mathbf{C} \cdot |\{j \geq -2 : 2^{-j} \geq x - x_{0}\}| \lesssim \mathbf{C} \log \frac{100}{x - x_{0}}, \quad x \in (x_{0},y_{0}) \cap [-2,2]. \end{displaymath} 
	This (combined with a similar estimate for $I_{[y_{0},1]}(x)$) proves \eqref{form136}.  \end{proof} 


We are prepared to prove Proposition \ref{prop17a}. We prove a slightly more technical result, below, and then derive the conclusions of Proposition \ref{prop17a} in Corollaries \ref{cor11} and \ref{cor12}.

\begin{proposition}\label{prop17} Let $\mathbf{C} \geq 1$ and $\alpha \in (0,1]$. Let $\Gamma(x) = (x,A(x))$, where $A \in \mathrm{Lip}(\R;\R^{d-1}) \cap C^{1,\alpha}(\R;\R^{d-1})$, and $\mathrm{Lip}(A)$ is so small that the conclusions of Lemmas \ref{PLemma} and \ref{RLemma} hold.
	
	Assume that $\mu \in \mathcal{M}(\R)$ satisfies the hypotheses of Theorem \ref{mainTechnical} on some compact interval $I_{0} \subset \R$. Thus, for some Lipschitz function $\mathcal{L} \colon \R \to \R$,
	\begin{itemize}
		\item $U^{\Gamma}\mu(x) = \mathcal{L}(x)$ for all $x \in \spt \mu \cap I_{0}$,
		\item $U^{\Gamma}\mu(x) \geq \mathcal{L}(x)$ for all $x \in I_{0}$.
	\end{itemize}
	Let further $I = [x_{0},y_{0}] \subset I_{0}$ be an interval such that
	\begin{equation}\label{form121} \mu(B(z_{0},r)) \leq \mathbf{C}r, \qquad z_{0} \in \{x_{0},y_{0}\}, \, r > 0. \end{equation}
	Assume finally that $\mu(B(x,r)) \leq \mathbf{C}r^{\gamma_{0}}$ for all $x \in I$ and $r > 0$, where $\gamma_{0} \in [0,1]$. Then 
	\begin{equation}\label{form137} U_{2}(\mu|_{I}) = H + L, \end{equation}
	where 
	\begin{itemize}
		\item[(a)] $H$ is $\min\{1,\alpha + \gamma\}$-H\"older continuous on $I$ for every $-\alpha \leq \gamma < \gamma_{0}$, and
		\item[(b)] $L \in \mathrm{Lip}_{\mathrm{loc}}(\R \, \setminus \{x_{0},y_{0}\})$, and 
		\begin{displaymath} |\partial L(x)| \lesssim \mathbf{C}(1 + \mathrm{Lip}(A))\left(\log \tfrac{100}{|x - x_{0}|} + \log \tfrac{100}{|x - y_{0}|} \right)\mathbf{1}_{[-2,2]}(x), \qquad \text{a.e. } x \in \R \, \setminus \, \{x_{0},y_{0}\}. \end{displaymath}
	\end{itemize}
	In particular, if $\alpha + \gamma_{0} > 1$, then $H \in \mathrm{Lip}(I)$.
\end{proposition} 

\begin{remark}\label{rem4} The H\"older (or Lipschitz) continuity constant of $H$ depends on $\mathbf{C},\gamma,\mathrm{Lip}(\mathcal{L})$, and the length of $I$, but there are additional dependencies, explained below: 
	\begin{itemize}
		\item If $\spt \mu \cap I = \emptyset$, then the constant may depend on $\dist(\spt \mu,I)$ and $\mu(I)$.
		\item If $\{x_{0},y_{0}\} \subset \spt \mu$, then the constant depends only on $\mathbf{C},\gamma$, and $\mathrm{Lip}(\mathcal{L})$.
		\item If $\dist(x_{0},\spt \mu) > 0$, the constant may depend on this distance, and $\mu(I),\|U\|_{L^{\infty}(I)}$.
		\item If $\dist(y_{0},\spt \mu) > 0$, the constant may depend on this distance, and $\mu(I),\|U\|_{L^{\infty}(I)}$
\end{itemize} \end{remark} 

We record two useful (and simpler-to-read) corollaries before the proof:

\begin{cor}\label{cor11} Under the hypotheses of Proposition \ref{prop17}, $U_{2}(\mu|_{I})$ is $\zeta$-H\"older continuous on $\R$ for all $\zeta \in [0,\min\{1,\alpha + \gamma_{0}\})$.	
\end{cor}

\begin{remark} A weakness of Corollary \ref{cor11} is that even in the case $\alpha + \gamma_{0} > 1$, it does not conclude that $U_{2}(\mu|_{I}) \in \mathrm{Lip}(\R)$. A slightly weaker (but sufficient) conclusion will be contained in Corollary \ref{cor12} below.  \end{remark} 

\begin{proof}[Proof of Corollary \ref{cor11}] Fix $\zeta \in [0,\min\{1,\alpha + \gamma_{0}\})$. Then $\zeta$ may be expressed as $\zeta = \min\{1,\alpha + \gamma\}$ for some $-\alpha \leq \gamma < \gamma_{0}$ (namely $\gamma = \zeta - \alpha$), so Proposition \ref{prop17}(a) tells us that $H$ is $\zeta$-H\"older continuous on $I$. From the estimates for $\partial L$ in Proposition \ref{prop17}(b), we may also easily deduce that $L$ is $\zeta$-H\"older continuous on $\R \, \setminus \, \{x_{0},y_{0}\}$, therefore $U_{2}(\mu|_{I}) = H + L$ is $\zeta$-H\"older continuous on $\R \, \setminus \, \{x_{0},y_{0}\}$.	

On the other hand, $U_{2}(\mu|_{I}) \in C(\R)$, since $\mu|_{I} \in \mathcal{M}(\R)$ by Lemma \ref{lemma12}. Now it suffices to note, in general, that if $f \in C(\R)$ is $\zeta$-H\"older on a dense subset of $\R$, then in fact $f$ is $\zeta$-H\"older continuous on $\R$.  \end{proof}

\begin{cor}\label{cor12} Under the hypotheses of Proposition \ref{prop17}, assume moreover that $\alpha + \gamma_{0} > 1$. Then the distributional derivative of $U_{2}(\mu|_{I})$ lies in $L^{p}(\R)$ for every $p \in [1,\infty)$.
\end{cor}

\begin{proof} Since $\alpha + \gamma_{0} > 1$, we deduce from Proposition \ref{prop17}(a) that $H \in \mathrm{Lip}(I)$. Therefore the decomposition \eqref{form137} reveals that $U_{2}(\mu|_{I}) \in \mathrm{Lip}_{\mathrm{loc}}(\R \, \setminus \, \{x_{0},y_{0}\})$, and the derivative (which exists for a.e. $x \in \R \, \setminus \, \{x_{0},y_{0}\}$) is bounded by 
	\begin{displaymath} |\partial U_{2}(\mu|_{I})(x)| \leq |\partial L(x)| + \mathrm{Lip}(H). \end{displaymath}
	Here $\mathrm{Lip}(H) < \infty$ (although the bounds are a little complicated, see Remark \ref{rem4}), and $|\partial L(x)|$ enjoys the upper bound stated in Proposition \ref{prop17}(b). Moreover, we know that $U_{2}(\mu|_{I}) \in C(\R)$, since $\mu|_{I} \in \mathcal{M}(\R)$ by Lemma \ref{lemma12}, and $U_{2} = U - U_{1}$, 
	
	Now, we note, in general, that if the following two conditions hold, then the distributional derivative agrees with the classical derivative:
	\begin{itemize}
		\item $f \in \mathrm{Lip}_{\mathrm{loc}}(\R \, \setminus \, \{x_{0},y_{0}\}) \cap C(\R)$
		\item $f' \in L^{1}(\R \, \setminus \, \{x_{0},y_{0}\})$
	\end{itemize}
	Let us justify this in the case where the "exceptional" set only contains one point $x_{0} = 0$ (which contains all the ideas). Fixing $g \in \mathcal{C}_{c}^{\infty}(\R)$, we decompose
	\begin{align*} \int f \cdot g' = \lim_{\epsilon \to 0} \left( \int_{-\infty}^{- \epsilon} f \cdot g' + \int_{-\epsilon}^{\epsilon} f \cdot g' + \int_{\epsilon}^{\infty} f \cdot g'\right) \end{align*} 
	The middle term tends to zero since $f$ is bounded on $[-1,1]$. For the other two terms, we use the Lipschitz continuity of $f$ on $\spt g$ to integrate by parts, for example
	\begin{displaymath} \int_{-\infty}^{-\epsilon} f \cdot g' = \Big/_{-\infty}^{-\epsilon} f \cdot g -\int_{-\infty}^{-\epsilon} f' \cdot g = f(-\epsilon)g(-\epsilon) - \int_{-\infty}^{-\epsilon} f' \cdot g. \end{displaymath}
	After doing this, note that
	\begin{displaymath} \lim_{\epsilon \to 0} f(-\epsilon)g(-\epsilon) - f(\epsilon)g(\epsilon) = 0 \end{displaymath}
	by the continuity of $f$, and on the other hand
	\begin{displaymath} \lim_{\epsilon \to 0} -\int_{(-\infty,-\epsilon) \cup (\epsilon,\infty)} f' \cdot g = \int_{\R} f' \cdot g \end{displaymath}
	by they hypothesis $f' \in L^{1}(\R)$. This shows that $\int f \cdot g' = -\int f' \cdot g$, as required.
	
	Now that we know that the distributional derivative agrees with the classical derivative of $U_{2}(\mu|_{I})$, the rest is easy: the bounds in Proposition \ref{prop17}(b) (together with $H \in \mathrm{Lip}(I)$) clearly imply that $\partial U_{2}(\mu|_{I}) \in L^{p}(\R)$ for every $p \in [1,\infty)$.  \end{proof} 

We then turn to the proof of Proposition \ref{prop17}.

\begin{proof}[Proof of Proposition \ref{prop17}] Fix $\gamma \in [-\alpha,\gamma_{0})$. In this proof, the "$\lesssim$" notation is allowed to hide constants which depend on $\mathbf{C}$, $\gamma < \gamma_{0}$, $\mathrm{Lip}(\mathcal{L}),I,\spt \mu$, and $\mu(I)$ in the ways described in Remark \ref{rem4}. Abbreviate $\bar{\mu} := \mu|_{I}$ and $U := U^{\Gamma}$. We may assume without loss of generality that 
	\begin{equation}\label{form130} I \cap \spt \mu \neq \emptyset. \end{equation}
	Otherwise $U_{2}(\mu|_{I})$ is Lipschitz continuous with constant depending on $\dist(\spt \mu,I)$, so we may simply define $H := U_{2}(\mu|_{I})$ and $L := 0$.

	In general, the decomposition alluded to in \eqref{form137} is given by $U_{2}\bar{\mu} = H + L$, where 
	\begin{displaymath} H = (U_{2}\mu)\mathbf{1}_{I} \quad \text{and} \quad L := - U_{2}(\mu - \bar{\mu})\mathbf{1}_{I}  + (U_{2}\bar{\mu})\mathbf{1}_{\R \, \setminus \, I}. \end{displaymath}
	Since 
	\begin{displaymath} \spt (\mu - \bar{\mu}) \cap \mathrm{int\,} I = \emptyset = \spt \bar{\mu} \cap (\R \, \setminus \, I), \end{displaymath}
	the properties of $L$ stated in (b) follow immediately from Lemma \ref{prop14}.
	
	We next claim that $H$ is $\min\{1,\alpha + \gamma\}$-H\"older continuous on $I$, as in Proposition \ref{prop17}(a). Since $U_{1}\mu$ is Lipschitz continuous by Proposition \ref{prop:lap-potential-Lp}, it suffices to show that
	\begin{displaymath} U\mu \text{ is $\min\{1,\alpha + \gamma\}$-H\"older continuous on $I$.} \end{displaymath}	
	Let $x,y \in I$. We claim that $|U\mu(x) - U\mu(y)| \lesssim |x - y|^{\min\{1,\alpha + \gamma\}}$. There are three cases:
	\begin{enumerate}
		\item $x,y \in \spt \mu \cap I \subset \spt \mu \cap I_{0}$, or
		\item $x,y$ lie in the closure of a single connected component of $I \, \setminus \, \spt \mu$, or
		\item $x,y$ lie in the closures of distinct connected components of $I \, \setminus \, \spt \mu$.
	\end{enumerate}	
	
	The case (1) follows from the hypothesis that $U\mu$ agrees with $\mathcal{L}$ on $\spt \mu \cap I_{0}$:
	\begin{equation}\label{form126} |U\mu(x) - U\mu(y)| = |\mathcal{L}(x) - \mathcal{L}(y)| \lesssim |x - y| \lesssim |x - y|^{\min\{1,\alpha + \gamma\}}. \end{equation}
	Let us move to case (2). Assume that $x,y \in [a,b]$, where $(a,b) \subset I \, \setminus \, \spt \mu$ is a component of $I \, \setminus \, \spt \mu$. Let us, additionally, assume for the moment that $a,b \in \spt \mu \cap I$. The opposite case is where either $a = x_{0} \notin \spt \mu$ or $b = y_{0} \notin \spt \mu$. In this case the component actually has the form $[x_{0},b)$ or $(a,y_{0}]$. We will treat these special cases separately in a moment.

	Without loss of generality assume $x<y$, and set $h=y-x$. Note first that
	\begin{equation}\label{form127} |U\mu(x) - U\mu(x + h)| \leq |P\mu(x) - P\mu(x + h)| + |R\mu(x) - R\mu(x + h)|. \end{equation}
	Here $R\mu$ is $\min\{1,\alpha + \gamma\}$-H\"older continuous on $I$ by Lemma \ref{lemma10}, so it suffices to show
	\begin{displaymath} |P\mu(x) - P\mu(x + h)| \lesssim |h|^{\min\{1,\alpha + \gamma\}} = |x - y|^{\min\{1,\alpha + \gamma\}}. \end{displaymath}
	By the convexity of $P\mu$ on $[a,b]$ (see Proposition \ref{prop15} applied to $\nu = \mu$, and note also that $P\mu = U\mu - R\mu \in C(\R)$),
	\begin{align*} 
		{|P\mu(x)-P\mu(x+h)|} &\leq |P\mu(a)-P\mu(a+h)| + |P\mu(b-h)-P\mu(b)| \\
		& \le |U\mu(a)-U\mu(a+h)|+|R\mu(a)-R\mu(a+h)| \notag\\
		& \quad +|U\mu(b-h)-U\mu(b)|+|R\mu(b-h)-R\mu(b)|. \notag\end{align*}
	The second and fourth terms are bounded from above by $\lesssim |x - y|^{\min\{1,\alpha + \gamma\}}$ by Lemma \ref{lemma10}, so it suffices to show that 
	\begin{displaymath} |U\mu(a) - U\mu(a + h)| \lesssim |x - y|^{\min\{1,\alpha + \gamma\}} \quad \text{and} \quad |U\mu(b - h) - U\mu(b)| \lesssim |x - y|^{\min\{1,\alpha + \gamma\}}. \end{displaymath}
	For this purpose, we apply Proposition \ref{prop16} to the functions $U = U\mu$, $P = P\mu$, and $R = R\mu$ on the interval $[a,b]$, with exponent $\min\{1,\alpha + \gamma\}$ in place of $\alpha$. Let us discuss the hypotheses of the proposition.
	
	First, we need that $P$ is convex on $[a,b]$. This follows from Proposition \ref{prop15} applied with $\nu = \mu$. Second, we need that $R$ is $\min\{1,\alpha + \gamma\}$-H\"older continuous on $[a,b]$. This follows from $[a,b] \subset I$, and Lemma \ref{lemma10}.
	
	Since the hypotheses of Proposition \ref{prop16} are met, and $U(a) = \mathcal{L}(a)$, $U(b) = \mathcal{L}(b)$, and $U(a + h) \geq \mathcal{L}(a + h)$ by the hypotheses of the current proposition, we may deduce that
	\begin{displaymath} \mathcal{L}(a + h) - \mathcal{L}(a) \leq U(a + h) - U(a) \leq |x - y| \cdot \frac{|\mathcal{L}(b) - \mathcal{L}(a)|}{b - a} + C|x - y|^{\min\{1,\alpha + \gamma\}}.  \end{displaymath} 
	Here $C \lesssim 1$ is twice the $\min\{1,\alpha + \gamma\}$-H\"older continuity constant of $R\mu$. This shows that
	\begin{displaymath} |U(a + h) - U(a)| \lesssim_{I} (\mathrm{Lip}(\mathcal{L}) + C)|x - y|^{\min\{1,\alpha + \gamma\}}. \end{displaymath}
	(The implicit constant is absolute if $\alpha + \gamma \geq 1$.) A similar argument gives
	\begin{displaymath} |U(b - h) - U(b)| \lesssim (C + \mathrm{Lip}(\mathcal{L}))|x - y|^{\min\{1,\alpha + \gamma\}}. \end{displaymath}
	This is what we needed to conclude in the case $a,b \in \spt \mu$.
	
	Let us then discuss the two special cases, where either 
	\begin{equation}\label{form131} a = x_{0} \notin \spt \mu \quad \text{and} \quad b \in \spt \mu \cap I, \end{equation}
	or 
	\begin{equation}\label{form132} a \in \spt \mu \cap I \quad \text{and} \quad b = y_{0} \notin \spt \mu. \end{equation}
	(It is impossible that both $a = x_{0} \notin \spt \mu$ and $b = y_{0} \notin \spt \mu$ thanks to \eqref{form130}). The cases \eqref{form131}-\eqref{form132} are similar, so let us assume we are in case \eqref{form131}. The only piece of information missing from the previous argument (the case $a,b \in \spt \mu$) is that $U(a) = \mathcal{L}(a)$. Indeed, we only know that $\mathcal{L}(a) \leq U(a) \leq \|U\|_{L^{\infty}(I)}$. Regardless, since $U(b) = \mathcal{L}(b)$ and $U(x) \geq \mathcal{L}(x)$ for $x \in [x_{0},b]$, Proposition \ref{prop16} gives
	\begin{displaymath} \mathcal{L}(b - h)  - \mathcal{L}(b) \leq U(b - h) - U(b) \leq |x - y| \cdot \frac{2\|U\|_{L^{\infty}(I)}}{b - x_{0}} + C|x - y|^{\min\{1,\alpha + \gamma\}}. \end{displaymath}
	Here $C$ is again twice the $\zeta$-H\"older continuity constant of $R\mu$ on $I$. Since $\|U\|_{L^{\infty}(I)} < \infty$ and $b - x_{0} = \dist(\spt \mu,x_{0}) > 0$, the estimate above implies the following (rather qualitative) H\"older continuity estimate at $b$:
	\begin{displaymath} |U(b - h) - U(b)| \lesssim |x - y|^{\min\{1,\alpha + \gamma\}}. \end{displaymath} 
	The same argument yields no useful reasonable bound for $|U(x_{0}) - U(x_{0} + h)|$, since we do not know that $U(x_{0}) = \mathcal{L}(x_{0})$. Instead, we rely on a very crude estimate, based on $U\mu \in \mathrm{Lip}_{\mathrm{loc}}(\R \, \setminus \, \spt \mu)$. Since $x_{0} \notin \spt \mu$, there is a constant $C' \geq 1$, depending only on $\dist(x_{0},\spt \mu)$ and $\mu(I)$, such that 
	\begin{displaymath} |U(x_{0} + h) - U(x_{0})| \leq C'|x - y| \lesssim |x - y|^{\min\{1,\alpha + \gamma\}}, \quad |x - y| = |h| \leq \tfrac{1}{2}\dist(x_{0},\spt \mu). \end{displaymath}
	For $\tfrac{1}{2}\dist(x_{0},\spt \mu) \leq |h| = |x - y| \leq \dist(x_{0},\spt \mu)$, we estimate trivially
	\begin{displaymath} |U(x_{0} + h) - U(x_{0})| \leq 2\|U\|_{L^{\infty}(I)} \lesssim \frac{|x - y|}{\dist(x_{0},\spt \mu)} \lesssim |x - y|^{\min\{1,\alpha + \gamma\}}. \end{displaymath}
	
	Combining all the cases above, and recalling \eqref{form127}, we have now shown that 
	\begin{equation}\label{form129} |U\mu(x) - U\mu(y)| \lesssim |x - y|^{\min\{1,\alpha + \gamma\}}, \qquad x,y \in [a,b], \end{equation} 
	and this concludes the estimate in case (2).
	
	Finally, as in case (3), assume that $x\in \overline{I_{x}}$ and  $y\in \overline{I_{y}}$, where $x < y$, and $I_x,I_y$ are distinct connected components of $I \, \setminus \, \spt \mu$. Let $b_x$ be the right endpoint of $I_x$, and $a_y$ the left endpoint of $I_y$. Then,
	\begin{displaymath}
		|U\mu(x)- U\mu(y)|\le |U\mu(x) - U\mu(b_x)|+|U\mu(b_x)-U\mu(a_y)|+|U\mu(a_y)-U\mu(y)|.
	\end{displaymath}
	The middle term is estimated as in \eqref{form126}, and the other two terms have the bound $\lesssim |x - y|^{\min\{1,\alpha + \gamma\}}$ thanks to \eqref{form129}. This completes the proof.	
\end{proof} 

\section{$L^{p}$ bounds for fractional Laplacians of $U\mu$}\label{s:Lpfractional} Let $\mu \in \mathcal{M}(\R)$ as in Theorem \ref{mainTechnical}. The main goal of this section is to prove Corollary \ref{cor14}, which shows (roughy speaking) that certain fractional Laplacians of $U\mu$ are $L^{p}$-functions. Most of the results in this section are stated for a a general measure $\nu\in\mathcal{M}(\R)$ (recall Definition \ref{def5Intro}). The special hypotheses of Theorem \ref{mainTechnical} will only be needed in Corollary \ref{cor14}. The graph in the definition of $U^{\Gamma}$ is only required to be Lipschitz up to Corollary \ref{cor14}, where we also need a $C^{1,\alpha}$-hypothesis. From this point on in the section, we abbreviate $U := U^{\Gamma}$, and we will also use the decomposition $U = U_{1} + U_{2}$ familiar from \eqref{form133}.

We start by recording a (well-known) lemma, stating that H\"older continuous functions have bounded fractional Laplacians:

\begin{proposition}\label{prop18} Let $h \colon \R \to \R$ be $\zeta$-H\"older continuous for some $0<\zeta\le 1$, and assume that $\spt h \subset [-C,C]$ for some $C \geq 1$. Then, for every $0<\zeta'<\zeta$, we have $h \in \dot{H}^{\zeta'}$, so $\Delta^{\zeta'/2}h \in L^{2}(\R)$ (see Definition \ref{def4}). Moreover, $\Delta^{\zeta'/2}h \in L^{\infty}(\R)$, with
	\begin{equation}\label{form138} \|\Delta^{\zeta'/2}h\|_{L^{\infty}} \lesssim_{C,\zeta'} \|h\|_{\Lambda_{\zeta}}, \end{equation}
	and moreover
	\begin{equation}\label{form143} |\Delta^{\zeta'/2}h(x)| \lesssim_{C,\zeta'} \|h\|_{\Lambda^{\zeta}}(1 + |x|)^{-1 - \zeta'}, \qquad x \in \R. \end{equation}
	Here $\|h\|_{\Lambda_{\zeta}} := \sup \{|h(x) - h(y)|/|x - y|^{\zeta} : x,y \in \R, \, x \neq y\}$.
\end{proposition}
\begin{proof} Let $\alpha > 0$. By \cite[Chapter 5, Theorem 5]{MR290095}, we have $\Lambda_{\alpha}^{2,2}(\R) = \mathcal{L}^{2}_{\alpha}(\R)$ with equivalent norms. Here $\Lambda_{\zeta}^{2,2}(\R)$ consists of functions $f \in L^{2}(\R)$ such that
	\begin{displaymath} \|f\|_{\Lambda_{\alpha}^{2,2}} := \|f\|_{2} + \left(\int_{\R} \frac{\|f(x + t) - f(x)\|_{L^{2}}^{2}}{|t|^{1 + 2\alpha}} \, dt \right)^{1/2} < \infty, \end{displaymath}
	and $\mathcal{L}^{2}_{\alpha}(\R)$ consists of functions $f \in L^{2}(\R)$ such that $\xi \mapsto (1 + |\xi|^{2})^{\alpha/2}\hat{f}(\xi) \in L^{2}(\R)$. Now, our hypothesis that $h$ is $\zeta$-H\"older continuous and compactly supported easily implies $h \in \Lambda_{\zeta'}^{2,2}(\R)$ for every $\zeta' \in (0,\zeta)$, and therefore also $h \in \mathcal{L}^{2}_{\zeta'}(\R) \subset \dot{H}^{\zeta'}$ (the latter inclusion being obvious from Definition \ref{def3}). We have now demonstrated that $\Delta^{\zeta'/2}h \in L^{2}(\R)$.
	
	To upgrade this conclusion to $\Delta^{\zeta'/2}U_{2}\mu \in L^{\infty}(\R)$, we employ a pointwise definition of the fractional Laplacian. Fix $\zeta' \in (0,\zeta)$, in particular $\zeta < 1$. Then, repeating the arguments below \eqref{form52}, we may apply \cite[(2.4.7)]{MR3243734} with the choice $N = 0$ to obtain the formula
	\begin{align}\label{form140} \sigma(\zeta'/2)(\Delta^{\zeta'/2}f)(x) & = \int_{|x - y| \geq 1} \frac{\sigma(-1 - \zeta')f(y)}{|x - y|^{1 + \zeta'}} \, dy + b(\zeta',0)f(x)\\
		& \qquad + \int_{|x - y| < 1} \frac{\sigma(-1 - \zeta')[f(y) - f(x)]}{|y - x|^{1 + \zeta'}} \, dy, \qquad f \in \mathcal{S}(\R), \, x \in \R. \notag \end{align} 
	We will check that the same formula remains valid for $f = h$. Let us first note that the pointwise formula is well-defined, since $\|h\|_{L^{\infty}} < \infty$, and the integral in the third term is absolutely convergent thanks to the $\zeta$-H\"older continuity of $h$. These facts also imply that if the pointwise formula holds for $\Delta^{\zeta'/2}h$, then
	\begin{displaymath} \|\Delta^{\zeta'/2}h\|_{L^{\infty}} \lesssim_{\zeta'} \|h\|_{L^{\infty}} + \|h\|_{\Lambda_{\zeta}} \lesssim_{C} \|h\|_{\Lambda_{\zeta}}, \end{displaymath}
	as desired in \eqref{form138} (the $L^{\infty}$-norm is bounded by the H\"older norm thanks to our bounded support hypothesis). Still assuming that the pointwise formula holds for $h$, the decay \eqref{form143} at infinity is determined by the term on line \eqref{form140}, since the two other terms vanish for $|x| \geq C + 1$. Now \eqref{form143} follows from Lemma \ref{lemma5} applied to $F(y) = \min\{|y|^{-1 - \zeta'},1\}$ and $G = h$ (the necessary decay of $G$ follows trivially from the compact support).
	
	Let us finally justify that the pointwise formula holds. Let $\{\varphi_{\delta}\}_{\delta > 0}$ be a standard approximate identity, and abbreviate $h_{\delta} := h \ast \varphi_{\delta} \in C^{\infty}_{c}(\R)$. Then, using the $L^{2}$-definition of $\Delta^{\zeta'/2}h$ and Plancherel, it is easy to see that
	\begin{equation}\label{form139} \int \Delta^{\zeta'/2}h \cdot \psi = \lim_{\delta \to 0} \int \Delta^{\zeta'/2}h_{\delta} \cdot \psi, \qquad \psi \in C_{c}^{\infty}(\R).  \end{equation}
	Since $h \in C_{c}^{\infty}(\R)$, the fractional Laplacian $\Delta^{\zeta'/2}h$ has the pointwise expression in \eqref{form140}. It remains to check that 
	\begin{align*} \lim_{\delta \to 0} & \int_{|x - y| \geq 1} \frac{\sigma(-1 - \zeta)h_{\delta}(y)}{|x - y|^{1 + \zeta}} \, dy + b(\zeta,0)h_{\delta}(x) + \int_{|x - y| < 1} \frac{\sigma(-1 - \zeta)[h_{\delta}(y) - h_{\delta}(x)]}{|y - x|^{1 + \zeta}} \, dy\\
		& = \int_{|x - y| \geq 1} \frac{\sigma(-1 - \zeta)h(y)}{|x - y|^{1 + \zeta}} \, dy + b(\zeta,0)h(x) + \int_{|x - y| < 1} \frac{\sigma(-1 - \zeta)[h(y) - h(x)]}{|y - x|^{1 + \zeta}} \, dy \end{align*}
	for all $x \in \R$, and then apply the dominated convergence theorem in \eqref{form139}. At this point, we will have shown that the $L^{2}$-function $\Delta^{\zeta'/2}h$ coincides a.e. with the $L^{\infty}$-representative given by the pointwise formula. We leave the straightforward if tedious details of the convergence argument to the reader. \end{proof}

The following corollary is immediate.

\begin{cor}\label{cor9} Let $\nu \in \mathcal{M}(\R)$ with $\spt \nu \subset [-1,1]$, and assume that $U_{2}\nu$ is $\zeta$-H\"older continuous on $\R$ for some $0<\zeta\le 1$. Then, for every $0<\zeta'<\zeta$, we have $\Delta^{\zeta'/2}U_{2}\nu \in L^{\infty}(\R)$, with
	\begin{displaymath} \|\Delta^{\zeta'/2}U_{2}\nu\|_{L^{\infty}} \lesssim_{\zeta'} \|U_{2}\nu\|_{\Lambda_{\zeta}}. \end{displaymath}
	Moreover,
	\begin{displaymath} |\Delta^{\zeta'/2}U_{2}\nu(x)| \lesssim_{\zeta'} \|U_{2}\nu\|_{\Lambda_{\zeta}}(1 + |x|)^{-1 - \zeta'}, \qquad x \in \R. \end{displaymath} \end{cor}

We then record another corollary (of both Corollary \ref{cor9} and Proposition \ref{prop:lap-potential-Lp}) which yields $L^{p}$-bounds for the fractional Laplacians of the whole logarithmic potential $U\nu$. This result is needed to apply Proposition \ref{prop8}.

\begin{cor}\label{cor10} Let $\nu \in \mathcal{M}(\R)$ with $\spt \nu \subset [-1,1]$ and $\|\nu\|=1$, and assume that $U_{2}\nu$ is $\zeta$-H\"older continuous on $\R$ for some $0 < \zeta \leq 1$. Then, 
	\begin{equation}\label{form144} \Delta^{\zeta'/2}U\nu \in L^{p}(\R), \qquad \zeta' \in (0,\zeta), \, p \in (1/\zeta',\infty), \end{equation}
	with $\|\Delta^{\zeta'/2}U\nu\|_{L^{p}} \lesssim_{\zeta',p} \|U_{2}\nu\|_{\Lambda_{\zeta}} + 1$. Here $\Delta^{\zeta'/2}U\nu$ is \emph{a priori} the tempered distribution familiar from Definition \ref{def6}.
\end{cor}  		

\begin{proof}[Proof of Corollary \ref{cor10}] Fix $0 < \zeta' < \zeta$. It is easy to check that 
	\begin{displaymath} \Delta^{\zeta'/2}U\nu = \Delta^{\zeta'/2}U_{1}\nu + \Delta^{\zeta'/2}U_{2}\nu \end{displaymath}
	as distributions, where $\Delta^{\zeta'/2}U_{1}\nu$ was defined inside Proposition \ref{prop:lap-potential-Lp}, and the main result of that proposition was that $\Delta^{\zeta'/2}U_{1}\nu \in L^{p}(\R)$ for all $p \in (\zeta',\infty)$, with $\|\Delta^{\zeta'/2}U_{1}\nu\|_{L^{p}} \lesssim_{p} 1$. On the other hand, Corollary \ref{cor9} shows that $\Delta^{\zeta'/2}U_{2}\nu \in L^{1}(\R) \cap L^{\infty}(\R)$, so in particular $\Delta^{\zeta'/2}U_{2}\nu \in L^{p}(\R)$ for $p \in (\zeta',\infty)$. This concludes the proof.  \end{proof}

The hypothesis of Corollary \ref{cor10} corresponds to the conclusion of Corollary \ref{cor11}. We will next record a variant of Corollary \ref{cor10} where the hypothesis rather corresponds to the conclusion of Corollary \ref{cor12}. We start with a (well-known) result of general nature.

\begin{lemma}\label{lemma11} Let $f \in L^{1}(\R) \cap L^{2}(\R)$, and assume that the distributional derivative $\partial f \in L^{p}(\R)$ for some $p \in (1,\infty)$. Define the tempered distribution $\Delta^{1/2}f \in \mathcal{S}'(\R)$ by
	\begin{equation}\label{form146} \int \Delta^{1/2}f \cdot \psi := \int f \cdot \Delta^{1/2}\psi, \qquad \psi \in \mathcal{S}(\R). \end{equation}
	(This is well-defined since $f \in L^{1}(\R)$.) Then $\Delta^{1/2}f \in L^{p}(\R)$, and $\|\Delta^{1/2}f\|_{L^{p}} \lesssim_{p} \|\partial f\|_{L^{p}}$.
\end{lemma}

\begin{proof} By hypothesis $\partial f \in L^{p}(\R)$, so $\widehat{\partial f} \in \mathcal{S}'(\R)$, and therefore for $\psi \in \mathcal{S}(\R)$,
	\begin{equation}\label{form145} \int \partial f \cdot \psi = \widehat{\partial f}(\widecheck{\psi}) = \hat{f}([\xi \mapsto 2\pi i \xi \cdot \widecheck{\psi}(\xi)]) = 2\pi i \int \hat{f}(\xi) \cdot \xi\widecheck{\psi}(\xi) \, d\xi.  \end{equation} 
	The final equation makes sense by $f \in L^{1}(\R)$. Now, we want to apply the relation above to functions $\psi \in \mathcal{S}(\R)$ of specific form, defined as follows. Let $\varphi \in C^{\infty}(\R)$ be an auxiliary function satisfying $\mathbf{1}_{\R \, \setminus \, [-1,1]} \leq \varphi \leq \mathbf{1}_{\R \, \setminus \, [-1/2,1/2]}$, and write $\varphi_{\epsilon}(\xi) := \varphi(\xi/\epsilon)$. Thus $\varphi_{\epsilon}(\xi) \to 1$ for every $\xi \in \R \, \setminus \, \{0\}$, as $\epsilon \to 0$. Moreover,
	\begin{equation}\label{form147} |\varphi_{\epsilon}^{(j)}(\xi)| \lesssim |\xi|^{-j}, \qquad \xi \in \R \, \setminus \, \{0\}, \, j \in \{0,1,2\}, \end{equation}
	Now, for $\eta \in \mathcal{S}(\R)$ and $\epsilon > 0$ fixed, note that 
	\begin{displaymath} \xi \mapsto \frac{|\xi| }{2\pi i \xi}\varphi_{\epsilon}(\xi)\widecheck{\eta}(\xi) \in \mathcal{S}(\R). \end{displaymath}
	Therefore we are allowed to apply \eqref{form145} to $\psi_{\epsilon} \in \mathcal{S}(\R)$, namely the Fourier transform of the above function. The result is 
	\begin{equation}\label{form148} \int \partial f \cdot \psi_{\epsilon} = \int \hat{f}(\xi) \cdot |\xi|\varphi_{\epsilon}(\xi) \widecheck{\eta}(\xi) \, d\xi. \end{equation}
	Dominated convergence and Plancherel (recall $f \in L^{2}(\R)$) shows 
	\begin{displaymath} \lim_{\epsilon \to 0} \int \hat{f}(\xi) \cdot |\xi|\varphi_{\epsilon}(\xi) \widecheck{\eta}(\xi) =  \int \hat{f}(\xi) |\xi|\widecheck{\eta}(\xi) \, d\xi \stackrel{\eqref{form146}}{=} \int \Delta^{1/2}f \cdot \eta. \end{displaymath} 
	On the other hand, we may write $\psi_{\epsilon} = T_{\epsilon}\eta$, where $T_{\epsilon}$ is the Fourier multiplier with symbol $m(\xi) = (|\xi|/ 2\pi i \xi)\varphi_{\epsilon}(\xi)$. Thanks to \eqref{form147} and the product rule, the symbol $m$ satisfies the hypothesis \eqref{form105} of Mihlin's multiplier theorem with constants independent of $\epsilon > 0$. Consequently $\|\psi_{\epsilon}\|_{L^{p'}} \lesssim \|\eta\|_{L^{p'}}$. Combining this information with \eqref{form148} shows that
	\begin{displaymath} \left| \int \Delta^{1/2}f \cdot \eta \right| \leq \limsup_{\epsilon \to 0} \left| \int \partial f \cdot \psi_{\epsilon} \right| \lesssim \|\partial f\|_{L^{p}}\|\eta\|_{L^{p'}}, \qquad \eta \in \mathcal{S}(\R). \end{displaymath}
	Thus $\Delta^{1/2}f \in L^{p}(\R)$ with $\|\Delta^{1/2}f\|_{L^{p}} \lesssim \|\partial f\|_{L^{p}}$. This concludes the proof. \end{proof}

\begin{cor}\label{cor13} Let $\nu \in \mathcal{M}(\R)$ with $\spt \nu \subset [-1,1]$ and $\|\nu\|=1$. Assume that the distributional derivative $\partial U_{2}\nu \in L^{p}(\R)$ for some $p \in (1,\infty)$. Then, $\Delta^{1/2}U\nu \in L^{p}(\R)$ for every $p \in (1,\infty)$, and indeed
	\begin{displaymath} \|\Delta^{1/2}U\nu\|_{L^{p}} \lesssim_{p} \|\partial U_{2}\nu\|_{L^{p}} + 1, \qquad p \in (1,\infty). \end{displaymath}	
\end{cor} 

\begin{proof} As in the proof of Corollary \ref{cor10}, the idea is to split
	\begin{displaymath} \Delta^{1/2}U\nu = \Delta^{1/2}U_{1}\nu + \Delta^{1/2}U_{2}\nu, \end{displaymath}
	and verify that both terms on the right are $L^{p}$-functions. In fact, Proposition \ref{prop:lap-potential-Lp} tells us that $\Delta^{1/2}U_{1}\nu \in L^{p}(\R)$ for all $p \in (1,\infty)$ with $\|\Delta^{1/2}U_{1}\nu\|_{L^{p}} \lesssim_{p} 1$. On the other hand, $U_{2}\nu$ is a bounded a compactly supported function, so certainly $U_{2}\nu \in L^{1}(\R) \cap L^{2}(\R)$. Therefore Lemma \ref{lemma11} implies that $\Delta^{1/2}U_{2}\nu \in L^{p}(\R)$ with $\|\Delta^{1/2}U_{2}\nu\|_{L^{p}} \lesssim_{p} \|\partial U_{2}\nu\|_{L^{p}}$. This completes the proof. \end{proof}

Corollaries \ref{cor10} and \ref{cor13} deal with general probability measures $\nu \in \mathcal{M}(\R)$, and show that \textbf{if} $U_{2}\nu$ is either H\"older continuous or suitably differentiable, then certain fractional Laplacians of $U\nu$ lie in $L^{p}$. We now specialise the discussion to the case where $\nu$ satisfies the hypotheses of Theorem \ref{mainTechnical}. Now the hypotheses of Corollaries \ref{cor10} and \ref{cor13} are satisfied thanks to Corollaries \ref{cor11} and \ref{cor12}. We put the pieces together in the following master corollary:

\begin{cor}\label{cor14} Let $\alpha \in (0,1]$. Let $\Gamma(x) = (x,A(x))$, where $A \in \mathrm{Lip}(\R) \cap C^{1,\alpha}(\R)$, and the Lipschitz constant of $A$ is so small that the conclusions of Lemmas \ref{PLemma} and \ref{RLemma} hold.
	
	Let $\gamma \in [0,1]$. Let $\mu \in \mathcal{M}(\R)$, and assume that $\mu$ satisfies the hypotheses of Theorem \ref{mainTechnical} on some compact interval $I_{0} \subset \R$. Thus, for some Lipschitz function $\mathcal{L} \colon \R \to \R$,
	\begin{itemize}
		\item $U^{\Gamma}\mu(x) = \mathcal{L}(x)$ for all $x \in \spt \mu \cap I_{0}$,
		\item $U^{\Gamma}\mu(x) \geq \mathcal{L}(x)$ for all $x \in I_{0}$.
	\end{itemize}
	Let $I = [x_{0},y_{0}] \subset I_{0}$ be an interval such that 
	\begin{itemize}
		\item[(a)] $\mu(B(z_{0},r)) \lesssim r$ for both $z_{0} \in \{x_{0},y_{0}\}$, and for all $r > 0$, and
		\item[(b)] $\mu(B(x,r)) \lesssim r^{\gamma}$ for all $x \in I$ and $r \in (0,1)$.
	\end{itemize}
	Then
	\begin{itemize}
		\item[(1)] $\Delta^{\zeta/2}U^{\Gamma}(\mu|_{I}) \in L^{p}(\R)$ for all $\zeta \in (0,\min\{1,\alpha + \gamma\})$ and all $p \in (1/\zeta,\infty)$.
		\item[(2)] If $\alpha + \gamma > 1$, then $\Delta^{1/2}U^{\Gamma}(\mu|_{I}) \in L^{p}(\R)$ for all $p \in (1,\infty)$.
	\end{itemize}
	Here $\Delta^{\zeta/2}U^{\Gamma}(\mu|_{I})$ is \emph{a priori} the tempered distribution from Definition \ref{def6}. \end{cor}

\begin{proof} To prove (1), fix $\zeta \in (0,\min\{1,\alpha + \gamma\})$ and $\zeta < \zeta' < \min\{1,\alpha + \gamma\}$. Then $U^{\Gamma}_{2}(\mu|_{I})$ is $\zeta'$-H\"older continuous on $\R$ by Corollary \ref{cor11}. Now (1) follows from Corollary \ref{cor10}. 
	
	To prove (2), fix $p \in (1,\infty)$, and note that $\partial U^{\Gamma}_{2}(\mu|_{I}) \in L^{p}(\R)$ by Corollary \ref{cor12}. Therefore also $\Delta^{1/2}U^{\Gamma}(\mu|_{I}) \in L^{p}(\R)$ by Corollary \ref{cor13}. \end{proof} 

\subsection{Proof of Theorem \ref{mainTechnical}}\label{s:proofMainTechnical}
Having Corollary \ref{cor14} at hand we are ready to prove the main technical result of this article, Theorem \ref{mainTechnical}. For the rest of this subsection we assume that $p\in (1,\infty)$, $A\in C^{1,\alpha}(\R;\R^{d-1})$, $\mathcal{L}\in\lip(\R)$, $\mu\in\mathcal{M}(\R)$ and $I_0\subset\R$ are as in Theorem \ref{mainTechnical}.

We start by choosing a sub-interval $I=[x_0,y_0]\subset I_0$ such that
\begin{equation}\label{eq:good-endpoints}
	\mu(B(z_{0},r)) \lesssim r \quad\text{for both $z_{0} \in \{x_{0},y_{0}\}$, and for all $r > 0$}.
\end{equation}
Let $\mathcal{M}\mu$ denote the centered Hardy-Littlewood maximal operator applied to $\mu$, so that
\begin{equation*}
	\mathcal{M}\mu(x)=\sup_{r>0}\frac{\mu(B(x,r))}{r}.
\end{equation*}
The classical weak-type estimates for $\mathcal{M}$ (see e.g. \cite[Theorem 2.5]{tolsa2014analytic}) imply that $\mathcal{M}\mu(x)<\infty$ for $\mathcal{H}^1$-a.e. $x\in\R$. We choose $x_0,y_0\in I_0$ such that $\mathcal{M}\mu(z_0)<\infty$ for $z_0\in \{x_{0},y_{0}\}$, so that \eqref{eq:good-endpoints} is satisfied. Note that we are free to choose $x_0$ and $y_0$ arbitrarily close to the endpoints of $I_0$. To prove Theorem \ref{mainTechnical} we need to show that $\mu\in L^p(I)$.

We first prove the following intermediate result. Recall that $\lip(A)\le\delta=\delta(p,\alpha,d)$.
\begin{lemma}\label{lem:int-Frost}
	If $\delta(\alpha,d)$ is chosen small enough, then there exists $\gamma\in (1-\alpha,1)$ and $C_\gamma>1$ such that
	\begin{equation}\label{eq:Frost-gamma}\tag{Frost$_\gamma$}
		\mu(B(x,r)) \le C_{\gamma} r^{\gamma} \quad\text{for all $x \in I$ and $r \in (0,1)$.}
	\end{equation}
\end{lemma}
\begin{proof}
	We prove \eqref{eq:Frost-gamma} by induction on $\gamma$. The inductive assumption is that $\mu$ satisfies \eqref{eq:Frost-gamma} for some $\gamma\in [0,1-\alpha]$.	
	We will show that this implies \eqref{eq:Frost-gamma} for $\gamma+{\alpha}/{2}$, that is
	\begin{equation}\label{eq:Frost2}
		\mu(B(x,r)) \le C_{\gamma+\alpha/2}\, r^{\gamma+\alpha/2} \quad\text{for all $x \in I$ and $r \in (0,1)$.}
	\end{equation}
	Once this is done, we get \eqref{eq:Frost-gamma} for some $\gamma\in (1-\alpha,1)$ after $\sim \alpha^{-1}$ many application of the inductive step. Note that the iteration can start because \eqref{eq:Frost-gamma} is trivially satisfied for $\gamma=0$.
	
	We begin the proof of \eqref{eq:Frost2}. Using our inductive assumption we can apply  Corollary \ref{cor14}. Conclusion (1) of that corollary gives $\Delta^{\zeta/2}U^{\Gamma}(\mu|_{I}) \in L^{q}(\R)$ for all $\zeta \in (0,\gamma+\alpha)$ and all $q \in (1/\zeta,\infty)$. We choose $\zeta = \gamma + \frac{3\alpha}{4}$ and $q=4/\alpha$.
	
	Now we apply Proposition \ref{prop8} to $\mu|_I$ with $\beta=1-\zeta$ and $p=q=4/\alpha$. Assuming $\lip(A)\le\delta(\alpha,d)$ is small enough we may conclude $\Delta^{(\zeta-1)/2}(\mu|_I)\in L^q$ for all $q \in (1/\zeta,\infty)$. Using Proposition \ref{frostmanProp} we arrive at
	\begin{equation*}
		\mu(I\cap B(x,r))\lesssim_{\zeta} \|\Delta^{(\zeta-1)/2}(\mu|_I)\|_{L^q}\cdot r^{\zeta-1/q}\quad x\in\R, r\in (0,1),
	\end{equation*}
	where $\Delta^{(\zeta - 1)/2}(\mu|_{I})$ refers to a convolution with a Riesz kernel, as in \eqref{form161}. Recalling that $\zeta = \gamma + \frac{3\alpha}{4}$ and $q=4/\alpha$ we get the desired Frostman condition \eqref{eq:Frost2} for $x\in I$ and $r\in (0,\min(|x-x_0|, |x-y_0|))$. 
	
	To get the estimate for the full range of radii $r\in (0,1)$ we use \eqref{eq:good-endpoints}. Let $x\in I$ and $r> \min(|x-x_0|, |x-y_0|)$. Without loss of generality assume the minimum is equal to $|x-x_0|$. Then 
	\begin{equation*}
		\mu(B(x,r))\le \mu(B(x, |x-x_0|)) + \mu(B(x_0, r))\\
		\lesssim |x-x_0|^{\gamma+\alpha/2} + r\lesssim r^{\gamma+\alpha/2},
	\end{equation*}
	with the implicit constants depending on $\|\Delta^{(\zeta-1)/2}(\mu|_I)\|_{L^q}$ and \eqref{eq:good-endpoints}. This completes the proof of \eqref{eq:Frost2}.
\end{proof}

We are ready to complete the proof of Theorem \ref{mainTechnical}.
\begin{lemma}
	If $\delta(p,d)$ is chosen small enough, then $\mu\in L^p(I)$.
\end{lemma}
\begin{proof}
	Thanks to Lemma \ref{lem:int-Frost} we can apply Corollary \ref{cor14} in the case $\alpha+\gamma>1$. Thus, conclusion (2) of that corollary gives $\Delta^{1/2}U^{\Gamma}(\mu|_{I}) \in L^{p}(\R)$ for all $p \in (1,\infty)$.
	
	Now we apply Proposition \ref{prop8} with $\beta=0$. If $\lip(A)\le\delta(p,d)$ is small enough we conclude $\|\mu|_I\|_{L^p}\lesssim \|\Delta^{1/2}U^{\Gamma}(\mu|_{I})\|_{L^p}<\infty$.\end{proof}

\section{Truncated operators $T^{\Gamma,\epsilon,R}_{\beta}$}\label{s:TGammaBeta}
Let $A \colon \R \to \R^{d - 1}$ be a Lipschitz function, and $\Gamma(x) \coloneqq (x,A(x))$. The purpose of Sections \ref{s:TGammaBeta}--\ref{s:limit operators} is to prove Proposition \ref{prop8}, which says that if $p \in (1/(1 - \beta),\infty)$, and $\lip(A) \leq \delta(d,p)$, then
\begin{equation}\label{eq:goal}
	\|\Delta^{-\beta/2}\mu\|_{L^{p}} \lesssim \|\Delta^{(1 - \beta)}U^{\Gamma}\mu\|_{L^{p}}.
\end{equation}

\subsection{Outline of the proof of Proposition \ref{prop8}}\label{s:outlineTGammaBeta} 
The estimate \eqref{eq:goal} will be proved by studying a family operators $T_{\beta}^{\Gamma}$, for $\Rea \beta \in [0,1]$, which are formally defined by
\begin{equation}\label{form66} T_{\beta}^{\Gamma} = \Delta^{(1 - \beta)/2}U^{\Gamma}\Delta^{\beta/2}. \end{equation}
These operators were already mentioned in Section \ref{s:outline}. Here $\Delta^{(1 - \beta)/2}$ and $\Delta^{\beta/2}$ are fractional Laplacians, and $U^{\Gamma}$ is the logarithmic graph potential from Definition \ref{defGraphPotential}. It takes a lot of work to make rigorous sense of the formula \eqref{form66}, and to prove everything we need about the operators $T_{\beta}^{\Gamma}$. We now give an outline of the steps involved.
\begin{itemize}
\item[(a)] In Section \ref{s:truncations}, we define doubly truncated versions of the operators $T_{\beta}^{\Gamma}$, denoted $T^{\Gamma,\epsilon,R}_{\beta}$. These truncations are initially defined $\dot{H}^{\beta} \to \mathcal{S}'(\R)$. The main result of Section \ref{s:truncations} is Proposition \ref{prop5} which shows that the operator family 
\begin{displaymath} \{T^{\Gamma,\epsilon,R}_{\beta}\}_{\beta} \end{displaymath}
is analytic in the strip $\Rea \beta \in (0,1)$, and continuous in the strip $\Rea \beta \in [0,1]$.
\item[(b)] In Section \ref{s:FlatCase} we show that if $\Gamma$ is "flat" (that is: $A \equiv 0$), the operators $T_{\beta}^{\Gamma,\epsilon,R}$ converge weakly to a multiple of the identity as $\epsilon \to 0$ and $R \to \infty$. 
\end{itemize}
The previous points contained the definitions and some basic features of the operators $T_{\beta}^{\Gamma,\epsilon,R}$, but did not yet shed light on their mapping properties. As the next step we include Section \ref{s:CZOPreliminaries} where we recap notions and results from \emph{Calder\'on-Zygmund theory}.
\begin{itemize}
\item[(c)] In Section \ref{s:L2Bounds} we will use the theory from Section \ref{s:CZOPreliminaries} to prove $L^{p} \to L^{p}$ bounds for the operators $T^{\Gamma,\epsilon,R}_{\beta}$, which are uniform in $\epsilon,R$. This will allow (or would allow) us to define the operators 
\begin{displaymath} T^{\Gamma}_{\beta} \colon L^{p}(\R) \to L^{p}(\R) \end{displaymath}
as weak limits of the operators $T^{\Gamma,\epsilon,R}_{\beta}$ along subsequences $\epsilon_{j},R_{j}$. In particular, by (c), this limit is a multiple of the identity operator in the case where $\Gamma$ is "flat". For technical convenience, the definition of the operators $T_{\beta}^{\Gamma}$ is postponed a bit further, until Theorem \ref{thm1}.
\item[(d)] In Section \ref{s:LpDifferences}, we will prove that 
\begin{displaymath} \|T_{\beta}^{\Gamma_{t}} - T_{\beta}^{\Gamma_{s}}\|_{L^{p} \to L^{p}} \lesssim_{d,p} \mathrm{Lip}(A) \cdot |t - s|, \quad p \in (1,\infty), \, \beta \in [0,1], \, s,t \in [0,1]. \end{displaymath}
Here $\Gamma_{s}$ is the graph of the Lipschitz function $A(x) = sA(x)$, so in particular $\Gamma_{0}$ is the "flat" graph. By the last conclusion in (d), this will imply that $T_{\beta}^{\Gamma} = T_{\beta}^{\Gamma_{1}}$ is invertible on $L^{p}(\R)$ if $\mathrm{Lip}(A)$ is sufficiently small, depending on $d,p$.
\item[(e)] In Section \ref{s:limit operators} we finally derive \eqref{eq:goal} as the main corollary of the $L^{p}$-invertibility of the operators $T_{\beta}^{\Gamma}$: \eqref{eq:goal} holds whenever $\mu \in \mathcal{M}(\R)$, and $\mathrm{Lip}(A)$ is so small that $T_{\beta}^{\Gamma}$ is invertible on $L^{p}(\R)$. \end{itemize}

\subsection{The truncated operators $T^{\Gamma,\epsilon,R}_{\beta}$}\label{s:truncations} We now start in earnest the program of defining and studying the operators $T^{\Gamma}_{\beta}$. The first step is to define the truncated versions $T^{\Gamma,\epsilon,R}_{\beta}$, and for that purpose we need suitable truncations of the kernel $x \mapsto -\log |x|$. On seeing the complicated statement in Lemma \ref{truncationLemma}, the reader may wonder why we do not simply define $\log_{\epsilon,R}(x) := -\varphi_{\epsilon,R}(x)\log |x|$, where $\varphi_{\epsilon,R}$ is a standard cut-off function. The reason is a peculiarity of the logarithmic kernel (in contrast to e.g. positive-order Riesz kernels): with the definition above, the derivatives of $\log_{\epsilon,R}$ would not satisfy the correct uniform decay bounds stated in Lemma \ref{truncationLemma}(iv). 
\begin{lemma}\label{truncationLemma} For each $0 < \epsilon \leq 1 \leq R$, there exists a function 
\begin{displaymath} \log_{\epsilon,R} \in C^{4}(\R \, \setminus \, \{0\}) \cap \mathrm{Lip}(\R) \end{displaymath}
with $\spt \log_{\epsilon,R} \subset \bar{B}(2R)$, and a constant $C_{R} \in \R$, depending only on $R$, with $|C_{R}| \lesssim \log R$, with the following properties:
\begin{itemize}
\item[(i)] $\log_{\epsilon,R}(x) = \log |x|^{-1} + C_{R}$ for $\epsilon \leq |x| \leq R$,
\item[(ii)] The restriction of $x \mapsto \log_{\epsilon,R}(x)$ to $\R \, \setminus \, [-R,R]$ is independent of $\epsilon$, and satisfies
\begin{displaymath} \|\log_{\epsilon,R}\|_{L^{\infty}(\R \, \setminus \, [-R,R])} \lesssim \log R, \end{displaymath}
\item[(iii)] The restriction of $x \mapsto \log_{\epsilon,R}(x) - C_{R}$ to $[-\epsilon,\epsilon]$ is independent of $R$, and satisfies 
\begin{displaymath} \|\log_{\epsilon,R} - C_{R}\|_{L^{\infty}([-\epsilon,\epsilon])} \lesssim 1 + \log \epsilon^{-1}, \end{displaymath}
\item[(iv)] The derivative 
\begin{displaymath} x \mapsto k_{\epsilon,R}(x) \coloneqq \partial_{x} \log_{\epsilon,R}(x) \in C^{3}_{c}(\R \, \setminus \, \{0\})
\end{displaymath} is a bounded odd function satisfying $\spt k_{\epsilon,R} \subset \bar{B}(2R)$, $\|k_{\epsilon,R}\|_{L^{\infty}(\R)} \lesssim \epsilon^{-1}$, and 
\begin{displaymath} |\partial_{x}^{j} k_{\epsilon,R}(x)| \lesssim |x|^{-1 - j}, \qquad x \in \R \, \setminus \, \{0\}, \, j \in \{0,1,2,3\}. \end{displaymath}
\end{itemize}
All the implicit "$\lesssim$" constants above are absolute.
\end{lemma}

\begin{remark}\label{rem2} The construction in Lemma \ref{truncationLemma} is a "double truncation". Sometimes it is sufficient to truncate only from the "$\epsilon$-side". In this case, the construction below yields a (\textbf{not} compactly supported) function $\log_{\epsilon} \in C^{4}(\R \, \setminus \, \{0\}) \cap \mathrm{Lip}(\R)$ with the property that $\log_{\epsilon}(x) = \log |x|^{-1}$ for $|x| \geq \epsilon$, $\|\log_{\epsilon}\|_{L^{\infty}([-\epsilon,\epsilon])} \lesssim 1 + \log \epsilon^{-1}$, and such that the derivative $k_{\epsilon} = \partial \log_{\epsilon}$ satisfies all the properties of (iv), except for the claim about compact support.  \end{remark} 

\begin{proof}[Proof of Lemma \ref{truncationLemma}] We first define the following coefficients (found with the kind assistance of \emph{Maxima}):
\begin{displaymath} 
	\begin{cases} a_1 \coloneqq -\tfrac{4}{\epsilon}, \\ a_2 \coloneqq \tfrac{3}{\epsilon^{2}}, \\ a_3 \coloneqq -\tfrac{4}{3\epsilon^{3}},\\ a_4\coloneqq \tfrac{1}{4\epsilon^4}
	\end{cases} \quad \text{and} \quad  
	\begin{cases} b_0 \coloneqq -209, \\ b_1 \coloneqq -2240 \cdot R, \\ b_2 \coloneqq 5040 \cdot R^{2}, \\ b_3 \coloneqq -\tfrac{24640}{3} \cdot R^{3}, \\ b_4 \coloneqq 8820 \cdot R^{4}, \\ b_5 \coloneqq -\tfrac{29568}{5} \cdot R^{5},\\ b_6\coloneqq 2240\cdot R^6,\\ b_7\coloneqq-\tfrac{2560}{7}\cdot R^7
\end{cases} 
\end{displaymath} 
Next, we define the auxiliary functions
\begin{displaymath} \mathfrak{p}(x) \coloneqq a_1|x| + a_2|x|^{2} + a_3|x|^{3} + a_4|x|^4, \qquad x \in \R, \end{displaymath} 
and
\begin{displaymath} \mathfrak{h}(x) \coloneqq b_0 \cdot \log |x| + \sum_{i=1}^7 \frac{b_i}{|x|^i}, \qquad x \in \R \, \setminus \, \{0\}.\end{displaymath}
We define the coefficient 
\begin{displaymath} C_{R} \coloneqq \log R + \mathfrak{h}(R) - \mathfrak{h}(2R), \end{displaymath}
and finally the function $\log_{\epsilon,R} \colon \R \to \R$:
\begin{displaymath} \log_{\epsilon,R}(x) \coloneqq \begin{cases} \mathfrak{p}(x) + \log \epsilon^{-1} + C_{R} - \mathfrak{p}(\epsilon), & |x| \leq \epsilon, \\ \log |x|^{-1} + C_{R}, & \epsilon \leq |x| \leq R, \\ \mathfrak{h}(x) - \mathfrak{h}(2R), & R \leq |x| \leq 2R, \\ 0, & |x| \geq 2R. \end{cases} \end{displaymath} 
The coefficients $a_i$ were chosen so that the $1^{st}$, $2^{nd}$, $3^{rd}$, and $4^{th}$ derivatives of $\mathfrak{p}$ coincide with the respective derivatives of $x \mapsto \log |x|^{-1}$ at the points $x \in \{-\epsilon,\epsilon\}$. Similarly, the coefficients $b_i$ were chosen so that the $1^{st}$, $2^{nd}$, $3^{rd}$, and $4^{th}$ derivatives of $\mathfrak{h}$:
\begin{itemize}
\item[(a)] coincide with the respective derivatives of $x \mapsto \log |x|^{-1}$ at the points $x \in \{R,-R\}$.
\item[(b)] vanish at $x \in \{2R,-2R\}$.
\end{itemize}
Since furthermore $\log_{\epsilon,R}$ is (thanks to all the constants in the definition) is continuous at $|x| \in \{\epsilon,R,2R\}$, we conclude that
\begin{displaymath} \log_{\epsilon,R} \in C^{4}(\R \, \setminus \, \{0\}) \quad \text{and} \quad \spt \log_{\epsilon,R} \subset \bar{B}(2R). \end{displaymath}
Evidently $\log_{\epsilon,R} \in \mathrm{Lip}(\R)$, since $\mathfrak{p}$ is locally Lipschitz continuous at $0$. The points (i)-(iii) in the statement of Lemma \ref{truncationLemma} now follow by straightforward inspection of the definition of $\log_{\epsilon,R}$, and the definitions of the coefficients $a_i, b_i$.

Regarding part (iv), write $k_{\epsilon,R}(x) \coloneqq \partial_{x} \log_{\epsilon,R}(x)$ for $x \in \R \, \setminus \, \{0\}$, thus
\begin{displaymath}  k_{\epsilon,R}(x) = \begin{cases} \tfrac{a_1x}{|x|} + 2a_2x + \tfrac{3a_3x^{3}}{|x|} + 4a_4 x^3, & 0 < |x| \leq \epsilon, \\
 -\frac{1}{x}, & 0 < |x| \leq R, \\ \frac{b_0}{x} + \sum_{i=1}^7 \frac{i b_i \sgn(x)}{|x|^{i+1}}, & R \leq |x| \leq 2R, \\ 0, & |x| \geq 2R. \end{cases} \end{displaymath} 
The derivative of an even function is always odd, so $k_{\epsilon,R}$ is an odd function (this is also apparent from the formula). As discussed around (a)-(b) above, the coefficients $a_i, b_i$ were chosen in such a way that $k_{\epsilon,R} \in C^{3}(\R \, \setminus \, \{0\})$. Furthermore, it follows from the definition of the coefficients $a_i, b_i$ that
\begin{displaymath} |\partial_{x}^{j}k_{\epsilon,R}(x)| \lesssim |x|^{-1 - j}, \qquad x \in \R \, \setminus \, \{0\}, \, j \in \{0,1,2\}. \end{displaymath}
Just to give an example of the relevant computations, note that for $x \in [-\epsilon,\epsilon] \, \setminus \, \{0\}$,
\begin{displaymath} |k_{\epsilon,R}(x)| \leq |a_1| + 2|a_2|\epsilon + 3|a_3|\epsilon^{2} +4|a_4|\epsilon^{3} \lesssim \epsilon^{-1} \leq |x|^{-1}, \end{displaymath}
and $|k_{\epsilon,R}'(x)| \leq 2|a_2| + 6|a_3|\epsilon +12|a_4|\epsilon^2 \lesssim \epsilon^{-2} \leq |x|^{-2}$. Further similar computations complete the proof of Lemma \ref{truncationLemma}. \end{proof}

We now define truncated a version of the logarithmic potential $U^{\Gamma}$ (Definition \ref{defGraphPotential}):

\begin{definition} Let $A \colon \R \to \R^{d - 1}$ be Lipschitz, and write $\Gamma(x) = (x,A(x))$. For $0 < \epsilon \leq 1 \leq R < \infty$, let $U^{\Gamma,\epsilon,R}$ be the operator defined by
\begin{equation}\label{form55} U^{\Gamma,\epsilon,R}f(x) \coloneqq \int f(y) \log_{\epsilon,R} |\Gamma(x) - \Gamma(y)| \, dy, \qquad f \in L^{p}(\R), \, 1 \leq p \leq \infty. \end{equation} 
Here $\log_{\epsilon,R} \in C^{4}(\R \, \setminus \, \{0\})$ is the function defined in Lemma \ref{truncationLemma}. \end{definition}
The definition also makes sense if $f$ is replaced by any finite measure $\mu$, for example $\mu \in \mathcal{M}(\R)$. Since the kernel $\log_{\epsilon,R} |\Gamma(x) - \Gamma(y)|$ is uniformly integrable (with bound depending only on $R$) in both the $x$ and $y$ variables, Schur's test implies that $U^{\Gamma,\epsilon,R}$ defines a bounded operator on $L^{p}(\R)$ with $\|U^{\Gamma,\epsilon,R}f\|_{L^{p}} \lesssim_{\epsilon,R} \|f\|_{L^{p}}$, $1 \leq p \leq \infty$.

We next aim to define carefully the operators $T^{\Gamma,\epsilon,R}_{\beta} = \Delta^{(1 - \beta)/2}U^{\Gamma,\epsilon,R}\Delta^{\beta/2}$. For this purpose, recall Definitions \ref{def3}-\ref{def4}. We first define $U^{\Gamma,\epsilon,R}\Delta^{\beta/2}$.

\begin{definition}[$U^{\Gamma,\epsilon,R}\Delta^{\beta/2}$ acting on $\dot{H}^{\beta}$]\label{def1} Let $\beta \in \C$ with $\Rea \beta > -1$. Fix $\Lambda \in \dot{H}^{\beta}$, so $f = \Delta^{\beta/2}\Lambda \in L^{2}(\R)$. Thus $U^{\Gamma,\epsilon,R}f$ may be pointwise defined by \eqref{form55}, and we set 
\begin{displaymath} (U^{\Gamma,\epsilon,R} \Delta^{\beta/2})\Lambda \coloneqq U^{\Gamma,\epsilon,R}f. \end{displaymath}
Note that $(U^{\Gamma,\epsilon,R}\Delta^{\beta/2})\Lambda \in L^{2}$, since $U^{\Gamma,\epsilon,R}$ is bounded on $L^{2}$.  \end{definition}  

We then define the operators $T_{\beta}^{\Gamma,\epsilon,R} = \Delta^{(1 - \beta)/2}U^{\Gamma,\epsilon,R}\Delta^{\beta/2}$ rigorously.

\begin{definition}[$T^{\Gamma,\epsilon,R}_{\beta}$ acting on $\dot{H}^{\beta}$]\label{def2} Let $\beta \in \C$ with $\Rea \beta \in (-1,1]$. Fix $\Lambda \in \dot{H}^{\beta}$, so $g \coloneqq (U^{\Gamma,\epsilon,R}\Delta^{\beta/2})\Lambda \in L^{2}(\R)$. We define 
\begin{displaymath} T^{\Gamma,\epsilon,R}_{\beta}\Lambda \coloneqq \Delta^{(1 - \beta)/2}g \in \mathcal{S}'(\R), \end{displaymath}
where the right hand side refers to Definition \ref{L2FracL} (note that $\Rea (1 - \beta) \geq 0$). \end{definition}

In summary, $T_{\beta}^{\Gamma,\epsilon,R}$ \emph{a priori} maps $\dot{H}^{\beta}$ to the space of tempered distributions. In Corollary \ref{cor6} we will see that $T_{\beta}^{\Gamma,\epsilon,R}$ extends to a bounded operator on $L^{p}(\R)$, indeed with operator norm independent of $\epsilon$ and $R$.

\begin{notation} We use the integral notation
\begin{displaymath} (T^{\Gamma,\epsilon,R}_{\beta}\Lambda)(g) =: \int T^{\Gamma,\epsilon,R}_{\beta}\Lambda \cdot g, \qquad g \in \mathcal{S}(\R). \end{displaymath} \end{notation} 

\begin{remark}\label{rem1} Let us spell out the meaning of Definition \ref{def2}. Fix $\Lambda \in \dot{H}^{\beta}$ and $g \in \mathcal{S}(\R)$. Then, writing "$\mathcal{F}$" for the $L^{2}$-Fourier transform,
\begin{align} \int T_{\beta}^{\Gamma,\epsilon,R}\Lambda \cdot g & \stackrel{\mathrm{Def. \ref{L2FracL}}}{=} \int |\xi|^{1 - \beta} \mathcal{F}[U^{\Gamma,\epsilon,R}(\Delta^{\beta/2}\Lambda)](\xi) \cdot \widecheck{g}(\xi) \, d\xi \notag\\
&\quad = \int \mathcal{F}[U^{\Gamma,\epsilon,R}(\Delta^{\beta/2}\Lambda)](\xi) \cdot |\xi|^{1 - \beta}\widecheck{g}(\xi) \, d\xi \notag\\
&\label{form56}\quad = \int U^{\Gamma,\epsilon,R}(\Delta^{\beta/2}\Lambda)(x) \cdot (\Delta^{(1 - \beta)/2}g)(x) \, dx, \end{align}
where the final equation follows from Plancherel. \end{remark} 

Lemma \ref{lemma3} implies a decay estimate for $U^{\Gamma,\epsilon,R}(\Delta^{\beta/2}f)$, when $f \in \mathcal{S}(\R)$:

\begin{lemma}\label{lemma4} Let $\Rea \beta \in (-1,2)$, $0 < \epsilon \leq 1 \leq R$, and $f \in \mathcal{S}(\R)$. Then,
\begin{displaymath} |U^{\Gamma,\epsilon,R}(\Delta^{\beta/2}f)(x)| \lesssim_{\beta,f,R} \min\{1,|x|^{-1 - \Rea \beta}\}, \qquad x \in \R. \end{displaymath}
Moreover, the implicit constants remain uniformly bounded if $f \in \mathcal{S}(\R)$ and $\epsilon,R$ are fixed, while $\beta$ ranges in a compact subset of the strip $-1 < \Rea \beta < 2$.
\end{lemma}

\begin{proof} The $L^{\infty}$-estimate was recorded in Definition \ref{def1}. Regarding the decay estimate, write $g \coloneqq \Delta^{\beta/2}f$, so
\begin{displaymath} |U^{\Gamma,\epsilon,R}(\Delta^{\beta/2}f)(x)| \leq \int |g(y)| |\log_{\epsilon, R} |\Gamma(x) - \Gamma(y)| | \, dy \lesssim_{R} \|g\|_{L^{\infty}(B(x,2R))},  \end{displaymath} 
using that $|\Gamma(x) - \Gamma(y)| \geq |x - y|$, and $\spt \log_{\epsilon,R} \subset \bar{B}(2R)$. But since all points $y \in B(x,2R)$ satisfy $|y| \sim_{R} |x|$ for $|x| \geq 3R$, the bound $|U^{\Gamma,\epsilon,R}g(x)| \lesssim_{\beta,f,R} |x|^{-1 - \Rea \beta}$ now follows directly from Lemma \ref{lemma3}.
\end{proof} 
We then prove that $\beta \mapsto T_{\beta}^{\Gamma,\epsilon,R}$ is analytic in the following sense:

\begin{proposition}\label{prop5} For $f,g \in \mathcal{S}(\R)$, the map
\begin{displaymath} \beta \mapsto \int T^{\Gamma,\epsilon,R}_{\beta}f \cdot g \end{displaymath}
is continuous on $\Rea \beta \in [0,1]$ and analytic on $\Rea \beta \in (0,1)$. \end{proposition}

\begin{proof} Fix $\beta \in \C$ with $\Rea \beta \in [0,1]$, and recall from \eqref{form56} that

\begin{equation}\label{form155} \int T^{\Gamma,\epsilon,R}_{\beta}f \cdot g = \int U^{\Gamma,\epsilon,R}(\Delta^{\beta/2}f)(x) \cdot (\Delta^{(1 - \beta)/2}g)(x) \, dx. \end{equation} 
For $f,g \in \mathcal{S}(\R)$, as we assume here, it is easy to check using Plancherel that
\begin{displaymath} \Delta^{\beta'/2}f \stackrel{L^{2}}{\to} \Delta^{\beta/2} f \quad \text{and} \quad \Delta^{(1 - \beta')/2}g \stackrel{L^{2}}{\to} \Delta^{(1 - \beta)/2}g, \end{displaymath}
as $\beta' \to \beta$. Since $U^{\Gamma,\epsilon,R}$ is bounded on $L^{2}(\R)$, the continuity of $\beta \mapsto \int T_{\beta}^{\Gamma,\epsilon,R}f \cdot g$ for $\Rea \beta \in [0,1]$ now follows from \eqref{form155}.

We next study the analyticity for $\Rea \beta \in (0,1)$ by directly examining the convergence of the difference quotients
\begin{align} \int & \frac{U^{\Gamma,\epsilon,R}(\Delta^{(\beta + h)/2}f)(x) \cdot \Delta^{(1 - (\beta + h))/2}g(x) - U^{\Gamma,\epsilon,R}(\Delta^{\beta/2}f)(x) \cdot \Delta^{(1 - \beta)/2}g(x)}{h} \, dx \notag\\
&\label{form57} = \int \frac{[U^{\Gamma,\epsilon,R}(\Delta^{(\beta + h)/2}f)(x) - U^{\Gamma,\epsilon,R}(\Delta^{\beta/2} f)(x)]}{h} \cdot (\Delta^{(1 - (\beta + h))/2}g)(x) \, dx\\
&\label{form58} \quad + \int U^{\Gamma,\epsilon,R}(\Delta^{\beta/2}f)(x) \cdot \frac{[\Delta^{(1 - (\beta + h))/2}g(x) - \Delta^{(1 - \beta)/2}g(x)]}{h} \, dx. \end{align} 
Here we assume that $|h|$ is so small that $\Rea (\beta + h) \in [0,1]$, so the expressions above are well-defined. As $h \to 0$, it will become evident from the proof below that the difference quotients visible inside the two integrals \eqref{form57}-\eqref{form58} (for $x \in \R$ fixed) converge to 
\begin{equation}\label{form156} U^{\Gamma,\epsilon,R}[\log \Delta \cdot \Delta^{\beta/2} f](x) \quad \text{and} \quad (\log \Delta \cdot \Delta^{(1 - \beta)/2}g)(x). \end{equation}
Here $\log \Delta \cdot \Delta^{z/2}$ refers to Fourier multiplication with symbol $\log |\xi| \cdot |\xi|^{z}$. Since $\Rea \beta \geq 0$, $\Rea (1 - \beta) \geq 0$, and $f,g \in \mathcal{S}(\R)$, the expressions in \eqref{form156} are defined pointwise as in \eqref{fracL}, and are elements of $L^{\infty}(\R)$. In particular $U^{\Gamma,\epsilon,R}[\log \Delta \cdot \Delta^{\beta/2}f](x)$ is defined by \eqref{form55}. 

To apply dominated convergence, it remains to find a $h$-independent integrable upper bound for the two integrands in \eqref{form57}-\eqref{form58}. This can be accomplished for $\Rea \beta \in (0,1)$ fixed, and provided that 
\begin{displaymath} |h| \leq \tfrac{1}{2}\dist(\Rea \beta,\{0,1\}). \end{displaymath}
Fix $\beta,h$ with these properties.

The most complicated term is
\begin{align*} & \frac{[U^{\Gamma,\epsilon,R}(\Delta^{(\beta + h)/2}f)(x) - U^{\Gamma,\epsilon,R}(\Delta^{\beta/2} f)(x)]}{h}\\
& \quad = \int \left(\int e^{2\pi i y \xi} \frac{|\xi|^{\beta + h} - |\xi|^{\beta}}{h} \hat{f}(\xi) \, d\xi \right) \log_{\epsilon,R} |\Gamma(x) - \Gamma(y)| \, dy. \end{align*} 
Since $\partial_{\beta} [\beta \mapsto |\xi|^{\beta}] = \log |\xi| \cdot |\xi|^{\beta}$, we have (write $|\xi|^{\beta + h} - |\xi|^{\beta} = \log |\xi| \int_{[\beta,\beta + h]} |\xi|^{z} \, dz$ as a complex path integral, and estimate $||\xi|^{z}| \leq (1 + |\xi|)$ for $\Rea z \in [0,1]$))
\begin{displaymath} \left| \frac{|\xi|^{\beta + h} - |\xi|^{\beta}}{h} \right| \leq |\log |\xi||(1 + |\xi|), \qquad \Rea \beta \in (0,1), \, |h| \leq \dist(\Rea \beta,\{0,1\}).  \end{displaymath}
Consequently,
\begin{align} & \left| \frac{[U^{\Gamma,\epsilon,R}(\Delta^{(\beta + h)/2}f)(x) - U^{\Gamma,\epsilon,R}(\Delta^{\beta/2} f)(x)]}{h} \right| \notag\\
&\label{form59} \qquad \leq \int \int |\log |\xi|| \cdot (1 + |\xi|) \cdot |\hat{f}(\xi)| \, d\xi  |\log_{\epsilon,R} |\Gamma(x) - \Gamma(y)|| \, dy \lesssim_{f,R} 1. \end{align}
On the other hand, since $|h| \leq \tfrac{1}{2}\dist(\Rea \beta,1)$, we have $\Rea (1 - (\beta + h)) \geq \tfrac{1}{2}\Rea \beta$, and it follows from Lemma \ref{lemma3} that 
\begin{displaymath} |\Delta^{(1 - (\beta + h))/2}g(x)| \lesssim_{\beta,g} \min\{1,|x|^{-1 - \tfrac{1}{2}\Rea \beta}\}, \qquad x \in \R. \end{displaymath}
A combination of this estimate, and \eqref{form59}, shows that the absolute value of the integrand in \eqref{form57} is bounded from above by
\begin{displaymath} \lesssim_{\beta,f,g,R} \min\{1,|x|^{-1 - \tfrac{1}{2}\Rea \beta}\}. \end{displaymath}
This upper bound is integrable and independent of $h$, so dominated convergence yields
\begin{align*} \lim_{h \to 0} & \int \frac{[U^{\Gamma,\epsilon,R}(\Delta^{(\beta + h)/2}f)(x) - U^{\Gamma,\epsilon,R}(\Delta^{\beta/2} f)(x)]}{h} \cdot \Delta^{(1 - (\beta + h))/2}g(x) \, dx\\
&\quad = \int U^{\Gamma,\epsilon,R}[\log \Delta \cdot \Delta^{\beta/2}f](x) \cdot (\Delta^{(1 - \beta)/2}g)(x) \, dx.   \end{align*}

A similar but slightly simpler argument works for the integral in \eqref{form58}. One checks that the difference quotients 
\begin{displaymath} x \mapsto \frac{[\Delta^{(1 - (\beta + h))/2}g(x) - \Delta^{(1 - \beta)/2}g(x)]}{h}  \end{displaymath}
are uniformly in $L^{\infty}(\R)$. On the other hand Lemma \ref{lemma4} guarantees that $U^{\Gamma,\epsilon,R}(\Delta^{\beta/2}f)$ has an integrable upper bound for $\Rea \beta > 0$. Therefore, 
\begin{align*} \lim_{h \to 0} & \int U^{\Gamma,\epsilon,R}(\Delta^{\beta/2}f)(x) \cdot \frac{[\Delta^{(1 - (\beta + h))/2}g(x) - \Delta^{(1 - \beta)/2}g(x)]}{h} \, dx\\
&\quad = \int U^{\Gamma,\epsilon,R}(\Delta^{\beta/2}f)(x) \cdot (\log \Delta \cdot \Delta^{(1 - \beta)/2}g)(x) \, dx \end{align*}
by the dominated convergence theorem.\end{proof}

\subsection{The operators $T^{\Gamma}_{\beta}$ in the case $A \equiv 0$}\label{s:FlatCase} In this section we show that that in the "flat case" $A \equiv 0$ the operators $T_{\beta}^{\Gamma,\epsilon,R}$ converge weakly to a multiple of the identity as $\epsilon \to 0$ and $R \to \infty$. The main result is Proposition \ref{prop4}.

Recall from Definition \ref{def2} that 
\begin{displaymath} T^{\Gamma,\epsilon,R}_{\beta}f = \Delta^{(1 - \beta)/2}[U^{\Gamma,\epsilon,R}(\Delta^{\beta/2}f)] \in \mathcal{S}'(\R), \qquad f \in \mathcal{S}(\R). \end{displaymath}
In this section, we consider the case where $\beta \in [0,1]$ is real, and $A \equiv 0$, or in other words 
\begin{displaymath} U^{\epsilon,R}f(x) \coloneqq U^{\Gamma,\epsilon,R}f(x) = \int f(y) \log_{\epsilon,R} |x - y| \, dy, \qquad x \in \R, \end{displaymath} 
is a convolution operator on $\R$. We will abbreviate 
\begin{displaymath} T_{\beta}^{\Gamma,\epsilon,R} =: T^{\epsilon,R}_{\beta} \end{displaymath}
in the case $A \equiv 0$. Proposition \ref{prop4} below shows that the operators $T^{\epsilon,R}_{\beta}$ converge to a multiple of the identity as $\epsilon \to 0$ and $R \to \infty$.
\begin{proposition}\label{prop4} Let $\beta \in [0,1]$ and $f \in \mathcal{S}(\R)$. Then, 
\begin{equation}\label{form79} \mathop{\lim_{\epsilon \to 0}}_{R \to \infty} T^{\epsilon,R}_{\beta}f = \tfrac{1}{2}f, \end{equation}
where the limit is a limit in $\mathcal{S}'(\R)$. \end{proposition} 

\begin{remark} The precise meaning of \eqref{form79} is the following. For every $\delta > 0$ and $f,g \in \mathcal{S}(\R)$, there exists a threshold $M = M(\delta,f,g) \geq 1$ such that
\begin{displaymath} \left| \int T^{\epsilon,R}_{\beta}f \cdot g - \tfrac{1}{2} \int f \cdot g \right| \leq \delta \end{displaymath}
for all $0 < \epsilon \leq M^{-1}$ and $R \geq M$. \end{remark} 

\begin{proof}[Proof of Proposition \ref{prop4}]  We need to show that
\begin{equation}\label{form65} \mathop{\lim_{\epsilon \to 0}}_{R \to \infty} \int T^{\epsilon,R}_{\beta}f \cdot g = \tfrac{1}{2}\int f \cdot g, \qquad g \in \mathcal{S}(\R), \end{equation}
where the (formal) integral on the left refers to the distribution $T_{\beta}^{\epsilon,R}f$ acting on $g$. By the definition of $T^{\epsilon,R}_{\beta}f$ (see Definition \ref{def2} and Remark \ref{rem1}), its action on $\mathcal{S}(\R)$ is given by
\begin{displaymath} \int T_{\beta}^{\epsilon,R}f \cdot g \stackrel{\mathrm{def.}}{=} \int |\xi|^{1 - \beta}\mathcal{F}(U^{\epsilon,R}(\Delta^{\beta/2}f))(\xi) \cdot \widecheck{g}(\xi). \end{displaymath}
For $\beta \in [0,1]$, we have $U^{\epsilon,R}(\Delta^{\beta/2}f) \in L^{1}$, and the Fourier transform can be defined pointwise as
\begin{displaymath} \mathcal{F}(U^{\epsilon,R}(\Delta^{\beta/2}f))(\xi) = \widehat{\log_{\epsilon,R}}(\xi)|\xi|^{\beta}\hat{f}(\xi), \qquad \xi \in \R. \end{displaymath}
Therefore,
\begin{align*} \int T^{\epsilon,R}_{\beta}f \cdot g & = \int \widehat{\log_{\epsilon,R}}(\xi)|\xi|^{1 - \beta}|\xi|^{\beta}\hat{f}(\xi) \cdot \widecheck{g}(\xi) \, d\xi = \int \widehat{\log_{\epsilon,R}}(\xi)  |\xi|\hat{f}(\xi) \cdot \widecheck{g}(\xi)\\
& = \int \left( \int e^{-2\pi ix \xi} \log_{\epsilon,R}(x) \, dx \right) |\xi|\hat{f}(\xi) \cdot \widecheck{g}(\xi) \, d\xi.  \end{align*} 
To proceed, note that 
\begin{displaymath} (x,\xi) \mapsto \log_{\epsilon,R}(x)|\xi|\hat{f}(\xi) \cdot \widecheck{g}(\xi) \in L^{1}(\R^{2}) \end{displaymath}
for $0 < \epsilon \leq 1 \leq R$ fixed, so we may apply Fubini's theorem to the integral above. This yields
\begin{equation}\label{form62} \mathop{\lim_{\epsilon \to 0}}_{R \to \infty} \int T^{\epsilon,R}_{\beta}f \cdot g = \mathop{\lim_{\epsilon \to 0}}_{R \to \infty} \int \log_{\epsilon,R}(x) \left( \int e^{-2\pi i x\xi}|\xi|\hat{f}(\xi)\widecheck{g}(\xi) d\xi \right) \, dx. \end{equation} 
Note next that the value of the integral in brackets is
\begin{equation}\label{form78} \int e^{-2\pi i x\xi}|\xi|\hat{f}(\xi)\widecheck{g}(\xi) \, d\xi = (\Delta^{1/2}f \ast g)(-x). \end{equation} 
Lemmas \ref{lemma3} and \ref{lemma5} imply that $|(\Delta^{1/2}f \ast g)(-x)| \lesssim_{f,g} (1 + |x|)^{-2}$, so $\Delta^{1/2}f \ast g \in L^{1}(\R)$. Since the Fourier transform of $\Delta^{1/2}f \ast g$ vanishes at $0$, we conclude that
\begin{displaymath} \int \int e^{-2\pi i x\xi}|\xi|\hat{f}(\xi)\widecheck{g}(\xi) \, d\xi \, dx = \int (\Delta^{1/2}f \ast g)(-x) \, dx = 0. \end{displaymath}
Therefore, the value of the limit in \eqref{form62} remains unchanged if we replace $\log_{\epsilon,R}(x)$ by $\log_{\epsilon,R}(x) - C_{R}$, where $C_{R} = O(\log R)$ is the constant from Lemma \ref{truncationLemma}. Recalling that $\log_{\epsilon,R}(x) - C_{R} = \log |x|^{-1}$ for $x \in B(R) \, \setminus \, B(\epsilon)$, we arrive at the following conclusion:
\begin{align}\label{form75} \mathop{\lim_{\epsilon \to 0}}_{R \to \infty}  \int T^{\epsilon,R}_{\beta}f \cdot g & = \mathop{\lim_{\epsilon \to 0}}_{R \to \infty} \int_{B(\epsilon)} (\log_{\epsilon,R}(x) - C_{R}) \left(\int \ldots d\xi \right) \, dx\\
&\label{form77} \qquad + \mathop{\lim_{\epsilon \to 0}}_{R \to \infty}  \int_{B(R) \, \setminus \, B(\epsilon)} \log |x|^{-1}\left( \int \ldots d\xi \right) \, dx\\
&\label{form76} \qquad + \mathop{\lim_{\epsilon \to 0}}_{R \to \infty}  \int_{\R \, \setminus \, B(R)} (\log_{\epsilon,R}(x) - C_{R}) \left( \int \ldots d\xi \right) \, dx. \end{align} 
The limit on line \eqref{form75} equals $0$, since $|\log_{\epsilon,R}(x) - C_{R}| \lesssim 1 + \log \epsilon^{-1}$ by Lemma \ref{truncationLemma}(iii), whereas the inner integral in brackets is a bounded function of $x$ (independently of $\epsilon,R$). It is worth noting that, by Lemma \ref{truncationLemma}(iii), the function $x \mapsto \log_{\epsilon,R}(x) - C_{R}$ is independent of $R$ for $x \in B(\epsilon)$. Therefore the integrals in \eqref{form75} actually only depend on $\epsilon$. 

The limit on line \eqref{form76} also equals $0$, since $|\log_{\epsilon,R}(x) - C_{R}| \lesssim \log R$ on $\R \, \setminus \, B(R)$ by Lemma \ref{truncationLemma}(ii), and since the integral in inner brackets is $O_{f,g}((1 + |x|)^{-2})$, as we noted below \eqref{form78}:
\begin{displaymath} \left| \int_{\R \, \setminus \, B(R)} (\log_{\epsilon,R}(x) - C_{R}) \left( \int \ldots  d\xi\right) \, dx \right| \lesssim \log R \int_{\R \, \setminus \, B(R)} |x|^{-2} \lesssim \frac{\log R}{R} \to 0. \end{displaymath} 
Again, it is worth noting that, by Lemma \ref{truncationLemma}(ii), the function $x \mapsto \log_{\epsilon,R}(x) - C_{R}$ is independent of $\epsilon$ on $\R \, \setminus \, B(R)$. Thus, the integrals in \eqref{form75} actually only depend on $R$.

Regarding the limit on line \eqref{form77}, the decay of the inner integral permits us to use the dominated convergence theorem:
\begin{displaymath} \mathop{\lim_{\epsilon \to 0}}_{R \to \infty} \int T^{\epsilon,R}_{\beta}f \cdot g =\int \log |x|^{-1} \left( \int e^{-2\pi i x\xi} |\xi|\hat{f}(\xi)\widecheck{g}(\xi) \right) \, dx. \end{displaymath} 
While the integral in \eqref{form77} now evidently depends on both $\epsilon,R$, the conclusion above is independent of the order of limits.

Finally, Lemma \ref{appLemma1} applied to $\hat{f}\widecheck{g} \in \mathcal{S}(\R)$ shows that
\begin{align*} \int \log |x|^{-1} \left( \int e^{-2\pi i x\xi} |\xi|\hat{f}(\xi)\widecheck{g}(\xi) \right) \, dx & = \tfrac{1}{2} \int \hat{f}(\xi)\widecheck{g}(\xi) \, d\xi = \tfrac{1}{2} \int f \cdot g. \end{align*}
This completes the proof of \eqref{form65}, and therefore the proof of the proposition. \end{proof} 

\section{$L^p$ estimates for the truncated operators}
\subsection{Preliminaries on Calder\'on-Zygmund operators}\label{s:CZOPreliminaries} In this section we define \emph{standard kernels} and \emph{Calder\'on-Zygmund operators}. The main technical result is Proposition \ref{prop9}, which follows by combining theorems due to Journ\'e and Tolsa: the proposition states that a class of standard kernels define integral operators bounded in $L^{2}$. All the kernels we encounter in the remainder of the paper will then have the form covered by Proposition \ref{prop9}. Our exposition mainly follows \cite[Chapter 4.1]{grafakos2014modern} and \cite[Part 2]{david1991wavelets}.

\begin{definition}[Standard kernels]\label{def:standardKernel}
	A function $K: (\R \times \R) \, \setminus \, \{(x,x) : x \in \R\} \to \C$ is called \emph{a standard kernel} if there exist constants $\delta, C_0\in (0,\infty)$ such that
	\begin{equation}\label{eq:SK-size}
		|K(x,y)|\le\frac{C_0}{|x-y|}, \qquad\quad\text{$x\neq y$,}
	\end{equation}
	and
	\begin{equation}\label{eq:SK-smooth}
		|K(x,y)-K(x',y)|+|K(y,x)-K(y,x')|\le C_0\frac{|x-x'|^\delta}{|x-y|^{1+\delta}}
	\end{equation}
	whenever $|x-x'|\le |x-y|/2$.
	
	We will also write that $\mathcal{K} \colon \R^{d} \, \setminus \, \{0\}$ is a standard kernel if $K(x,y) := \mathcal{K}(x - y)$ is a standard kernel in the sense above.
\end{definition}
It is well known that if a kernel $K \in \mathrm{Lip}(\R \times \R \, \setminus \, \{(x,x) : x \in \R\})$ satisfies
\begin{equation}\label{eq:SK-smooth2}
	|\partial_x K(x,y)| + |\partial_y K(x,y)|\le \frac{C}{|x-y|^2}, \qquad \text{a.e. $x\neq y$,}
\end{equation}
then it satisfies \eqref{eq:SK-smooth} with $\delta=1$ and $C_0\sim C$.

\begin{definition}[Operators associated to kernels]
	We will say that a continuous linear operator $T:\mathcal{S}(\R)\to\mathcal{S}'(\R)$ is \emph{associated} to a standard kernel $K$ if for all $f,g\in C_c^\infty(\R)$ with $\spt f\cap \spt g=\varnothing$,
	\begin{equation*}
		\langle Tf, g \rangle = \iint K(x,y) f(y) g(x)\, dy\, dx.
	\end{equation*}
\end{definition}

\begin{definition}[Calderón-Zygmund operator]
	If an operator $T$ associated to a standard kernel can be extended to a bounded operator $T:L^2(\R)\to L^2(\R)$, then we say that $T$ is a \emph{Calderón-Zygmund operator}.
\end{definition}

Below we state several results concerning Calderón-Zygmund operators we will use in the sequel. We start with the classical result of Calderón and Zygmund.
\begin{thm}[{\cite[Theorem 4.2.2]{grafakos2014modern}}]\label{LpExtension}
	If $T$ is a Calderón-Zygmund operator, then for any $1<p<\infty$ it can be extended to a bounded operator $T:L^p(\R)\to L^p(\R)$ with operator norm depending only on $p,$ its standard kernel constant $C_0$, and $\| T\|_{L^2\to L^2}$. 
\end{thm}
The following result on ``stable kernels'' is a consequence of the $T1$ theorem, and it can be found in \cite[p. 52]{david1991wavelets}
\begin{thm}\label{thm:stable-kernels}
	Suppose that $T$ is a Calderón-Zygmund operator associated to an antisymmetric standard kernel $K(x,y)$, thus $K(x,y)=-K(y,x)$. Given a Lipschitz function $B:\R\to\R$ and a $m\in \N$, there exists a Calderón-Zygmund operator $T_{B,m}$ associated to the kernel
	\begin{equation*}
		K_{B,m}(x,y)=K(x,y)\cdot \left(\frac{B(x)-B(y)}{x-y}\right)^m.
	\end{equation*}
\end{thm}

We will also need a result due to Tolsa \cite[Theorem 1.3]{tolsa2008uniform} concerning singular integrals on ``uniformly rectifiable sets''. We state a simplified version relevant to our application to $1$-dimensional Lipschitz graphs.
\begin{thm}\label{thm:Tolsa}
	Let $\mathcal{K} \in C^{2}(\R^d \, \setminus \, \{0\} )$ be an odd function satisfying  
	\begin{equation}\label{eq:K-standard}
		|\nabla^{j}\mathcal{K}(z)| \le C_\mathcal{K} |z|^{-1 - j}, \qquad j \in \{0,1,2\}. 
	\end{equation}
	Let $\Gamma(x)\coloneqq (x, A(x))$ for a Lipschitz function $A:\R\to\R^{d-1}$. Then, for all $\varepsilon>0$, the formula
	\begin{equation*}
		(T_\varepsilon f)(x) = \int_{|\Gamma(x)-\Gamma(y)|>\varepsilon} \mathcal{K}(\Gamma(x)-\Gamma(y)) f(y)\, dy, \qquad x \in \R,
	\end{equation*}
	defines a bounded operator $T_{\varepsilon}:L^2(\R)\to L^2(\R)$ with $\|T_\varepsilon\|_{L^2\to L^2} \lesssim_{C_\mathcal{K},\lip(A)} 1$.
\end{thm}
We use Theorems \ref{thm:stable-kernels} and \ref{thm:Tolsa} to prove the following.
\begin{proposition}\label{prop9}
	Let $\mathcal{K} \in C^{2}(\R^{d} \, \setminus \, \{0\})\cap L^\infty(\R^d)$ be an odd function satisfying \eqref{eq:K-standard}. Suppose $B_1,B_2 \colon \R \to \R$ are Lipschitz functions, $\Gamma(x)\coloneqq (x, A(x))$ for some Lipschitz function $A:\R\to\R^{d-1}$, and $m_1,m_2 \in \N$. Consider the kernel $K_{B,m} \colon (\R \times \R) \, \setminus \, \{(x,x) : x \in \R\} \to \C$ defined by
	\begin{equation}\label{form106} K_{B,m}(x,y) \coloneqq \mathcal{K}(\Gamma(x) - \Gamma(y)) \cdot \left(\frac{B_1(x) - B_1(y)}{x - y} \right)^{m_1}\left(\frac{B_2(x) - B_2(y)}{x - y} \right)^{m_2}\end{equation}
	for $x,y \in \R, \, x \neq y$. Then $K_{B,m}$ is a standard kernel, and the formula 
	\begin{equation*}
		(T_{B,m} f)(x) \coloneqq \int K_{B,m}(x,y)f(y)\, dy, \qquad f \in L^{2}(\R), \, x \in \R, 
	\end{equation*}
	defines a Calder\'on-Zygmund operator. The standard kernel constants and $\|T_{B,m}\|_{L^{2} \to L^{2}}$ only depend on $C_\mathcal{K}$, $\lip(A)$, $\lip(B_i)$ and $m_i$.
\end{proposition}
\begin{remark}
	Note that $K_{B,m} \in L^\infty(\R^{2})$. This fact together with the decay estimate \eqref{eq:K-standard} ensure that the integral appearing in the definition of $T_{B,m}f(x)$ is absolutely convergent.
\end{remark}
\begin{proof}[Proof of Proposition \ref{prop9}]
	First we check that $K_{B,m}$ is a standard kernel. The size condition \eqref{eq:SK-size} follows immediately from \eqref{eq:K-standard} with $j = 0$:
	\begin{align*}
		|K_{B,m}(x,y)| & \leq C_{\mathcal{K}}\lip(B_1)^{m_1}\lip(B_2)^{m_2} |\Gamma(x)-\Gamma(y)|^{-1}\\
		& \le C_{\mathcal{K}}\lip(B_1)^{m_1}\lip(B_2)^{m_2}|x-y|^{-1} \quad\text{for } x\neq y. \end{align*}
	To check the smoothness condition \eqref{eq:SK-smooth} (with $\delta = 1$) it suffices to prove \eqref{eq:SK-smooth2}. This follows by a straightforward, if tedious, computation, which also reveals that the constant "$C$" in \eqref{eq:SK-smooth2} only depends on $C_{\mathcal{K}},\lip(A),\lip(B_{1}),\lip(B_{2}),m_{1},m_{2}$. 

	Finally we prove that $T_{B,m}$ is a Calderón-Zygmund operator. First we consider the case $B_1=B_2= \mathrm{id}$, in which case $K_{B,m}(x,y)=K_{\mathrm{id}}(x,y)=\mathcal{K}(\Gamma(x)-\Gamma(y))$. Since $\mathcal{K}$ is an odd standard kernel satisfying \eqref{eq:K-standard}, Theorem \ref{thm:Tolsa} implies that for every $\varepsilon>0$ the corresponding truncated operator
	\begin{equation*}
		T_{\mathrm{id},\varepsilon} f(x) \coloneqq \int_{|\Gamma(x)-\Gamma(y)|>\varepsilon} \mathcal{K}(\Gamma(x)-\Gamma(y)) f(y)\, dy, \qquad x \in \R,
	\end{equation*}
	is bounded on $L^2(\R)$ with $\|T_{\mathrm{id},\varepsilon}\|_{L^{2} \to L^{2}} \lesssim_{C_{\mathcal{K}},\lip(A)} 1$. We wish to show that
	\begin{equation*}
		T_\mathrm{id} f(x) \coloneqq \int \mathcal{K}(\Gamma(x)-\Gamma(y)) f(y)\, dy
	\end{equation*}
	is still $L^2$-bounded with the same estimates. 		
	We fix $f\in L^2(\R)$ and use the Cauchy-Schwarz inequality to estimate
	\begin{multline*}
		\|T_\mathrm{id}f-T_{\mathrm{id},\varepsilon} f\|_{L^2}^2 = \int_\R \left| \int_{|\Gamma(x)-\Gamma(y)|\le\varepsilon}\mathcal{K}(\Gamma(x)-\Gamma(y)) f(y)\, dy  \right|^2\, dx\\
		\le \|\mathcal{K}\|_{L^\infty}^2 \int_\R\int_{|x-y|\le C \varepsilon}|f(y)|^2\, dy\, dx
		= \|\mathcal{K}\|_{L^\infty}^2 \int_\R |f(y)|^2\int_{|x-y|\le C \varepsilon} dx\, dy \lesssim \varepsilon\|\mathcal{K}\|_{L^\infty}^2\|f\|_{L^2}^2,
	\end{multline*}
	so that
	\begin{equation}\label{eq:Tid-bdd}
		\|T_\mathrm{id}f\|_{L^2} \le \|T_{\mathrm{id},\varepsilon} f\|_{L^2}+ \|T_\mathrm{id}f-T_{\mathrm{id},\varepsilon} f\|_{L^2} \lesssim_{C_\mathcal{K},\lip(A)}\|f\|_{L^2} + \varepsilon^{1/2}\|\mathcal{K}\|_{L^\infty}\|f\|_{L^2}.
	\end{equation}
	Letting $\varepsilon\to 0$ we get that $T_\mathrm{id}$ is $L^2$-bounded with estimates depending only on $C_{\mathcal{K}}$ and $\lip(A)$. In particular, it is a Calderón-Zygmund operator.
	
	Now we assume that $B_1:\R\to\R$ is a general Lipschitz function, and $B_2=\mathrm{id}$. We denote the corresponding kernel and operator by $K_{B_1,m_1}$ and $T_{B_1,m_1}$. 
	
	Observe that the kernel $K_{\mathrm{id}}(x,y)$ studied before is antisymmetric, and the associated operator $T_\mathrm{id}$ is Calderón-Zygmund. This means we may use Theorem \ref{thm:stable-kernels} to conclude that there exists a Calderón-Zygmund operator $S_{B_1,m_1}$ associated to the kernel $K_{B_1,m_1}$, with $\|S_{B_1,m_1}\|_{L^2\to L^2}$ depending only on $\lip(B_1),$ $m_1$ and $\|T_\mathrm{id}\|_{L^2\to L^2}$.
	
	The $L^2$-boundedness of $S_{B_1,m_1}$ together with Cotlar's inequality \cite[Section 4.2.2]{grafakos2014modern} implies that all the truncated operators
	\begin{equation*}
		T_{B_1,m_1,\varepsilon}f(x)=\int_{|x-y|\ge\varepsilon} K_{B_1,m_1}(x,y)f(y)\, dy
	\end{equation*}
	are bounded on $L^2(\R)$ with estimates depending only on $C_\mathcal{K}$, $\lip(A)$, $\lip(B_1)$ and $m_1$. Now we may use the estimate $\|K_{B_1,m_1}\|_{L^\infty}\lesssim \|\mathcal{K}\|_{L^\infty}\lip(B_1)^{m_1}<\infty$ to argue similarly like we did for $T_{\mathrm{id}}$ above \eqref{eq:Tid-bdd} and conclude that $T_{B_1,m_1}$ is a bounded operator on $L^2(\R)$.
	
	Finally, to get that $T_{B,m}$ is a Calderón-Zygmund operator for general Lipschitz functions $B_1, B_2:\R\to\R$, we repeat the reasoning used above for $T_{B_1,m_1}$. The only difference is that we apply Theorem \ref{thm:stable-kernels} to the antisymmetric kernel $K=K_{B_1,m_1}$ instead of $K_{\mathrm{id}}$.
\end{proof}

\subsection{Uniform $L^{p} \to L^{p}$ bounds for the operators $T^{\Gamma,\epsilon,R}_{\beta}$}\label{s:L2Bounds} The purpose of this section is to show that the operators $T_{\beta}^{\Gamma,\epsilon,R}$ have bounded extensions to $L^{p}$, $p \in (1,\infty)$, with operator norms independent of the truncation parameters $\epsilon,R$. 

The steps are the following:
\begin{itemize}
\item[(i)] For $\Rea \beta \in \{0,1\}$, apply Proposition \ref{prop9} combined with the $L^{p}$-boundedness of purely imaginary fractional Laplacians $\Delta^{\beta}$, $\Rea \beta = 0$, see Corollary \ref{cor3}.
\item[(ii)] For $\Rea \beta \in (0,1)$, use complex interpolation, enabled by Proposition \ref{prop5}.
\end{itemize}

To begin with, fix $0 < \epsilon \leq 1 \leq R$, a Lipschitz function $A \colon \R \to \R^{d - 1}$ with graph map $\Gamma(x) = (x,A(x))$, and consider the following two operators $\partial U^{\Gamma,\epsilon,R}$ and $U^{\Gamma,\epsilon,R}\partial$:
\begin{equation}\label{partialU} (\partial U^{\Gamma,\epsilon,R})f(x) := \int \partial_{x}(x \mapsto \log_{\epsilon,R} |\Gamma(x) - \Gamma(y)|) f(y) \, dy, \qquad f \in \mathcal{S}(\R), \, x \in \R, \end{equation}
and
\begin{align} (U^{\Gamma,\epsilon,R} \partial)f(x) & := \int \log_{\epsilon,R} (|\Gamma(x) - \Gamma(y)|) f'(y) \, dy \notag\\
&\label{Upartial} = \int -\partial_{y}(y \mapsto \log_{\epsilon,R} |\Gamma(x) - \Gamma(y)|)f(y) \, dy, \qquad f \in \mathcal{S}(\R), \, x \in \R. \end{align} 
In both formulae, $\log_{\epsilon,R} \in C^{4}(\R \, \setminus \, \{0\}) \cap \mathrm{Lip}(\R)$ is the bounded compactly supported "truncation" of the logarithm provided by Lemma \ref{truncationLemma}. The integration by parts in \eqref{Upartial} is justified by these properties of $\log_{\epsilon,R}$, and since $\Gamma$ is Lipschitz. Thanks to the truncations, both integrals \eqref{partialU}-\eqref{Upartial} are absolutely convergent.

\begin{remark}\label{rem3} What is the connection between the operators $\partial U^{\Gamma,\epsilon,R},U^{\Gamma,\epsilon,R}\partial$ and the operators $T^{\Gamma,\epsilon,R}_{\beta}$? The answer is  that 
\begin{displaymath} \partial U^{\Gamma,\epsilon,R} \cong T_{0}^{\Gamma,\epsilon,R} \quad \text{and} \quad U^{\Gamma,\epsilon,R} \partial \cong T_{1}^{\Gamma,\epsilon,R}, \end{displaymath}
where "$\cong$" means "modulo bounded operators on $L^{p}$". For a more precise and general statement, see Proposition \ref{prop12}. To reach that point, however, we first need to establish the $L^{p}$-boundedness of $\partial U^{\Gamma,\epsilon,R}$ and $U^{\Gamma,\epsilon,R}\partial$. \end{remark}

Let us compute explicitly the $\partial_{x}$ and $\partial_{y}$ derivatives of $\log_{\epsilon,R} |\Gamma(x) - \Gamma(y)|$ appearing in \eqref{partialU}-\eqref{Upartial}. The answers are:
\begin{equation}\label{form134} K_{\epsilon,R}(x,y) := k_{\epsilon,R}(|\Gamma(x) - \Gamma(y)|) \cdot \frac{(x - y) + \nabla A(x) \cdot (A(x) - A(y))}{|\Gamma(x) - \Gamma(y)|}, \qquad x,y \in \R, \end{equation} 
and
\begin{displaymath} L_{\epsilon,R}(x,y) := k_{\epsilon,R}(|\Gamma(x) - \Gamma(y)|) \cdot \frac{(x - y) + \nabla A(y) \cdot (A(x) - A(y))}{|\Gamma(x) - \Gamma(y)|}, \qquad x,y \in \R. \end{displaymath}
(For accuracy: the kernel $L_{\epsilon,R}(x,y)$ also incorporates the minus sign in \eqref{Upartial}.) The function $k_{\epsilon,R} \in C^{3}(\R \, \setminus \, \{0\}) \cap L^{\infty}(\R)$ is the bounded compactly supported standard kernel introduced in Lemma \ref{truncationLemma}.

We plan to deduce the $L^{p}$-boundedness of the operators $\partial U^{\Gamma,\epsilon,R}$ and $U^{\Gamma,\epsilon,R}\partial$ directly from Proposition \ref{prop9}. To this end, it suffices to note that both $K_{\epsilon,R}$ and $L_{\epsilon,R}$ can be expressed (up to $L^{\infty}$-factors) as the sum of $d$ kernels of the form \eqref{prop9}. In fact,
\begin{displaymath} K_{\epsilon,R}(x,y) = \mathcal{K}(\Gamma(x) - \Gamma(y)) + \sum_{j = 1}^{d - 1} A_{j}'(x) \cdot \mathcal{K}(\Gamma(x) - \Gamma(y)) \cdot \frac{A_{j}(x) - A_{j}(y)}{x - y}, \end{displaymath} 
where $\mathcal{K}(z) = k_{\epsilon,R}(|z|) \cdot z_{1}/|z|$ for $z = (z_{1},\ldots,z_{d}) \in \R^{d}$, and $A = (A_{1},\ldots,A_{d - 1})$. Note that $\mathcal{K}$ is an \textbf{odd} and bounded standard kernel (with standard kernel constants independent of $\epsilon,R$). This is needed for applying Proposition \ref{prop9}. With the same notation,
\begin{displaymath} L_{\epsilon,R}(x,y) = \mathcal{K}(\Gamma(x) - \Gamma(y)) + \sum_{j = 1}^{d - 1} A_{j}'(y) \cdot \mathcal{K}(\Gamma(x) - \Gamma(y)) \cdot \frac{A_{j}(x) - A_{j}(y)}{x - y}. \end{displaymath} 
We now deduce the following corollary of Proposition \ref{prop9} (and Theorem \ref{LpExtension}):
\begin{cor}\label{cor4} The operators $\partial U^{\Gamma,\epsilon,R}$ and $U^{\Gamma,\epsilon,R}\partial$ are bounded on $L^{p}(\R)$ for every $p \in (1,\infty)$. The $L^{p}$-operator norms depend on $\mathrm{Lip}(A),d$ and $p$ but not on $\epsilon,R$.
\end{cor}

\begin{remark} Taking into account that the kernels in \eqref{partialU}-\eqref{Upartial} are bounded and compactly supported, the pointwise definitions appearing on these lines are well-posed for $f \in L^{p}(\R)$, $p \in [1,\infty]$. This is why there is no need to talk about $L^{p}$-\emph{extensions} of the operators appearing in Corollary \ref{cor4}.   \end{remark} 

The next step is to formalise the connection between the operators $T^{\Gamma,\epsilon,R}_{\beta}$ with $\Rea \beta \in \{0,1\}$ and $\partial U^{\Gamma,\epsilon,R},U^{\Gamma,\epsilon,R}\partial$, as hinted at in Remark \ref{rem3}. We start with an easy computation of the Fourier transform of $(\partial U^{\Gamma,\epsilon,R})f$:

\begin{lemma}\label{lemma9} Let $f \in L^{2}(\R)$. Then $(\partial U^{\Gamma,\epsilon,R})f \in L^{2}(\R)$ equals the distributional derivative of $U^{\Gamma,\epsilon,R}f \in L^{2}(\R)$. In particular, the ($L^{2}$-)Fourier transform of $(\partial U^{\Gamma,\epsilon,R})f$ is the function
\begin{displaymath} \xi \mapsto (2\pi i \xi) \widehat{U^{\Gamma,\epsilon,R}f}(\xi). \end{displaymath}   
Here $U^{\Gamma,\epsilon,R}f \in L^{2}(\R)$ is defined by \eqref{form55}.
\end{lemma} 

\begin{proof} For $f \in \mathcal{S}(\R)$, the formula 
\begin{displaymath} \int \partial U^{\Gamma,\epsilon,R}f \cdot g = - \int U^{\Gamma,\epsilon,R}f \cdot g, \qquad g \in \mathcal{S}(\R), \end{displaymath}
is straightforward to check from the definition \eqref{partialU}. For $f \in L^{2}(\R)$, the formula stays valid because both operators $\partial U^{\Gamma,\epsilon,R}$ and $U^{\Gamma,\epsilon,R}$ are bounded on $L^{2}(\R)$. \end{proof} 

Regarding $(U^{\Gamma,\epsilon,R}\partial)f$, we need the following consistency result:

\begin{lemma}\label{lemma8} Let $f \in L^{2}(\R) \cap \dot{H}^{1}$. Then 
\begin{equation}\label{form101} (U^{\Gamma,\epsilon,R}\partial)f = U^{\Gamma,\epsilon,R}(\partial f). \end{equation}
In particular, the ($L^{2}$-)Fourier transforms of $(U^{\Gamma,\epsilon,R}\partial)f$ and $U^{\Gamma,\epsilon,R}(\partial f)$ coincide. \end{lemma} 

\begin{proof} For $f \in \mathcal{S}(\R)$, \eqref{form101} follows from the definition of $U^{\Gamma,\epsilon,R}\partial$, recall \eqref{Upartial}. Assume next that $f \in L^{2}(\R) \cap \dot{H}^{1}$. By the density of $\mathcal{S}(\R)$ in $ L^2(\R) \cap \dot{H}^1$ (see \cite[Proposition 1, p. 122]{MR290095}) there exists $\{f_{j}\}_{j \in \N} \subset \mathcal{S}(\R)$ such that $f_{j} \to f$ in $L^{2}$ and $\partial f_{j} \to \partial f$ in $L^{2}(\R)$.  Then
\begin{displaymath} (U^{\Gamma,\epsilon,R}\partial)f = \lim_{j \to \infty} (U^{\Gamma,\epsilon,R}\partial)f_{j} = \lim_{j \to \infty} U^{\Gamma,\epsilon,R}(\partial f_{j}) = U^{\Gamma,\epsilon,R}(\partial f), \end{displaymath} 
where the limits are taken in $L^{2}$. \end{proof} 

We finally arrive at the statement which relates $T^{\Gamma,\epsilon,R}_{\beta}$ to $\partial U^{\Gamma,\epsilon,R}$ and $U^{\Gamma,\epsilon,R}\partial$. To understand the statement, one should recall from Corollary \ref{cor3} that the Fourier multipliers with symbols $|\xi|^{\beta}$ and $|\xi|^{1 - \beta}/\xi$ define bounded operators on $L^{p}(\R)$ for $\Rea \beta = 0$. 

\begin{proposition}\label{prop12} Let $f,g \in \mathcal{S}(\R)$, $p \in (1,\infty)$, and $\beta \in \C$ with $\Rea \beta \in \{0,1\}$. Then,
\begin{equation}\label{form103} \left| \int T_{\beta}^{\Gamma,\epsilon,R}f \cdot g\right| \leq \|(\partial U^{\Gamma,\epsilon,R})(\Delta^{\beta/2}f)\|_{L^{p}}\|(\partial^{-1}\Delta^{(1 - \beta)/2})g\|_{L^{p'}}, \qquad \Rea \beta = 0, \end{equation}
\begin{equation}\label{form104} \left| \int T_{\beta}^{\Gamma,\epsilon,R}f \cdot g\right| \leq \|(U^{\Gamma,\epsilon,R}\partial)(\partial^{-1}\Delta^{\beta/2}f)\|_{L^{p}}\|\Delta^{(1 - \beta)/2}g\|_{L^{p'}}, \qquad \Rea \beta = 1. \end{equation}
Clarifications: 
\begin{itemize}
\item In \eqref{form103} the notation $(\partial U^{\Gamma,\epsilon,R})(\Delta^{\beta/2}f)$ refers to $\partial U^{\Gamma,\epsilon,R}$ acting on $\Delta^{\beta/2}f \in L^{p}(\R)$, and the notation $(\partial^{-1}\Delta^{(1 - \beta)/2})g$ refers to the $L^{p'}$-function with Fourier transform $\xi \mapsto |\xi|^{1 - \beta}/(2\pi i \xi)\widecheck{g}(\xi)$.
\item In \eqref{form104}, the notation $(U^{\Gamma,\epsilon,R}\partial)(\partial^{-1}\Delta^{\beta/2}f)$ refers to $U^{\Gamma,\epsilon,R}\partial$ acting on the $L^{p}$-function with Fourier transform $\xi \mapsto |\xi|^{\beta}/(2\pi i \xi)\hat{f}(\xi)$. 
\end{itemize} 

\end{proposition} 

\begin{proof} We start with \eqref{form103}. Besides using definitions, Plancherel, and H\"older's inequality, the only non-trivial step is to apply Lemma \ref{lemma9} to the function $\Delta^{\beta/2}f \in L^{2}(\R)$:
\begin{align*} \left| \int T^{\Gamma,\epsilon,R}_{\beta}f \cdot g \right| & \stackrel{\mathrm{def.}}{=}  \left|\int \mathcal{F}[U^{\Gamma,\epsilon,R}(\Delta^{\beta/2}f)](\xi) \cdot |\xi|^{1 - \beta}\widecheck{g}(\xi) \, d\xi \right| \\
& = \left|\int (2\pi i\xi) \mathcal{F}[U^{\Gamma,\epsilon,R}(\Delta^{\beta/2}f)](\xi) \cdot \frac{|\xi|^{1 - \beta}}{2\pi i \xi}\widecheck{g}(\xi) \, d\xi \right|\\
& \stackrel{\mathrm{L.\,} \ref{lemma9}}{=} \left| \int \mathcal{F}[(\partial U^{\Gamma,\epsilon,R})(\Delta^{\beta/2}f)](\xi) \cdot \frac{|\xi|^{1 - \beta}}{2\pi i \xi}\widecheck{g}(\xi) \, d\xi \right|\\
& \leq \left| \int (\partial U^{\Gamma,\epsilon,R})(\Delta^{\beta/2}f)(x) (\partial^{-1}\Delta^{(1 - \beta)/2})g(x) \, dx \right|\\
& \leq \|(\partial U^{\Gamma,\epsilon,R})(\Delta^{\beta/2}f)\|_{L^{p}}\|(\partial^{-1}\Delta^{(1 - \beta)/2})g\|_{L^{p'}}.
\end{align*}
We then move to the proof of \eqref{form104}, where $\Rea \beta = 1$. This time, the only non-trivial step is to apply Lemma \ref{lemma8} to the function $\partial^{-1}\Delta^{\beta/2}f \in L^{2}(\R) \cap \dot{H}^{1}$:
\begin{align*} \left| \int T^{\Gamma,\epsilon,R}_{\beta}f \cdot g \right| & \stackrel{\mathrm{def.}}{=} \left|\int \mathcal{F}[U^{\Gamma,\epsilon,R}(\Delta^{\beta/2}f)](\xi) \cdot |\xi|^{1 - \beta}\widecheck{g}(\xi) \, d\xi \right| \\
& \stackrel{\mathrm{L.\,} \ref{lemma8}}{=} \left|\int \mathcal{F}[(U^{\Gamma,\epsilon,R}\partial)(\partial^{-1}\Delta^{\beta/2}f)](\xi) \cdot |\xi|^{1 - \beta}\widecheck{g}(\xi) \, d\xi \right| \\
& = \left| \int (U^{\Gamma,\epsilon,R}\partial)(\partial^{-1}\Delta^{\beta/2}f)(x) \cdot \Delta^{(1 - \beta)/2}g(x) \, dx \right|\\
& \leq \|(U^{\Gamma,\epsilon,R}\partial)(\partial^{-1}\Delta^{\beta/2}f)\|_{L^{p}}\|\Delta^{(1 - \beta)/2}g\|_{L^{p'}}. \end{align*} 
This concludes the proof. \end{proof} 

We derive the following corollary:

\begin{cor}\label{cor5} Let $0 < \epsilon \leq 1 \leq R$, $\Rea \beta \in \{0,1\}$, and $f \in \mathcal{S}(\R)$. Then, the tempered distribution $T_{\beta}^{\Gamma,\epsilon,R}f$ from Definition \ref{def2} is an element of $L^{p}$ for every $p \in (1,\infty)$. Moreover,
\begin{equation}\label{form107} \|T_{\beta}^{\Gamma,\epsilon,R}f\|_{L^{p}} \lesssim_{\mathrm{Lip}(A),d,p}(1 + (\mathrm{Im\,} \beta)^{2})^{2} \cdot \|f\|_{L^{p}}, \qquad f \in \mathcal{S}(\R). \end{equation} 
\end{cor}

\begin{proof} Fix $p \in (1,\infty)$ and consider first the case $\Rea \beta = 0$. Let $T$ be the Fourier multiplier operator with symbol $m = |\xi|^{1 - \beta}/(2\pi i\xi)$. Then, from the inequality \eqref{form103}, $L^{p}$-duality, and Corollary \ref{cor4}, we see that 
\begin{displaymath} \|T_{\beta}^{\Gamma,\epsilon,R}f\|_{L^{p}} \lesssim_{\mathrm{Lip}(A),d,p} \|\Delta^{\beta/2}f\|_{L^{p}}\|T\|_{L^{p'} \to L^{p'}}. \end{displaymath} 
Since $\Rea \beta = 0$, it now follows from Corollary \ref{cor3} that 
\begin{displaymath} \|\Delta^{\beta/2}f\|_{L^{p}} \lesssim_{p} (1 + (\mathrm{Im\,}\beta)^{2})\|f\|_{L^{p}} \quad \text{and} \quad \|T\|_{L^{p'} \to L^{p'}} \lesssim_{p'}(1 + (\mathrm{Im\,} \beta)^{2}). \end{displaymath}
This yields \eqref{form107} in the case $\Rea \beta = 0$. Moving on to $\Rea \beta = 1$, a similar argument, now applying \eqref{form104} instead of \eqref{form103}, shows that
\begin{displaymath} \|T_{\beta}^{\Gamma,\epsilon,R}f\|_{L^{p}} \lesssim_{\mathrm{Lip}(A),d,p} \|\partial^{-1}\Delta^{\beta/2}f\|_{L^{p}}\|\Delta^{(1 - \beta)/2}g\|_{L^{p'}}. \end{displaymath}
Applying Corollary \ref{cor3} to the factors on the right, we deduce further that
\begin{displaymath} \|\partial^{-1}\Delta^{\beta/2}f\|_{L^{p}} \lesssim_{p}(1 + (\mathrm{Im\,} \beta)^{2})\|f\|_{L^{p}}, \end{displaymath}
and
\begin{displaymath} \|\Delta^{(1 - \beta)/2}g\|_{L^{p'}} \lesssim_{p'} (1 + (\mathrm{Im\,} \beta)^{2})\|g\|_{L^{p'}}. \end{displaymath} 
Combining these estimates completes the proof of \eqref{form107} in the case $\Rea \beta = 1$.  \end{proof}

We finally deal with the general case $\Rea \beta \in [0,1]$. The ingredients are Corollary \ref{cor5}, the analyticity of the map $\beta \mapsto T^{\Gamma,\epsilon,R}_{\beta}$ established in Proposition \ref{prop5}, and the \emph{three lines lemma}, quoted below (see \cite[Lemma 1.3.5]{MR3243734}):
\begin{lemma}[Three lines lemma]\label{l:threeLines} Let $F$ be analytic in the open strip $S = \{\beta \in \C : \Rea \beta \in (0,1)\}$, continuous and bounded in the closure $\bar{S}$, and such that $|F(\beta)| \leq B$ for $\Rea \beta \in \{0,1\}$, where $B > 0$. Then $|F(\beta)| \leq B$ for $\beta \in \bar{S}$.
\end{lemma}

\begin{cor}\label{cor6}  Let $0 < \epsilon \leq 1 \leq R < \infty$, $\Rea \beta \in [0,1]$, and $f \in \mathcal{S}(\R)$. Then, the tempered distribution $T_{\beta}^{\Gamma,\epsilon,R}f$ from Definition \ref{def2} is an element of $L^{p}$ for every $p \in (1,\infty)$, and
\begin{displaymath} \|T^{\Gamma,\epsilon,R}_{\beta}f\|_{L^{p}} \lesssim_{\mathrm{Lip}(A),d,p} e^{(\mathrm{Im\,} \beta)^{2}} \|f\|_{L^{p}}. \end{displaymath}
\end{cor}

\begin{proof} Fixing $f,g \in \mathcal{S}(\R)$, consider the function $F \colon \bar{S} \to \C$ defined by
\begin{displaymath} F(\beta) := e^{\beta^{2}} \int T_{\beta}^{\Gamma,\epsilon,R}f \cdot g, \qquad \beta \in \bar{S}.  \end{displaymath} 
It suffices to check that
\begin{equation}\label{form110} |F(\beta)| \lesssim_{\mathrm{Lip}(A),d,p} \lesssim \|f\|_{L^{p}}\|g\|_{L^{p'}}, \qquad p \in (1,\infty). \end{equation}

We plan to apply Lemma \ref{l:threeLines} to $F$. The analyticity of $F$ on $S$ and continuity on $\bar{S}$ follows from Proposition \ref{prop5}. Additionally, $F$ bounded on $\bar{S}$. This follows from the calculations
\begin{equation}\label{form108} |e^{\beta^{2}}| \leq e^{1-(\mathrm{Im\,} \beta)^{2}} \lesssim 1, \qquad \beta \in \bar{S}, \end{equation}
and
\begin{align*} |F(\beta)| & \lesssim \left| \int T^{\Gamma,\epsilon,R}_{\beta} f \cdot g \right| \stackrel{\eqref{form56}}{\lesssim_{\epsilon,R}} \|\Delta^{\beta/2}f\|_{L^{2}}\|\Delta^{(1 - \beta)/2}g\|_{L^{2}} \lesssim_{f,g} 1, \qquad \beta \in \bar{S}.\end{align*} 
Here we used the $L^{2}$-boundedness of the truncated logarithmic potential $U^{\Gamma,\epsilon,R}$. The $L^{2}$-norms $\|\Delta^{\beta/2}f\|_{L^{2}}$ and $\|\Delta^{(1 - \beta)/2}g\|_{L^{2}}$ were finally estimated by Plancherel, and noting that $||\xi|^{\beta}| \leq |\xi|^{\Rea \beta} \leq 1 + |\xi|$ for $\xi \in \R$ and $\Rea \beta \in [0,1]$.

We have now verified that $F$ satisfies the "qualitative" hypotheses of Lemma \ref{l:threeLines}. To check the quantitative hypothesis, fix $p \in (1,\infty)$ and $\beta \in \bar{S}$ with $\Rea \beta \in \{0,1\}$. Write $B := \|f\|_{L^{p}}\|g\|_{L^{p'}}$. Then, by combining Corollary \ref{cor5} and \eqref{form108}, we obtain
\begin{equation}\label{form109} |F(\beta)| \lesssim_{\mathrm{Lip}(A),d,p} e^{-(\mathrm{Im\,} \beta)^{2}} (1 + (\mathrm{Im\,} \beta)^{2})^{2} \cdot B. \end{equation}
Since $e^{-(\mathrm{Im\,} \beta)^{2}}(1 + (\mathrm{Im\,} \beta)^{2}) \lesssim 1$, we conclude from \eqref{form109} that $|F(\beta)| \lesssim_{\mathrm{Lip}(A),d,p} B$ for $\Rea \beta \in \{0,1\}$. The three lines lemma now implies \eqref{form110}, and the proof is complete. \end{proof}

\subsection{Uniform $L^{p} \to L^{p}$ bounds for the differences $T^{\Gamma_{t},\epsilon,R}_{\beta} - T^{\Gamma_{s},\epsilon,R}_{\beta}$}\label{s:LpDifferences} In this section most arguments are adaptations of \cite[Section 9]{MR4649162}. In particular, we need the following elementary lemma, taken from \cite[Lemma 9.14]{MR4649162}: 

\begin{lemma}\label{l:metric} Let $(X,d)$ be a metric space, and let $\gamma \colon [0,1] \to X$ be a map satisfying 
\begin{displaymath} \limsup_{r \to s} \frac{d(\gamma(s),\gamma(r))}{|r - s|} \leq C, \qquad 0 \leq s \leq 1. \end{displaymath}
Then $\gamma$ is $C$-Lipschitz.
\end{lemma}
 
\subsubsection{$L^{p}$ bounds for $\beta = 0$} Let $A \colon \R \to \R^{d - 1}$ be Lipschitz. For $s \in \R$, define $A_{s}(x) := sA(x)$. For $0 < \epsilon \leq 1 \leq R < \infty$ and $s \in [0,1]$, fixed let $\partial U^{\Gamma_{s},\epsilon,R}$ be the operator 
\begin{equation}\label{form99} (\partial U^{\Gamma_{s},\epsilon,R})f(x) := \int \partial_{x} (x \mapsto \log_{\epsilon,R} |\Gamma_{s}(x) - \Gamma_{s}(y)|) f(y) \, dy, \qquad f \in \mathcal{S}(\R), \, x \in \R. \end{equation}
Here $\Gamma_{s}(x) = (x,A_{s}(x))$, and $\log_{\epsilon,R}$ is the function from Lemma \ref{truncationLemma}. Since $\log_{\epsilon,R} \in \mathrm{Lip}(\R)$, the kernel $\partial_{x}(x \mapsto \log_{\epsilon,R} |\Gamma_{s}(x) - \Gamma_{s}(y)|)$ equals a.e. the compactly supported $L^{\infty}$-function 
\begin{equation}\label{form92} K_{s,\epsilon,R}(x,y) := k_{\epsilon,R}(|\Gamma_{s}(x) - \Gamma_{s}(y)|) \frac{(x - y) + s^{2}\nabla A(x) \cdot (A(x) - A(y))}{|\Gamma_{s}(x) - \Gamma_{s}(y)|}, \quad x,y \in \R. \end{equation} 

We start by establishing the following:

 \begin{proposition}\label{prop7} Assume that $\mathrm{Lip}(A) \leq 1$. Then, for $0 < \epsilon \leq 1 \leq R < \infty$, $s,t \in [0,1]$, and $p \in (1,\infty)$ fixed,
 \begin{displaymath} \|\partial U^{\Gamma_{s},\epsilon,R} - \partial U^{\Gamma_{t},\epsilon,R}\|_{L^{p} \to L^{p}} \lesssim_{d,p} \mathrm{Lip}(A) \cdot |s - t|. \end{displaymath}
 \end{proposition} 
 The idea of the proof is morally to show that $s \mapsto \partial U^{\Gamma_{s},\epsilon,R}$ is differentiable $\R \to \mathcal{B}(L^{p},L^{p})$ with bounded derivative. This will be formally justified via Lemma \ref{l:metric}. We start by introducing the "$s$-derivative" operator" $\partial_{s} \partial U^{\Gamma_{s},\epsilon,R}$ defined by 
 \begin{displaymath} \partial_{s}\partial U^{\Gamma_{s},\epsilon,R}f(x) := \int \partial_{s} K_{s,\epsilon,R}(x,y)f(y) \, dy, \qquad f \in \mathcal{S}(\R), \, x \in \R. \end{displaymath} 
We note that whenever $\nabla A(x)$ exists and $y \neq x$, the function $s \mapsto K_{s,\epsilon,R}(x,y)$ is differentiable, and $\partial_{s}K_{s,\epsilon,R}(x,y)$ refers to this $s$-derivative. The plan is to show that $\partial_{s} \partial U^{\Gamma_{s},\epsilon,R}$ is bounded on $L^{p}$, see Corollary \ref{cor1}. This will be a consequence of Proposition \ref{prop9} (combined with Theorem \ref{LpExtension}), and to this end will need to compute the kernel $\partial_{s}K_{s,\epsilon,R}(x,y)$ explicitly. In view of \eqref{form92}, this is an exercise in using the product rule. First, 
\begin{displaymath} \partial_{s}(s \mapsto k_{\epsilon,R}(|\Gamma_{s}(x) - \Gamma_{s}(y)|)) = k_{\epsilon,R}'(|\Gamma_{s}(x) - \Gamma_{s}(y)|)\frac{\sum_{j = 1}^{d - 1} (A_{j}(x) - A_{j}(y))^{2}}{2|\Gamma_{s}(x) - \Gamma_{s}(y)|}. \end{displaymath}
second,
\begin{displaymath} \partial_{s}\left(s \mapsto \frac{1}{|\Gamma_{s}(x) - \Gamma_{s}(y)|}\right) = -\sum_{j = 1}^{d}\frac{(A_{j}(x) - A_{j}(y))^{2}}{2|\Gamma_{s}(x) - \Gamma_{s}(y)|^{3}}, \end{displaymath} 
and third $\partial_{s}(s \mapsto s^{2}\nabla A(x) \cdot (A(x) - A(y))) = 2s \nabla A(x) \cdot (A(x) - A(y))$. Therefore,
\begin{align} \partial_{s} K_{s,\epsilon,R}(x,y) & = k_{\epsilon,R}'(|\Gamma_{s}(x) - \Gamma_{s}(y)|) \cdot \frac{(x - y)\sum_{j = 1}^{d - 1}(A_{j}(x) - A_{j}(y))^{2}}{2|\Gamma_{s}(x) - \Gamma_{s}(y)|^{2}} \notag\\
& \quad - k_{\epsilon,R}(|\Gamma_{s}(x) - \Gamma_{s}(y)|) \cdot \frac{(x - y)\sum_{j = 1}^{d - 1}(A_{j}(x) - A_{j}(y))^{2}}{2|\Gamma_{s}(x) - \Gamma_{s}(y)|^{3}} \notag\\
& \quad + k_{\epsilon,R}'(|\Gamma_{s}(x) - \Gamma_{s}(y)|) \cdot \frac{s^{2}\nabla A(x) \cdot (A(x) - A(y))\sum_{j = 1}^{d - 1} (A_{j}(x) - A_{j}(y))^{2}}{2|\Gamma_{s}(x) - \Gamma_{s}(y)|^{2}} \notag\\
& \quad - k_{\epsilon,R}(|\Gamma_{s}(x) - \Gamma_{s}(y)|) \cdot \frac{s^{2}\nabla A(x) \cdot (A(x) - A(y)) \sum_{j = 1}^{d - 1} (A_{j}(x) - A_{j}(y))^{2}}{2|\Gamma_{s}(x) - \Gamma_{s}(y)|^{3}} \notag\\
&\label{form93} \quad + k_{\epsilon,R}(|\Gamma_{s}(x) - \Gamma_{s}(y)|) \cdot \frac{2s \nabla A(x) \cdot (A(x) - A(y))}{|\Gamma_{s}(x) - \Gamma_{s}(y)|}. \end{align}
The lessons of this computation are recorded in the following proposition:
\begin{proposition}\label{prop6} For $s \in [0,1]$, the kernel $\partial_{s} K_{s,\epsilon,R}(x,y)$ can be written as a sum of $\lesssim_{d} 1$ (precisely $(d - 1)(2d + 1)$) bounded kernels of the form
\begin{equation}\label{form94} (x,y) \mapsto \mathrm{Lip}(A)^{m} \cdot B(x) \cdot \mathcal{K}(\Gamma_{s}(x) - \Gamma_{s}(y)) \cdot \left(\tfrac{L_{1}(x) - L_{1}(y)}{x - y} \right)^{n_{1}}\left(\tfrac{L_{2}(x) - L_{2}(y)}{x - y} \right)^{n_{2}}, \end{equation}  
where
\begin{itemize}
\item $\mathcal{K} \in C^{2}_{c}(\R^{d} \, \setminus \, \{0\})$ is an odd bounded standard kernel in the sense of Definition \ref{def:standardKernel},
\item $L_{0},L_{1} \colon \R \to \R$ are $1$-Lipschitz,
\item $B \in \{1,L_{1}'\} \subset L^{\infty}(\R)$,
\item $m,n_{1},n_{2} \in \N$ with $m \geq 1$.
\end{itemize}
Moreover, the standard kernel constants of $\mathcal{K}$ are absolute. 
\end{proposition}

\begin{remark} The most non-trivial points of Proposition \ref{prop6} are that the kernels $\mathcal{K}$ are all odd, and the factor $\mathrm{Lip}(A)$ appears with a non-zero exponent. These facts will be crucial to obtain the correct $L^{p}$-bounds for the operator associated with $\partial_{s}K_{s,\epsilon,R}(x,y)$. \end{remark} 

\begin{proof}[Proof of Proposition \ref{prop6}] The kernels can be read off from the five lines leading to \eqref{form93}. In fact, the standard kernel $\mathcal{K} \in C^{2}_{c}(\R^{d} \, \setminus \, \{0\})$ is common for the kernels arising from each line, and is given by the following expressions, where $z = (z_{1},\ldots,z_{d}) \in \R^{d} \, \setminus \, \{0\}$:
\begin{align*} & \mathcal{K}_{1}(z) := \tfrac{1}{2}(z_{1} \cdot k_{\epsilon,R}'(|z|)) \cdot \frac{z_{1}^{2}}{|z|^{2}}, \\
& \mathcal{K}_{2}(z) := -\tfrac{1}{2}k_{\epsilon,R}(|z|) \cdot \frac{z_{1}^{3}}{|z|^{3}}, \\
& \mathcal{K}_{3}(z) := \tfrac{s^{2}}{2}(z_{1} \cdot k_{\epsilon,R}'(|z|)) \cdot \frac{z_{1}^{2}}{|z|^{2}}, \\
& \mathcal{K}_{4}(z) := -\tfrac{s^{2}}{2}k_{\epsilon,R}(|z|) \cdot \frac{z_{1}^{3}}{|z|^{3}},\\ 
& \mathcal{K}_{5}(z) := 2sk_{\epsilon,R}(|z|) \cdot \frac{z_{1}}{|z|}. \end{align*}
The oddness of the kernels $\mathcal{K}_{j}$ is evident, and the standard kernel properties follow from Lemma \ref{truncationLemma}(iv) plus the product rule. The main point is that $z \mapsto k_{\epsilon,R}(|z|)$ and $z \mapsto z_{1} \cdot k_{\epsilon,R}'(|z|)$ are both bounded standard kernels in $C_{c}^{2}(\R^{d} \, \setminus \, \{0\})$, and multiplication by factor of the form $(z_{1}/|z|)^{j}$ does not change this property.

With this notation, it is routine to check that the kernels above \eqref{form93} can be expressed in the form \eqref{form94}. Just to given an example, the kernel appearing exactly on line \eqref{form93} is evidently the sum of $(d - 1)$ kernels of the form
\begin{align*} & k_{\epsilon,R}(|\Gamma_{s}(x) - \Gamma_{s}(y)|) \cdot \frac{2sA_{i}'(x)(A_{i}(x) - A_{i}(y))}{|\Gamma_{s}(x) - \Gamma_{s}(y)|}\\
& = \mathrm{Lip}(A)^{2} \cdot \left(\tfrac{A_{i}}{\mathrm{Lip}(A)}\right)'(x) \cdot \left( k_{\epsilon,R}(|\Gamma_{s}(x) - \Gamma_{s}(y)| \cdot \frac{2s(x - y)}{|\Gamma_{s}(x) - \Gamma_{s}(y)|} \right) \cdot \frac{\tfrac{A_{i}}{\mathrm{Lip}(A)}(x) - \tfrac{A_{i}}{\mathrm{Lip}(A)}(y)}{x - y}.  \end{align*} 
At this point, one finally notes that the expression in middle brackets equals
\begin{displaymath} k_{\epsilon,R}(|\Gamma_{s}(x) - \Gamma_{s}(y)|) \cdot \frac{2s(x - y)}{|\Gamma_{s}(x) - \Gamma_{s}(y)|} = \mathcal{K}_{5}(\Gamma_{s}(x) - \Gamma_{s}(y)), \end{displaymath} 
since $\Gamma_{s}(x) - \Gamma_{s}(y) = (x - y,A_{s}(x) - A_{s}(y))$. We omit the remaining computations. \end{proof} 

\begin{cor}\label{cor1} Assume that $\mathrm{Lip}(A) \leq 1$. For $s \in [0,1]$ and $p \in (1,\infty)$, the operator $\partial_{s} \partial U^{\Gamma_{s},\epsilon,R}$ is bounded on $L^{p}(\R)$, and indeed
\begin{displaymath} \|\partial_{s} \partial U^{\Gamma_{s},\epsilon,R}\|_{L^{p} \to L^{p}} \lesssim_{d,p} \mathrm{Lip}(A). \end{displaymath}
\end{cor} 

\begin{proof} This follows from Proposition \ref{prop9} combined with Theorem \ref{LpExtension}, noting that the kernels of the form \eqref{form94} without the $\mathrm{Lip}(A)^{m}$-factor (with $m \geq 1$) satisfy the hypotheses of Proposition \ref{prop9} with constants independent of $\epsilon,R$ or $\mathrm{Lip}(A)$. The hypothesis $\mathrm{Lip}(A) \leq 1$ is needed to ensure that (a) $\mathrm{Lip}(A)^{m} \leq \mathrm{Lip}(A)$, and (b) the graphs $\Gamma_{s}$ appearing in $\mathcal{K}(\Gamma_{s}(x) - \Gamma_{s}(y))$ all have Lipschitz constant $\leq 1$. \end{proof} 

The next lemma explains the relation between $\partial_{s}\partial U^{\Gamma_{s},\epsilon,R}$ and $\partial U^{\Gamma_{s},\epsilon,R} - \partial U^{\Gamma_{t},\epsilon,R}$:

\begin{lemma}\label{lemma6} For $s \in [0,1]$ and $f,g \in \mathcal{S}(\R)$ fixed,
\begin{displaymath} \lim_{t \to s} \frac{1}{s - t} \left(\int \partial U^{\Gamma_{s},\epsilon,R}f \cdot g - \int \partial U^{\Gamma_{t},\epsilon,R}f \cdot g \right) = \int \partial_{s} \partial U^{\Gamma_{s},\epsilon,R}f \cdot g. \end{displaymath}
\end{lemma} 

\begin{proof} Fix $f \in \mathcal{S}(\R)$, $s \neq t$ and $x \in \R$ such that $\nabla A(x)$ exists (then $s \mapsto K_{s,\epsilon,R}(x,y)$ is differentiable for all $y \neq x$). Abbreviate $K_{s} := K_{s,\epsilon,R}$, and write
\begin{align}\label{form96} & \frac{\partial U^{\Gamma_{s},\epsilon,R}f(x) - \partial U^{\Gamma_{t},\epsilon,R}f(x)}{s - t} - \partial_{s} \partial U^{\Gamma_{s},\epsilon,R}f(x)\\
& \qquad = \frac{1}{s - t} \int_{t}^{s} \int \left[\partial_{r} K_{r}(x,y) - \partial_{s} K_{s}(x,y) \right] f(y) \, dy \, dr \notag\\
&\label{form95} \qquad = \int_{0}^{1} \int \left[ \partial_{r(s - t) + t} K_{r(s - t) + t}(x,y)  -\partial_{s} K_{s}(x,y) \right] f(y) \, dy \, dr.  \end{align} 
Now, recall from Proposition \ref{prop6} the explicit form of the kernels $\partial_{s} K_{s}(x,y)$. Using these expressions, one sees that
\begin{displaymath} |\partial_{r(s - t) + t} K_{r(s - t) + t}(x,y) - \partial_{s} K(x,y)| \lesssim_{d} |\mathcal{K}(\Gamma_{r(s - t) + t}(x) - \Gamma_{r(s - t) + t}(y)) - \mathcal{K}(\Gamma_{s}(x) - \Gamma_{s}(y))|, \end{displaymath} 
for each $(x,y) \in \R^{2}$, where $\mathcal{K} \in C^{2}_{c}(\R^{d} \, \setminus \, \{0\})$ is one of $\lesssim_{d} 1$ possible bounded and compactly supported standard kernels (the choice may depend on $(x,y)$). In particular, $|\partial_{r(s - t) + t} K_{r(s - t) + t}(x,y) - \partial_{s} K(x,y)| \lesssim_{A,d,\epsilon,R} 1$, where the bound is independent of $x,y$. For fixed $r \in [0,1]$ and $x \neq y$, the difference also tends to zero as $t \to s$, since
\begin{displaymath} \lim_{t \to s} (\Gamma_{r(t - s) + t}(x) - \Gamma_{r(t - s) + t}(y)) = \Gamma_{s}(x) - \Gamma_{s}(y) \neq 0. \end{displaymath}
From this information one first deduces that
\begin{displaymath} \lim_{t \to s}  \int [\partial_{r(s - t) + t}K_{r(s - t) + t}(x,y) - \partial_{s} K_{s}(x,y)]f(y) \, dy = 0, \qquad x \in \R, \, r \in [0,1]. \end{displaymath}
This implies, using the bounded convergence theorem, that the limit as $t \to s$ of the expressions on line \eqref{form95} equals zero. We have now shown that the limit (as $t \to s$) of the differences on line \eqref{form96} equals zero for every $x \in \R$ such that $\nabla A(x)$ exists. The argument also shows that these differences are uniformly bounded by $\lesssim_{A,d,\epsilon,R} \|f\|_{L^{1}}$. Since $\nabla A(x)$ exists a.e. we deduce from the dominated convergence theorem that 
\begin{displaymath} \lim_{t \to s} \int \left( \frac{\partial U^{\Gamma_{s},\epsilon,R}f(x) - \partial U^{\Gamma_{t},\epsilon,R}f(x)}{s - t} - \partial_{s} \partial U^{\Gamma_{s},\epsilon,R}f(x) \right) g(x) \, dx = 0 \end{displaymath} 
for all $g \in \mathcal{S}(\R)$, which is equivalent to the claim in the lemma. \end{proof} 

We may now prove Proposition \ref{prop7}.

\begin{proof}[Proof of Proposition \ref{prop7}] For $f,g \in \mathcal{S}(\R)$ fixed, define the path $\gamma_{f,g} \colon [0,1] \to \R$,
\begin{displaymath} \gamma_{f,g}(s) := \int \partial U^{\Gamma_{s},\epsilon,R}f \cdot g, \qquad s \in [0,1]. \end{displaymath}
From Lemma \ref{lemma6} and Corollary \ref{cor1}, we infer that (for $p \in (1,\infty)$)
\begin{displaymath} \limsup_{t \to s} \frac{|\gamma_{f,g}(s) - \gamma_{f,g}(t)|}{|s - t|} \leq \left| \int \partial_{s} \partial U^{\Gamma_{s},\epsilon,R}f \cdot g \right| \lesssim_{d,p} \mathrm{Lip}(A)\|f\|_{L^{p}}\|g\|_{L^{p'}}, \qquad s \in [0,1]. \end{displaymath} 
This implies by Lemma \ref{l:metric} that $\gamma_{f,g}$ is $C$-Lipschitz with $C \lesssim_{d,p} \mathrm{Lip}(A)\|f\|_{L^{p}}\|g\|_{L^{p'}}$. Consequently, for $f \in \mathcal{S}(\R)$ with $\|f\|_{L^{p}} = 1$,
\begin{align*} \|\partial U^{\Gamma_{s},\epsilon,R}f - \partial U^{\Gamma_{t},\epsilon,R}f\|_{L^{2}} & = \sup_{\|\varphi\|_{p'} = 1} \left| \int \partial U^{\Gamma_{s},\epsilon,R}f \cdot \varphi - \int \partial U^{\Gamma_{t},\epsilon,R}f \cdot \varphi \right|\\
& = \sup_{\|\varphi\|_{p'} = 1} |\gamma_{f,\varphi}(s) - \gamma_{f,\varphi}(t)| \lesssim_{d,p} \mathrm{Lip}(A) \cdot |s - t|.  \end{align*}  
This completes the proof of Proposition \ref{prop7}. \end{proof}

\subsubsection{$L^{p}$ bounds for $\beta = 1$} As in the previous section, let $A \colon \R \to \R^{d - 1}$ be Lipschitz, and $A_{s}(x) := sA(x)$ for $s,x \in \R$. The purpose of this section is to establish an analogue of Proposition \ref{prop7} for the operators $U^{\Gamma_{s},\epsilon,R}\partial$ defined by
\begin{align} (U^{\Gamma_{s},\epsilon,R}\partial) f(x) & := \int \log_{\epsilon,R} |\Gamma_{s}(x) - \Gamma_{s}(y)| f'(y) \, dy \notag\\
&\label{form102} = \int -\partial_{y}(y \mapsto \log_{\epsilon,R} |\Gamma_{s}(x) - \Gamma_{s}(y)|) f(y) \, dy, \qquad f \in \mathcal{S}(\R), x \in \R. \end{align}
Both integrals above are absolutely convergent, and integration by parts is justified since $y \mapsto \log_{\epsilon,R} |\Gamma_{s}(x) - \Gamma_{s}(y)|$ is a compactly supported Lipschitz function. Explicitly, the kernel $-\partial_{y}(y \mapsto \log_{\epsilon,R} |\Gamma_{s}(x) - \Gamma_{s}(y)|)$ equals a.e. the compactly supported $L^{\infty}$-function 
\begin{equation}\label{form97} L_{s,\epsilon,R}(x,y) := k_{\epsilon,R}(|\Gamma_{s}(x) - \Gamma_{s}(y)|) \cdot \frac{(x - y) + s^{2}\nabla A(y) \cdot (A(x) - A(y))}{|\Gamma_{s}(x) - \Gamma_{s}(y)|}, \quad x,y \in \R. \end{equation} 
So, the only difference between the kernels $K_{s,\epsilon,R}(x,y)$ (from \eqref{form92}) and $L_{s,\epsilon,R}(x,y)$ is the one between the factors $\nabla A(x)$ and $\nabla A(y)$.  Here is the analogue of Proposition \ref{prop7}:

 \begin{proposition}\label{prop10} Assume $\mathrm{Lip}(A) \leq 1$. For $0 < \epsilon \leq 1 \leq R < \infty$, $s,t \in [0,1]$, and $p \in (1,\infty)$,
 \begin{displaymath} \|U^{\Gamma_{s},\epsilon,R}\partial - U^{\Gamma_{t},\epsilon,R}\partial\|_{L^{p} \to L^{p}} \lesssim_{d,p} \mathrm{Lip}(A) \cdot |s - t|. \end{displaymath}
 \end{proposition} 
 The proof is similar to the proof of Proposition \ref{prop7}. Again we introduce the "$s$-derivative operator" $\partial_{s} U^{\Gamma_{s},\epsilon,R}\partial$ defined by
  \begin{displaymath} (\partial_{s} U^{\Gamma_{s},\epsilon,R}\partial)f(x) := \int \partial_{s} L_{s,\epsilon,R}(x,y)f(y) \, dy, \qquad f \in \mathcal{S}(\R), \, x \in \R. \end{displaymath} 
Analogously to \eqref{form93}, we note that $s \mapsto L_{s,\epsilon,R}(x,y)$ is differentiable for all $x \neq y$ for which $\nabla A(y)$ exists, and the derivative $\partial_{s}L_{s,\epsilon,R}(x,y)$ is given by the following expression:
  \begin{align} \partial_{s} L_{s,\epsilon,R}(x,y) & = k_{\epsilon,R}'(|\Gamma_{s}(x) - \Gamma_{s}(y)|) \frac{(x - y)\sum_{j = 1}^{d - 1}(A(x) - A(y))^{2}}{2|\Gamma_{s}(x) - \Gamma_{s}(y)|^{2}} \notag\\
& \quad - k_{\epsilon,R}(|\Gamma_{s}(x) - \Gamma_{s}(y)|) \frac{(x - y) \sum_{j = 1}^{d - 1}(A_{j}(x) - A_{j}(y))^{2}}{2|\Gamma_{s}(x) - \Gamma_{s}(y)|^{3}} \notag\\
& \quad + k_{\epsilon,R}'(|\Gamma_{s}(x) - \Gamma_{s}(y)|) \frac{s^{2}\nabla A(y) \cdot (A(x) - A(y)) \sum_{j = 1}^{d - 1} (A_{j}(x) - A_{j}(y))^{2}}{|\Gamma_{s}(x) - \Gamma_{s}(y)|^{2}} \notag\\
& \quad - k_{\epsilon,R}(|\Gamma_{s}(x) - \Gamma_{s}(y)|) \frac{s^{2}\nabla A(y) \cdot (A(x) - A(y)) \sum_{j = 1}^{d - 1} (A_{j}(x) - A_{j}(y))^{2}}{2|\Gamma_{s}(x) - \Gamma_{s}(y)|^{3}} \notag\\
&\quad + k_{\epsilon,R}(|\Gamma_{s}(x) - \Gamma_{s}(y)|) \cdot \frac{2s\nabla A(y) \cdot (A(x) - A(y))}{|\Gamma_{s}(x) - \Gamma_{s}(y)|}. \notag \end{align}
The only difference to \eqref{form93} is the appearance of $\nabla A(y)$ as opposed to $\nabla A(x)$. From these expressions one may deduce the following counterpart of Proposition \ref{prop6}:

\begin{proposition}\label{prop11} For $s \in [0,1]$, the kernel $\partial_{s} L_{s,\epsilon,R}(x,y)$ can be written as a sum of $\lesssim_{d} 1$ bounded kernels of the form
\begin{displaymath} (x,y) \mapsto \mathrm{Lip}(A)^{m} \cdot B(y) \cdot \mathcal{K}(\Gamma_{s}(x) - \Gamma_{s}(y)) \cdot \left(\tfrac{L_{1}(x) - L_{1}(y)}{x - y} \right)^{n_{1}}\left(\tfrac{L_{2}(x) - L_{2}(y)}{x - y} \right)^{n_{2}} \end{displaymath}  
where
\begin{itemize}
\item $\mathcal{K} \in C^{2}_{c}(\R^{d} \, \setminus \, \{0\})$ is an odd bounded standard kernel in the sense of Definition \ref{def:standardKernel}.
\item $L_{1},L_{2} \colon \R \to \R$ are $1$-Lipschitz,
\item $B \in \{1,L_{1}'\} \subset L^{\infty}(\R)$,
\item $m,n_{1},n_{2} \in \N$ with $n \geq 1$.
\end{itemize}
Moreover, the standard kernel constants of $\mathcal{K}$ are absolute. 
\end{proposition}

We do not repeat the proof: the five kernels are the same as those in Proposition \ref{prop6}, except that the $L^{\infty}$-factor $B(x)$ is replaced by $B(y)$. Since this difference has no qualitative impact on $L^{p}$-mapping properties, we have the following counterpart of Corollary \ref{cor1}:

\begin{cor}\label{cor2} Assume $\mathrm{Lip}(A) \leq 1$. For $s \in [0,1]$ and $p \in (1,\infty)$ fixed, the operator $\partial_{s} U^{\Gamma_{s},\epsilon,R} \partial$ is bounded on $L^{p}(\R)$, and indeed
\begin{displaymath} \|\partial_{s} U^{\Gamma_{s},\epsilon,R}\partial\|_{L^{p} \to L^{p}} \lesssim_{d,p} \mathrm{Lip}(A). \end{displaymath}
\end{cor} 

Here is finally the counterpart of Lemma \ref{lemma6}:

\begin{lemma}\label{lemma7} For $s \in [0,1]$ and $f,g \in \mathcal{S}(\R)$ fixed,
\begin{displaymath} \lim_{t \to s} \frac{1}{s - t} \left(\int (U^{\Gamma_{s},\epsilon,R} \partial) f \cdot g - \int (U^{\Gamma_{t},\epsilon,R}\partial) f \cdot g \right) = \int (\partial_{s} U^{\Gamma_{s},\epsilon,R}\partial) f \cdot g. \end{displaymath}
\end{lemma}

\begin{proof} With Proposition \ref{prop11} taking the role of Proposition \ref{prop6}, the proof is the same as the proof of Lemma \ref{lemma6} with one small difference: this time the central formula
\begin{align} & \frac{(U^{\Gamma_{s},\epsilon,R}\partial) f(x) - (U^{\Gamma_{t},\epsilon,R} \partial) f(x)}{s - t} - (\partial_{s}  U^{\Gamma_{s},\epsilon,R}\partial) f(x) \notag\\
&\label{form98}\qquad = \int_{0}^{1} \int \left[ \partial_{r(s - t) + t} L_{r(s - t) + t}(x,y) - \partial_{s} L_{s}(x,y) \right] f(y) \, dy \, dr \end{align}
holds for all $x \in \R$, and not just those at which $\nabla A(x)$ exists. In fact, fixing $x \in \R$ and abbreviating $L_{s} := L_{s,\epsilon,R}$, one may first write
\begin{displaymath} (U^{\Gamma_{s},\epsilon,R}\partial) f(x) - (U^{\Gamma_{t},\epsilon,R} \partial) f(x) = \int [L_{s}(x,y) - L_{t}(x,y)]f(y) \, dy.  \end{displaymath}
The map $s \mapsto L_{s}(x,y)$ is differentiable for all $y \neq x$ such that $\nabla A(y)$ exists, and in particular for a.e. $y \in \R$. Therefore,
\begin{displaymath} \int [L_{s}(x,y) - L_{t}(x,y)]f(y) \, dy = \int \left( \int_{0}^{1} \partial_{r(s - t) + t}L_{r(s - t) + t}(x,y) \, dr \right) f(y) \, dy. \end{displaymath} 
Now the boundedness of the kernels $\partial_{s}L_{s}(x,y)$ enables one to use Fubini's theorem to deduce \eqref{form98}. The remainder of the argument proceeds as in Lemma \ref{lemma6}: for every $(x,y) \in \R^{2}$ fixed, one of the $\lesssim_{d} 1$ kernels $\mathcal{K}$ from Proposition \ref{prop11} satisfies
\begin{displaymath} |\partial_{r(s - t) + t} L_{r(s - t) + t}(x,y) - \partial_{s} L(x,y)| \lesssim_{d} |\mathcal{K}(\Gamma_{r(s - t) + t}(x) - \Gamma_{r(s - t) + t}(y)) - \mathcal{K}(\Gamma_{s}(x) - \Gamma_{s}(y))|, \end{displaymath} 
We leave the remaining details to the reader. \end{proof} 

With Lemma \ref{lemma7} in hand, the proof of Proposition \ref{prop10} follows the proof of Proposition \ref{prop7}, and we will not repeat the details.

\subsubsection{$L^{p}$-bounds for $\Rea \beta \in \{0,1\}$} Propositions \ref{prop7} and \ref{prop10} handled the cases $\beta \in \{0,1\}$, and in this section we generalise this to the cases $\Rea \beta \in \{0,1\}$. The argument is the same as the deduction of Corollary \ref{cor5} from Corollary \ref{cor4}, but we re-record the main steps. Here is the result we are aiming for:
\begin{cor}\label{cor7} Assume $\mathrm{Lip}(A) \leq 1$. Let $f \in \mathcal{S}(\R)$, $p \in (1,\infty)$, $s,t \in [0,1]$, and $\Rea \beta \in \{0,1\}$. Then, for $0 < \epsilon \leq 1 \leq R$, the tempered distribution 
\begin{displaymath} \Lambda := T^{\Gamma_{s},\epsilon,R}_{\beta}f - T^{\Gamma_{t},\epsilon,R}_{\beta}f \end{displaymath}
is an $L^{p}$-function with norm
\begin{displaymath} \|\Lambda\|_{L^{p}} \lesssim_{d,p} (1 + (\mathrm{Im\,} \beta)^{2})^{2} \cdot \mathrm{Lip}(A) \cdot |s - t| \cdot \|f\|_{L^{p}}. \end{displaymath} \end{cor} 
The key step in the proof of Corollary \ref{cor5} was Proposition \ref{prop12}. That argument works exactly in the same way for the difference operators $T_{\beta}^{\Gamma_{s},\epsilon,R} - T_{\beta}^{\Gamma_{t},\epsilon,R}$ in place of $T_{\beta}^{\Gamma,\epsilon,R}$:

\begin{proposition}\label{prop13} Let $f,g \in \mathcal{S}(\R)$, $p \in (1,\infty)$, $s,t \in [0,1]$, and $\Rea \beta \in \{0,1\}$. Then,
\begin{displaymath} \left| \int (T_{\beta}^{\Gamma_{s},\epsilon,R} - T_{\beta}^{\Gamma_{t},\epsilon,R})f \cdot g\right| \leq \|(\partial U^{\Gamma_{s},\epsilon,R} - \partial U^{\Gamma_{t},\epsilon,R})(\Delta^{\beta/2}f)\|_{L^{p}}\|(\partial^{-1}\Delta^{(1 - \beta)/2})g\|_{L^{p'}}, \quad \Rea \beta = 0, \end{displaymath}
\begin{displaymath} \left| \int  (T_{\beta}^{\Gamma_{s},\epsilon,R} - T_{\beta}^{\Gamma_{t},\epsilon,R})f \cdot g\right| \leq \|(U^{\Gamma_{s},\epsilon,R}\partial - U^{\Gamma_{t},\epsilon,R}\partial)(\partial^{-1}\Delta^{\beta/2}f)\|_{L^{p}}\|\Delta^{(1 - \beta)/2}g\|_{L^{p'}}, \quad \Rea \beta = 1. \end{displaymath}
\end{proposition}
For further clarifications on the notation, see the statement of Proposition \ref{prop12}. We omit the proof, as it is the same as the proof of Proposition \ref{prop12}.

\begin{proof}[Proof of Corollary \ref{cor7}] The $L^{p}$-norm of $\Lambda$ is estimated by duality and Proposition \ref{prop13}. In the case $\Rea \beta = 0$, this leads to
\begin{displaymath} \|\Lambda\|_{L^{p}} \leq \|\partial U^{\Gamma_{s},\epsilon,R} - \partial U^{\Gamma_{t},\epsilon,R}\|_{L^{p} \to L^{p}} \cdot \|\Delta^{\beta/2}f\|_{L^{p}} \cdot \|(\partial^{-1}\Delta^{(1 - \beta)/2})g\|_{L^{p'}}, \end{displaymath}
and in the case $\Rea \beta = 1$ similarly 
\begin{displaymath} \|\Lambda\|_{L^{p}} \leq \|U^{\Gamma_{s},\epsilon,R}\partial - U^{\Gamma_{t},\epsilon,R}\partial\|_{L^{p} \to L^{p}} \cdot \|\partial^{-1}\Delta^{\beta/2}f\|_{L^{p}} \cdot \Delta^{(1 - \beta)/2}f\|_{L^{p'}}. \end{displaymath} 
The operator norms are now estimated by Proposition \ref{prop7} and Proposition \ref{prop10}, and for the remaining factors one may copy the estimates from the proof of Corollary \ref{cor5}.  \end{proof} 

\subsubsection{$L^{p}$ bounds for $\Rea \beta \in [0,1]$} We finally use complex interpolation to deduce a version of Corollary \ref{cor7} valid for all $\Rea \beta \in [0,1]$. This is analogous to deriving Corollary \ref{cor6} from Corollary \ref{cor5}, so we refer the reader to the proof of Corollary \ref{cor6}.

\begin{cor}\label{cor8} Assume $\mathrm{Lip}(A) \leq 1$. Let $f \in \mathcal{S}(\R)$, $p \in (1,\infty)$, $s,t \in [0,1]$, and $\Rea \beta \in [0,1]$. Then, for $0 < \epsilon \leq 1 \leq R < \infty$, the distribution 
\begin{displaymath} \Lambda := T^{\Gamma_{s},\epsilon,R}_{\beta}f - T^{\Gamma_{t},\epsilon,R}_{\beta}f \end{displaymath}
is an $L^{p}$-function with norm
\begin{displaymath} \|\Lambda\|_{L^{p}} \lesssim_{d,p} e^{(\mathrm{Im\,} \beta)^{2}} \mathrm{Lip}(A) \cdot |s - t| \cdot \|f\|_{L^{p}}. \end{displaymath} \end{cor} 

\section{Limit operators \texorpdfstring{$T_{\beta}^{\Gamma}$}{}}\label{s:limit operators}
\subsection{Definition and $L^{p}$ invertibility of $T_{\beta}^{\Gamma}$} Let $A \colon \R \to \R^{d - 1}$ be Lipschitz, $\Gamma(x) = (x,A(x))$, and $\beta \in \C$ with $\Rea \beta \in [0,1]$. In Definition \ref{def2}, we initially defined $T_{\beta}^{\Gamma,\epsilon,R}$ as a mapping from $\mathcal{S}(\R)$ (or slightly more generally $\dot{H}^{\beta}$) to $\mathcal{S}'(\R)$. In Corollary \ref{cor6} we however learned that
\begin{displaymath} \|T^{\Gamma,\epsilon,R}_{\beta}f\|_{L^{p}} \lesssim_{\mathrm{Lip}(A),d,\beta,p} \|f\|_{L^{p}}, \qquad f \in \mathcal{S}(\R), \end{displaymath}
and consequently $T^{\Gamma,\epsilon,R}_{\beta}$ extends to a bounded operator on $L^{p}(\R)$, for $p \in (1,\infty)$. We keep the same notation for the extension. 

The purpose of this section is to verify that there exist (possibly non-unique) subsequential limits of the operators $T^{\Gamma,\epsilon,R}_{\beta}$, as $\epsilon \to 0$ and $R \to \infty$, and that such limits are invertible on $L^{p}$, provided that $\mathrm{Lip}(A)$ is sufficiently small, depending on $d$ and $p$.

We start by quoting a result in abstract functional analysis which says that if $H$ is a separable Hilbert space, then the closed unit ball of $\mathcal{B}(H)$ is compact in the \emph{weak operator topology} (WOT). For a reference, see \cite[Chapter IX, Proposition 5.5]{MR1070713}, and note also that the WOT on the unit ball of $\mathcal{B}(H)$ is metrizable by \cite[Chapter IX, Proposition 1.3(e)]{MR1070713}. Here is what the previous result means. Assume that $\{T_{n}\}$ is a sequence of uniformly bounded operators on $L^{2}(\R)$. Then, there exists a subsequence $\{n_{j}\}_{j \in \N}$, and a bounded operator $T \colon L^{2}(\R) \to L^{2}(\R)$ such that
\begin{equation}\label{form111} \lim_{j \to \infty} \int T_{n_{j}}f \cdot g = \int Tf \cdot g, \qquad f,g \in L^{2}(\R). \end{equation}
In our concrete situation, we may derive the following corollary: 

\begin{thm}[Definition of $T^{\Gamma}_{\beta}$]\label{thm1} Let $\beta \in [0,1]$, and $p \in (1,\infty)$. Then, there exists a vanishing sequence $\{\epsilon_{j}\}_{j \in \N}$ of positive numbers, and a bounded operator $T_{\beta}^{\Gamma} \colon L^{p}(\R) \to L^{p}(\R)$ such that
\begin{equation}\label{form113} \int T^{\Gamma}_{\beta}f \cdot g = \lim_{j \to \infty} \int T^{\Gamma,\epsilon_{j},1/\epsilon_{j}}_{\beta}f \cdot g, \qquad f \in L^{p}(\R), \, g \in L^{p'}(\R). \end{equation}
Moreover, $T_{\beta}^{\Gamma}$ is invertible on $L^{p}(\R)$ if $\mathrm{Lip}(A)$ is sufficiently small, depending only on $d,p$.
\end{thm}

\begin{proof} According to Corollary \ref{cor6}, the operators $T_{\beta}^{\Gamma,\epsilon,1/\epsilon}$ are uniformly (in $\epsilon > 0$) bounded on both $L^{2}(\R)$ and $L^{p}(\R)$, with operator norm depending only on $\mathrm{Lip}(A),d$, and $p$. By the functional analytic theorem discussed at \eqref{form111}, there exists a vanishing sequence $\{\epsilon_{j}\}_{j \in \N} \subset  (0,\infty)$, and a bounded operator $T \colon L^{2}(\R) \to L^{2}(\R)$ with the property
\begin{equation}\label{form112} \int Tf \cdot g = \lim_{j \to \infty} \int T_{\beta}^{\Gamma,\epsilon_{j},1/\epsilon_{j}}f \cdot g, \qquad f,g \in L^{2}(\R). \end{equation} 
We abbreviate in the sequel
\begin{displaymath} T^{\Gamma,j}_{\beta} := T^{\Gamma,\epsilon_{j},1/\epsilon_{j}}_{\beta}, \qquad j \in \N. \end{displaymath}
It follows from \eqref{form112} (applied with $f,g \in \mathcal{S}(\R)$) that $T$ admits a bounded extension (from $L^{2}(\R) \cap L^{p}(\R)$) to $L^{p}(\R)$, and indeed the operator norm of that extension is bounded from above by
\begin{displaymath} \mathbf{C} := \sup_{\epsilon > 0} \|T_{\beta}^{\Gamma,\epsilon,1/\epsilon}\|_{L^{p} \to L^{p}} < \infty. \end{displaymath}
This extension is our desired operator $T^{\Gamma}_{\beta}$. We next check that \eqref{form113} holds for all $f \in L^{p}(\R)$ and $g \in L^{p'}(\R)$. Fix $f \in L^{p}(\R)$ and $g \in L^{p'}(\R)$, $\delta \in (0,1]$, and pick $\varphi,\gamma \in \mathcal{S}(\R)$ with 
\begin{displaymath} \|f - \varphi\|_{L^{p}} \leq \delta \quad \text{and} \quad \|g - \gamma\|_{L^{p'}} \leq \delta. \end{displaymath}
Since $\varphi,\gamma \in L^{2}(\R) \cap L^{p}(\R)$, by \eqref{form112} there exists $j_{0} = j_{0}(\delta,\varphi,\gamma) \in \N$ such that
\begin{displaymath} \left| \int T^{\Gamma}_{\beta}\varphi \cdot \gamma - \int T^{\Gamma,j}_{\beta} \varphi \cdot \gamma \right| \leq \delta, \qquad j \geq j_{0}. \end{displaymath}
Now, for $j \geq j_{0}$, we conclude
\begin{align*} \left| \int T^{\Gamma}_{\beta}f \cdot g - \int T^{\Gamma,j}_{\beta}f \cdot g \right| & \leq \left| \int T_{\beta}^{\Gamma}(f - \varphi) \cdot g - \int T_{\beta}^{\Gamma,j} (f - \varphi) \cdot g \right|\\
& \quad + \left| \int T_{\beta}^{\Gamma} \varphi \cdot \gamma - \int T_{\beta}^{\Gamma,j} \varphi \cdot \gamma \right|\\
& \quad + \left| \int T_{\beta}^{\Gamma} \varphi \cdot (g - \gamma) - \int T_{\beta}^{\Gamma,j}\varphi \cdot (g - \gamma) \right|\\
& \leq 2\mathbf{C}\|f - \varphi\|_{L^{p}} \|g\|_{L^{p'}} + \delta + 2\mathbf{C}\|f\|_{L^{p}}\|g - \gamma\|_{L^{p'}}\\
& \leq 2\mathbf{C}\delta\|g\|_{L^{p'}} + \delta + 2\mathbf{C}\delta\|f\|_{L^{p}}. \end{align*} 
Since the right hand side can be made arbitrarily small by decreasing the value of $\delta > 0$, we have established \eqref{form113}.

It remains to prove that $T_{\beta}^{\Gamma}$ is invertible on $L^{p}(\R)$, provided that the Lipschitz constant of $A$ is sufficiently small in terms of $p$. To see this, write
\begin{displaymath} T^{\Gamma_{s},j}_{\beta} := T^{\Gamma_{s},\epsilon_{j},\epsilon_{j}^{-1}}_{\beta}, \qquad s \in [0,1], \end{displaymath}
where $\Gamma_{s}(x) = (x,sA(x))$, thus $T_{\beta}^{\Gamma_{1},j} = T_{\beta}^{\Gamma,j}$ in our previous notation. Furthermore $\Gamma_{0}(x) = (x,0)$, so it follows from Proposition \ref{prop4} that
\begin{equation}\label{form114} \lim_{j \to \infty} \int T^{\Gamma_{0},j}_{\beta} \cdot g = \tfrac{1}{2}\int f \cdot g, \qquad f,g \in \mathcal{S}(\R). \end{equation}
Using the uniform $L^{p}$-boundedness of the operators $T_{\beta}^{\Gamma_{0},j}$, the conclusion above can be upgraded to hold for all $f \in L^{p}(\R)$ and $g \in L^{p'}(\R)$ (by repeating the argument we used to establish \eqref{form113}).

Let $I$ be the identity operator on $L^{p}(\R)$. Then, using \eqref{form113} and \eqref{form114}, we deduce that
\begin{displaymath} \int (\tfrac{1}{2}I - T^{\Gamma}_{\beta})f \cdot g = \lim_{j \to \infty} \int (T_{\beta}^{\Gamma_{0},j} - T_{\beta}^{\Gamma_{1},j})f \cdot g, \qquad f \in L^{p}(\R), \, g \in L^{p'}(\R). \end{displaymath}
It now follows from Corollary \ref{cor8} applied with $(s,t) = (0,1)$ that the operator norm of $\tfrac{1}{2}I - T_{\beta}^{\Gamma}$ is bounded from above by 
\begin{displaymath} \|\tfrac{1}{2}I - T_{\beta}^{\Gamma}\|_{L^{p} \to L^{p}} \lesssim_{d,p} \mathrm{Lip}(A), \end{displaymath} 
provided that \emph{a priori} $\mathrm{Lip}(A) \leq 1$. Since $\tfrac{1}{2}I$ is invertible on $L^{p}(\R)$, it now follows (from a standard Neumann series argument) that $T^{\Gamma}_{\beta}$ is also invertible on $L^{p}(\R)$ if $\mathrm{Lip}(A)$ is sufficiently small, depending only on $d,p$. This completes the proof of Theorem \ref{thm1}. \end{proof}

\subsection{Bounding $\Delta^{-\beta/2}\mu$ in terms of $\Delta^{(1 - \beta)/2}U^{\Gamma}\mu$}\label{s:Adams} The purpose of this section is to bound the $L^{p}$-norm of $\Delta^{-\beta/2}\mu$ in terms of the $L^{p}$-norm of $\Delta^{(1 - \beta)/2}\mu$. The formal proof is simple, using the $L^{p}$-invertibility of $\Delta^{(1 - \beta)/2}U^{\Gamma}\Delta^{\beta/2}$:
\begin{displaymath} \|\Delta^{-\beta/2}\mu\|_{L^{p}} \lesssim \|\Delta^{(1 - \beta)/2}U^{\Gamma}\Delta^{\beta/2}\Delta^{-\beta/2}\mu\|_{L^{p}} = \|\Delta^{(1 - \beta)/2}U^{\Gamma}\mu\|_{L^{p}}.  \end{displaymath}
Making sure that everything commutes and converges in the right way is verified in the following proposition. Recall the space $\mathcal{M}(\R)$ from Definition \ref{def5Intro}. We will use the fact that if $\mu \in \mathcal{M}(\R)$, then $x \mapsto \int \max\{-\log |x - y|,0\} \, d\mu(y) \in L^{\infty}(\R)$.

\begin{proposition}\label{prop8-2} Let $\beta \in [0,1)$ and $p > 1/(1 - \beta)$. Let $T^{\Gamma}_{\beta} \colon L^{p}(\R) \to L^{p}(\R)$ be an operator as in Theorem \ref{thm1}. More precisely, we assume the existence of sequences $\{\epsilon_{j}\}_{j},\{R_{j}\}_{j}$ of positive reals such that $\epsilon_{j} \to 0$ and $R_{j} \to \infty$ as $j \to \infty$, and 
\begin{equation}\label{form82} \int T_{\beta}^{\Gamma}f \cdot \psi = \lim_{j \to \infty} \int T_{\beta}^{\Gamma,\epsilon_{j},R_{j}}f \cdot \psi, \qquad f \in L^{p}(\R), \, \psi \in \mathcal{S}(\R). \end{equation} 
We also assume that $T_{\beta}^{\Gamma}$ is invertible on $L^{p}(\R)$ (thus, the proposition applies to the operators in Theorem \ref{thm1} if $\mathrm{Lip}(A)$ is sufficiently small depending on $d,p$).

Then,
\begin{equation}\label{form165} \|\Delta^{-\beta/2}\mu\|_{L^{p}} \lesssim_{T_{\beta}^{\Gamma}} \|\Delta^{(1 - \beta)/2}U^{\Gamma}\mu\|_{L^{p}}, \qquad \mu \in \mathcal{M}(\R), \end{equation}
where the implicit constant only depends on the operator norm of the $L^{p}$-inverse of $T_{\beta}^{\Gamma}$. Here $\Delta^{(1 - \beta)/2}U^{\Gamma}\mu \in \mathcal{S}'(\R)$ is the tempered distribution familiar from Definition \ref{def6}, and its $L^{p}$-norm in \eqref{form165} should be interpreted as $+\infty$ if $\Delta^{(1 - \beta)/2}U^{\Gamma}\mu \notin L^{p}(\R)$. \end{proposition}
 
 \begin{remark}\label{rem5} The expression $\Delta^{-\beta/2}\mu(x) \in \R \cup \{+\infty\}$ is defined pointwise by \eqref{form161}. Proposition \ref{prop8-2} does not imply that $\Delta^{-\beta/2}\mu \in L^{p}(\R)$. This conclusion only follows if the tempered distribution $\Delta^{(1 - \beta)/2}U^{\Gamma}\mu$ is represented by an $L^{p}$-function.  \end{remark} 

\begin{proof}[Proof of Proposition \ref{prop8-2}] To apply the $L^{p}$-invertibility of $T_{\beta}^{\Gamma}$, it would be useful to know that $\Delta^{-\beta/2}\mu \in L^{p}(\R)$. We do not know this \emph{a priori}, but a viable substitute is to consider the functions $\Delta^{-\beta/2}\mu_{\delta}$ instead. Here $\mu_{\delta} \coloneqq \mu \ast \varphi_{\delta} \in C^{\infty}_{c}(\R)$, where $\varphi \in C^{\infty}(\R)$ with the usual properties $0 \leq \varphi \lesssim \mathbf{1}_{[-1,1]}$, $\int \varphi = 1$, and $\varphi_{\delta}(x) = \delta^{-1}\varphi(x/\delta)$. 

Recall from \eqref{form161} that
\begin{displaymath} \Delta^{-\beta/2}\mu_{\delta} = c(\beta)(\mu_{\delta} \ast |\cdot|^{\beta - 1}). \end{displaymath}
It now follows easily from the smoothness of $\mu_{\delta}$ and the decay of $|\cdot|^{\beta - 1}$ that $\Delta^{-\beta/2}\mu_{\delta} \in L^{\infty}(\R) \cap L^{p}(\R)$ for $p > 1/(1 - \beta)$. We also claim that
\begin{displaymath} \|\Delta^{-\beta/2}\mu\|_{L^{p}} \leq \liminf_{\delta \to 0} \|\Delta^{-\beta/2}\mu_{\delta}\|_{L^{p}}. \end{displaymath}
To see this, note that $\Delta^{-\beta/2}\mu \in L^{1}_{\mathrm{loc}}(\R)$ (since the kernel $|\cdot|^{\beta - 1} \in L^{1}_{\mathrm{loc}}(\R)$), and also that
\begin{displaymath} \left| \int \Delta^{-\beta/2}\mu \cdot g \right| = \lim_{\delta \to 0} \left| \int \Delta^{-\beta/2}\mu \cdot g_{\delta} \right| = \lim_{\delta \to 0} \left| \int \Delta^{-\beta/2}\mu_{\delta}  \cdot g\right| \leq \liminf_{\delta \to 0} \|\Delta^{-\beta/2}\mu_{\delta}\|_{L^{p}}\|g\|_{L^{p'}} \end{displaymath}
for all $g \in \mathcal{C}^{\infty}_{c}(\R)$. After these preliminaries, the estimate \eqref{form165} is reduced to showing
\begin{equation}\label{form166} \liminf_{\delta \to 0} \|\Delta^{-\beta/2}\mu_{\delta}\|_{L^{p}} \lesssim \|\Delta^{(1 - \beta)/2}U^{\Gamma}\mu\|_{L^{p}}. \end{equation}

Fix $\delta > 0$. Since $T^{\Gamma}_{\beta} \colon L^{p}(\R) \to L^{p}(\R)$ is invertible,
\begin{displaymath} \|\Delta^{-\beta/2}\mu_{\delta}\|_{L^{p}} \lesssim_{T_{\beta}^{\Gamma}} \|T^{\Gamma}_{\beta}(\Delta^{-\beta/2}\mu_{\delta})\|_{L^{p}}. \end{displaymath}
We may furthermore select $\psi \in \mathcal{S}(\R)$ with $\|\psi\|_{L^{p'}} \leq 1$ such that
\begin{equation}\label{form167} \|\Delta^{-\beta/2}\mu_{\delta}\|_{L^{p}} \lesssim_{T_{\beta}^{\Gamma}} \int T^{\Gamma}_{\beta}(\Delta^{-\beta/2}\mu_{\delta}) \cdot \psi. \end{equation}
Since $\Delta^{-\beta/2}\mu_{\delta} \in L^{p}(\R)$ and $\psi \in L^{p'}(\R)$, according to \eqref{form82} the integral on the right equals
\begin{displaymath} \lim_{j \to \infty} \int T^{\Gamma,\epsilon_{j},R_{j}}_{\beta}(\Delta^{-\beta/2}\mu_{\delta}) \cdot \psi. \end{displaymath}  
Recall that the operators $T_{\beta}^{\Gamma,\epsilon_{j},R_{j}}$ (introduced in Definition \ref{def2}) were defined on $\dot{H}^{\beta}$, and clearly $\Delta^{-\beta/2}\mu_{\delta} \in \dot{H}^{\beta}$, since $\Delta^{\beta/2}(\Delta^{-\beta/2}\mu_{\delta}) = \mu_{\delta} \in L^{2}(\R)$. By the definition of of $T^{\Gamma,\epsilon_{j},R_{j}}_{\beta}$:
\begin{align*} \int T^{\Gamma,\epsilon_{j},R_{j}}_{\beta}(\Delta^{-\beta/2}\mu_{\delta}) \cdot \psi \stackrel{\mathrm{Def.\,} \ref{def2}}{=} \int U^{\Gamma,\epsilon_{j},R_{j}}\mu_{\delta} \cdot \Delta^{(1 - \beta)/2}\psi. \end{align*}
Here $\Delta^{(1 - \beta)/2}\psi \in L^{1}(\R)$ thanks to Lemma \ref{lemma2}, and $\int \Delta^{(1 - \beta)/2}\psi = 0$. Note also that $\mu_{\delta}$ is a finite measure with total mass $\|\mu\|$, so 
\begin{displaymath} U^{\Gamma,\epsilon_{j},R_{j}}\mu_{\delta}(x) = \int [\log_{\epsilon_{j},R_{j}} |\Gamma(x) - \Gamma(y)| - C_{R_{j}}] \, d\mu_{\delta}(y) + C_{R_{j}}\|\mu\|, \qquad x \in \R, \end{displaymath}
where $C_{R_{j}} = O(\log R_{j})$ is the constant from Lemma \ref{truncationLemma}. Taking into account the zero mean of $\Delta^{(1 - \beta)/2}\psi$, we may further write
\begin{displaymath} \int T^{\Gamma,\epsilon_{j},R_{j}}_{\beta}(\Delta^{-\beta/2}\mu_{\delta}) \cdot \psi = \int \left(\int [\log_{\epsilon_{j},R_{j}} |\Gamma(x) - \Gamma(y)| - C_{R_{j}}] \, d\mu_{\delta}(y) \right) \Delta^{(1 - \beta)/2}\psi(x) \, dx.  \end{displaymath} 
Fixing $x \in \R$, and using Lemma \ref{truncationLemma}(i), the inner integral comprises of three parts:
\begin{align} \int_{|\Gamma(x) - \Gamma(y)| \leq \epsilon_{j}} O(\log \epsilon_{j}^{-1}) \, d\mu_{\delta}(y) & + \int_{\epsilon_{j} \leq |\Gamma(x) - \Gamma(y)| \leq R_{j}} \log \tfrac{1}{|\Gamma(x) - \Gamma(y)|} d\mu_{\delta}(y) \notag\\
&\label{form84} \quad + \int_{|\Gamma(x) - \Gamma(y)| \geq R_{j}} O(\log R_{j}) \, d\mu_{\delta}(y).   \end{align}
\begin{itemize}
\item Since $\|\mu_{\delta}\|_{L^{\infty}} \lesssim_{\delta} 1$, the first term is bounded by $\lesssim_{\delta} \epsilon_{j}\log \epsilon_{j}^{-1} \lesssim 1$, and tends to zero as $j \to \infty$.
\item The third term also tends to zero thanks to the compact support of $\mu$: for $x \in \R$ fixed, $\mu_{\delta}(y) = 0$ for all $y \in \R$ with $|\Gamma(x) - \Gamma(y)| \geq R_{j}$, as soon as $R_{j} \geq C(|x| + \diam \spt \mu_{\delta})$. This also shows that $\eqref{form84} = O(1 + \max\{\log |x|,0\})$. 
\item Finally, the middle term is bounded in absolute value by $\lesssim \|\mu\| + \|\mu_{\delta}\|_{L^{\infty}}$, and tends (for $x \in \R$ fixed) to $U^{\Gamma}\mu_{\delta}(x)$ as $j \to \infty$.
\end{itemize}
Since $x \mapsto (1 + \max\{\log |x|,0\})\Delta^{(1 - \beta)/2}\psi(x) \in L^{1}(\R)$ thanks to Lemma \ref{lemma3}, and $\beta < 1$, we may now apply dominated convergence to deduce that
\begin{displaymath} \|\Delta^{-\beta/2}\mu_{\delta}\|_{L^{p}} \stackrel{\eqref{form167}}{\lesssim_{T_{\beta}^{\Gamma}}} \lim_{j \to \infty} \int T^{\Gamma,\epsilon_{j},R_{j}}_{\beta}(\Delta^{-\beta/2}\mu_{\delta}) \cdot \psi = \int U^{\Gamma}\mu_{\delta} \cdot \Delta^{(1 - \beta)/2}\psi. \end{displaymath} 
To prove \eqref{form166} (and the proposition), it remains to show that
\begin{equation}\label{form168} \liminf_{\delta \to 0} \left| \int U^{\Gamma}\mu_{\delta} \cdot \Delta^{(1 - \beta)/2}\psi \right| \leq \|\Delta^{(1 - \beta)/2}U^{\Gamma}\mu\|_{L^{p}}. \end{equation} 
We claim that 
\begin{equation}\label{form88} \lim_{\delta \to 0} \int U^{\Gamma}\mu_{\delta} \cdot \Delta^{(1 - \beta)/2}\psi = \int U^{\Gamma}\mu \cdot \Delta^{(1 - \beta)/2}\psi \stackrel{\mathrm{def.}}{=} \int \Delta^{(1 - \beta)/2}U^{\Gamma}\mu \cdot \psi, \end{equation}
from which \eqref{form168} follows, recalling that $\|\psi\|_{L^{p'}} \leq 1$. 

To prove \eqref{form88}, we start by checking that 
\begin{equation}\label{form80} \lim_{\delta \to 0} U^{\Gamma}\mu_{\delta}(x) = U^{\Gamma}\mu(x), \qquad x \in \R. \end{equation}
It will be useful to split the logarithmic kernel into positive and negative parts:
\begin{displaymath} U^{\Gamma}\nu(x) = U^{\Gamma}_{+}\nu(x) + U^{\Gamma}_{-}\nu(x) = \int \log_{+} \frac{1}{|\Gamma(x) - \Gamma(y)|} \, d\nu(y) + \int \log_{-} \frac{1}{|\Gamma(x) - \Gamma(y)|} \, d\nu(y), \end{displaymath} 
for $\nu \in \{\mu,\mu_{\delta}\}$. Here, and only in this proof $\log_{+} := \max\{\log(x),1\}$ (the "$1$" is not a typo but a matter of technical convenience)  and $\log_{-} = \log - \log_{+}$. Thus,
\begin{itemize}
\item $\log_{+}(x) \geq 0$, $\log_{+}$ is non-decreasing, and $\log_{+}(x) = \log(x)$ for all $|x| \geq 2$,
\item $\log_{+}(x) = 1$ for $|x| \leq 2$, and 
\item $\log_{-} \in C(\R \, \setminus \, \{0\})$ and $\log_{-}(x) = 0$ for $|x| \geq 2$.
\end{itemize}
Now, to establish \eqref{form80}, fix $x \in \R$, and expand
\begin{equation}\label{form89} U^{\Gamma}_{\pm}\mu_{\delta}(x) = \int \left( \int \log_{\pm} \frac{1}{|\Gamma(x) - \Gamma(y)|} \varphi_{\delta}(y - z) \, dy \right) \, d\mu(z). \end{equation} 
Note that $z \mapsto \log_{-} 1/|\Gamma(x) - \Gamma(z)| \in C(\R)$. Therefore
\begin{equation}\label{form85} \lim_{\delta \to 0} \int \log_{-} \frac{1}{|\Gamma(x) - \Gamma(y)|} \varphi_{\delta}(y - z) \, dy = \log_{-} \frac{1}{|\Gamma(x) - \Gamma(z)|}, \qquad z \in \R. \end{equation}
Using the compact support of $\mu$, one may then easily justify an application of dominated convergence to show that $\lim_{\delta \to 0} U^{\Gamma}_{-}\mu_{\delta}(x) = U^{\Gamma}_{-}\mu(x)$.

It remains to show that $\lim_{\delta \to 0} U^{\Gamma}_{+}\mu_{\delta}(x) = U^{\Gamma}_{+}\mu(x)$. This time $z \mapsto \log_{+} 1/|\Gamma(x) - \Gamma(z)|$ is continuous on $\R \, \setminus \, \{x\}$, so the $\log_{+}$ counterpart of \eqref{form85} holds for all $z \in \R \, \setminus \, \{x\}$, and in particular $\mu$ almost everywhere ($\mathcal{M}(\R)$ consists of non-atomic measures). So, we need to justify an application of the dominated convergence theorem. To this end, we first observe that
\begin{displaymath} \log_{+} \frac{1}{|\Gamma(x) - \Gamma(y)|} \leq \log_{+} \frac{1}{|x - y|}, \qquad x,y \in \R, \end{displaymath} 
since $|\Gamma(x) - \Gamma(y)| \geq |x - y|$, and $\log_{+}$ is non-decreasing. We claim that
\begin{equation}\label{form86} \sup_{\delta \in (0,\tfrac{1}{6}]} \int \log_{+} \frac{1}{|x - y|} \varphi_{\delta}(y - z) \, dy \lesssim \max\{-\log |x - z|,0\} + 1, \quad x,z \in \R, \end{equation}
Since $z \mapsto  \max\{-\log |x - z|,0\}  \in L^{1}(\mu)$ by the hypothesis $\mu \in \mathcal{M}(\R)$, this will justify an application of dominated convergence. 

To prove \eqref{form86}, we may and will assume that $x = 0$. We briefly dispose of the case $|z| = |z - x| > \tfrac{2}{3}$. Then, for $\delta \in (0,\tfrac{1}{6}]$, we have $|y| \geq \tfrac{1}{2}$ for all $y \in \R$ with $\varphi_{\delta}(y - z) \neq 0$. Thus $y \mapsto \log_{+} |y|^{-1} = 1$ on the support of the integrand, and the integral equals $\int \varphi_{\delta} = 1$.

We then consider the case $|z| \leq \tfrac{2}{3}$, and explicitly $z \in [0,\tfrac{2}{3}]$. We claim that
\begin{equation}\label{form87} \frac{1}{\delta} \int_{z - \delta}^{z + \delta} \log \frac{1}{|y|} \, dy \leq 6 \log \frac{1}{z} + 6, \qquad \delta \in (0,\tfrac{1}{6}]. \end{equation} 
Assume first that $0 < \delta \leq z/2$. In this case, we estimate the left hand side of \eqref{form87} by 
\begin{displaymath} 2\log \frac{1}{z - \delta} \leq 2\log \frac{2}{z} = 2\log \frac{1}{z} + 2\log 2. \end{displaymath}
Next, in the case $\delta \geq z/2$, we estimate
\begin{displaymath} \frac{1}{\delta}\int_{z - \delta}^{z + \delta} \log \frac{1}{|y|} \, dy \leq \frac{2}{\delta}\int_{0}^{3\delta} \log \frac{1}{y} \, dy = 6 + 6 \log \frac{1}{3\delta} \leq 6 + 6\log \frac{1}{z},  \end{displaymath}
proving \eqref{form87}. We have now shown separately that $U_{+}^{\Gamma}\mu_{\delta}(x) \to U_{+}^{\Gamma}\mu(x)$ and $U_{-}^{\Gamma}\mu_{\delta}(x) \to U_{-}^{\Gamma}\mu(x)$ for every $x \in \R$, and \eqref{form80} follows immediately. 

Finally, justifying \eqref{form88} is one more exercise in using dominated convergence. Using \eqref{form86}, we observe that
\begin{align*} \sup_{\delta \in (0,\tfrac{1}{6}]} U^{\Gamma}_{+}\mu_{\delta}(x) & \leq \int \sup_{\delta \in (0,\tfrac{1}{6}]} \left( \int \log_{+} \frac{1}{|\Gamma(x) - \Gamma(y)|} \varphi_{\delta}(y - z) \, dy \right) \, d\mu(z)\\
& \lesssim \int \max\{-\log |x - z|,0\} \, d\mu(z) + \|\mu\|, \end{align*} 
so $\sup_{\delta \in (0,1/6]} U^{\Gamma}_{+}\mu_{\delta} \in L^{\infty}(\R)$ by the hypothesis $\mu \in \mathcal{M}(\R)$. It is elementary to check that $\sup_{\delta \in (0,1/6]} U^{\Gamma}_{-}\mu_{\delta}$ is a locally bounded function with logarithmic growth, so all in all
\begin{displaymath} \sup_{\delta \in (0,\tfrac{1}{6}]} |U^{\Gamma}\mu_{\delta}(x)| \lesssim_{\mu} 1 + \max\{\log |x|,0\}. \end{displaymath}
The right hand side is absolutely integrable against $\Delta^{(1 - \beta)/2}\psi$ thanks to Lemma \ref{lemma3}, so \eqref{form80} combined with dominated convergence implies \eqref{form88}. As mentioned after \eqref{form88}, this completes the proof of the proposition. \end{proof}

\appendix

\section{The Fourier transform of the logarithmic kernel}

\begin{lemma}\label{appLemma1} Let $f \in \mathcal{S}(\R)$. Then,
\begin{equation}\label{form64} \int_{\R} \log |x| \int_{\R} e^{-2\pi ix\xi} |\xi|f(\xi) \, d\xi \, dx = -\tfrac{1}{2} \int_{\R} f(\xi) \, d\xi. \end{equation}
\end{lemma} 

\begin{proof} Start by noting that 
\begin{displaymath} \left| \int_{\R} e^{-2\pi ix\xi} |\xi|f(\xi) \, d\xi \right| = |(\Delta^{1/2}\widecheck{f})(x)| \lesssim_{f} \min\{1,|x|^{-3/2}\} \end{displaymath}
by Lemma \ref{lemma3}, so the integral on the left hand side of \eqref{form64} is absolutely convergent.

 Observe next  that
\begin{align*} \int_{\R} \log|x| \int_{\R} e^{-2\pi i x \xi}|\xi|f(\xi) \, d\xi \, dx & = \int_{\R} \log |x| \int_{\R} |\xi|f(\xi) \cos(2\pi x\xi) \, d\xi \, dx\\
&\qquad - i \int_{\R} \log |x| \int_{\R} |\xi|f(\xi) \sin(2\pi x\xi) \, d\xi \, dx. \end{align*} 
Here
\begin{displaymath} x \mapsto \int_{\R} |\xi| f(\xi) \sin (2\pi x \xi) \, d\xi \end{displaymath}
is odd and decays like $\lesssim |x|^{-3/2}$, and $x \mapsto \log |x|$ is even, so the second integral vanishes. Therefore, 
\begin{align} \int_{\R} \log|x| \int_{\R} e^{-2\pi i x \xi}|\xi|f(\xi) \, d\xi & = 2 \int_{0}^{\infty} \log x \int_{\R} |\xi|f(\xi) \cos(2\pi x \xi) \, d\xi \, dx \notag\\
&\label{form83} = 2 \lim_{R \to \infty} \int_{\R} |\xi|f(\xi) \int_{0}^{R} \log x \cos(2\pi x \xi) \, dx \, d\xi. \end{align}
For $\xi \neq 0$ fixed, we integrate by parts to evaluate the inner integral:
\begin{align*} \int_{0}^{R} \log x \cos (2\pi x \xi) \, dx & = \Big/_{x = 0}^{R} \frac{\log x \sin (2\pi x \xi)}{2\pi \xi} - \int_{0}^{R} \frac{\sin (2\pi x \xi)}{2\pi x \xi} \, dx\\
& = \frac{\log R \sin (2\pi R\xi)}{2\pi \xi} - \int_{0}^{R} \frac{\sin (2\pi x \xi)}{2\pi x\xi} \, dx. \end{align*}
Therefore, we find 
\begin{align*} 2\lim_{R \to \infty} \int_{\R} |\xi|f(\xi) \int_{0}^{R} \log x \cos(2\pi x \xi) \, dx \, d\xi & = \lim_{R \to \infty}  \log R \int_{\R} \frac{|\xi|f(\xi)}{\pi \xi} \sin(2\pi R\xi) \, d\xi\\
& \quad - \lim_{R \to \infty} \int \frac{|\xi|f(\xi)}{\pi \xi} \int_{0}^{R} \frac{\sin (2\pi x\xi)}{x} \, dx \, d\xi. \end{align*}
The first term equals zero, as we can see by splitting the $\xi$-integration into $(-\infty,0]$ and $[0,\infty)$, and integrating by parts. Regarding $[0,\infty)$, for example,
\begin{align*} \lim_{R \to \infty} \frac{\log R}{2\pi} \int_{0}^{\infty} f(\xi) \sin (2\pi R\xi) \, d\xi & = -\lim_{R \to \infty} \frac{\log R}{(2\pi)^{2}R} \Big/_{\xi = 0}^{\infty} f(\xi) \cos(2\pi R\xi)\\
& \quad + \lim_{R \to \infty} \frac{\log R}{(2\pi)^{2} R} \int_{0}^{\infty} f'(\xi)\cos(2\pi R\xi) \, d\xi = 0. \end{align*} 
Therefore, in fact
\begin{displaymath} \int_{\R} \log|x| \int_{\R} e^{-2\pi i x \xi}|\xi|f(\xi) \, d\xi = -\lim_{R \to \infty} \int \frac{|\xi|f(\xi)}{\pi \xi} \int_{0}^{R} \frac{\sin (2\pi x\xi)}{x} \, dx \, d\xi.\end{displaymath}
It is well-known (see \cite[Exercise 4.1.1]{MR3243734}) that
\begin{displaymath} \lim_{R \to \infty} \int_{0}^{R} \frac{\sin(2\pi x\xi)}{x} \, dx = \frac{\pi\sgn(\xi)}{2}, \qquad \xi \neq 0, \end{displaymath}
so by the dominated convergence theorem
\begin{displaymath} \lim_{R \to \infty} \int \frac{|\xi|f(\xi)}{\pi \xi} \int_{0}^{R} \frac{\sin (2\pi x\xi)}{x} \, dx \, d\xi = -\frac{1}{2} \int f(\xi) \, d\xi, \end{displaymath}
as desired. \end{proof} 

\section{Proof of Corollary \ref{thm:minimum principleExtension}}\label{appB}

The purpose of this section is to give more details on Corollary \ref{thm:minimum principleExtension}. The main tool is the following "intermediate" extension of \cite[Theorem 2.5]{reznikov2017minimum} (or Theorem \ref{thm:minimum principle}), where the global Ahlfors regularity assumption is kept intact, but the constant "$M$" has been replaced by a more general continuous lower bound:

\begin{thm}\label{thm:minimum principleStrengthening} Let $s \in (0,d]$, and let $\gamma\subset \R^d$ be a compact Ahlfors $s$-regular set. Let $\Omega \subset \R^{d}$ be an open set, and let $N \colon \gamma \cap \Omega \to \R$ be a continuous function. Let $\mu$ be a Radon measure on $\gamma$ satisfying 
	\begin{equation}\label{form164}
		U\mu(x)\ge N(x)\quad\text{for approximately all $x\in \gamma \cap \Omega$.}
	\end{equation}
Then $U\mu(x)\ge N(x)$ for all $x\in \gamma \cap \Omega$. \end{thm}

\begin{proof} The main work behind \cite[Theorem 2.5]{reznikov2017minimum} goes into \cite[Proposition 2.7]{reznikov2017minimum}. That result says that, under the Ahlfors $s$-regularity assumption, every point $x \in \gamma$ is a weak $s$-Lebesgue point for $U\mu$:
\begin{equation}\label{form163} U\mu(x) = \lim_{r \to 0} \frac{1}{\mathcal{H}^{s}(\gamma \cap B(x,r))} \int_{\gamma \cap B(x,r)} U\mu(y) \, d\mathcal{H}^{s}(y), \qquad x \in \gamma. \end{equation}
With this result in hand, the proof of \cite[Theorem 2.5]{reznikov2017minimum} is short, see \cite[p. 242]{reznikov2017minimum}. The key observation is that if a property $P$ holds approximately everywhere, then $P$ holds $\mathcal{H}^{s}$ everywhere (for the proof, see \cite[p. 242]{reznikov2017minimum}). In particular, our main hypothesis \eqref{form164} holds for $\mathcal{H}^{s}$ almost all $x \in \gamma \cap \Omega$. Therefore, using \eqref{form163},
\begin{align*} U\mu(x) & = \lim_{r \to 0} \frac{1}{\mathcal{H}^{s}(\gamma \cap B(x,r))} \int_{\gamma \cap B(x,r)} U\mu(y) \, d\mathcal{H}^{s}(y)\\
& \geq \lim_{r \to 0} \frac{1}{\mathcal{H}^{s}(\gamma \cap B(x,r))} \int_{\gamma \cap B(x,r)} N(y) \, d\mathcal{H}^{s}(y) = N(x), \qquad x \in \gamma \cap \Omega, \end{align*}
by the continuity of $N$, and the openness of $\Omega$. \end{proof}

We finally deduce Corollary \ref{thm:minimum principleExtension} from Theorem \ref{thm:minimum principleStrengthening}:

\begin{proof}[Proof of Corollary \ref{thm:minimum principleExtension}] Write $\bar{\gamma} := \gamma \cap \overline{\Omega}$, and decompose
\begin{displaymath} \mu := \mu|_{\bar{\gamma}} + \mu|_{\gamma \, \setminus \, \overline{\Omega}} =: \mu_{\mathrm{in}} + \mu_{\mathrm{out}}. \end{displaymath}
Now, by hypothesis,
\begin{displaymath} U\mu_{\mathrm{in}}(x) \geq M - U\mu_{\mathrm{out}}(x) =: N(x) \quad \text{for approximately all $x \in \gamma \cap \Omega$}. \end{displaymath}
Note that $N$ is continuous on $\gamma \cap \Omega$. Since $\mu_{\mathrm{in}}$ is supported on a compact Ahlfors $s$-regular set by hypothesis, Theorem \ref{thm:minimum principleStrengthening} implies $U\mu_{\mathrm{in}}(y) \geq N(x)$ for all $x \in \gamma \cap \Omega$. Thus $U\mu(x) = U\mu_{\mathrm{in}}(x) + U\mu_{\mathrm{out}}(x) \geq M$ for all $x \in \gamma \cap \Omega$. \end{proof}

\bibliographystyle{plain}
\bibliography{references}

\end{document}